%% file: SDGLEels_sf.tex
\documentclass[3p,times]{elsarticle}
\usepackage{epstopdf}
\usepackage{bbm}
\usepackage{bm}
\usepackage{array} 
\usepackage{color}
\usepackage{xcolor}

\usepackage{graphicx}
\usepackage{psfrag}

\usepackage{amssymb}
\usepackage{amsmath}
\usepackage{dsfont}
\usepackage{rotating}
\usepackage{pgf,tikz}
\usetikzlibrary{arrows}

\definecolor{ttzzqq}{rgb}{0.2,0.6,0}

\definecolor{qqttcc}{rgb}{0,0.2,0.8}

\definecolor{qqttzz}{rgb}{0,0.2,0.6}
\definecolor{ffqqqq}{rgb}{1,0,0}
\definecolor{qqwuqq}{rgb}{0,0.39,0}
\definecolor{zzttqq}{rgb}{0.6,0.2,0}
\definecolor{qqqqff}{rgb}{0,0,1}
\definecolor{ttttqq}{rgb}{0.2,0.2,0}

\definecolor{qqwwtt}{rgb}{0,0.4,0.2}
\definecolor{ubqqys}{rgb}{0.29,0,0.51}

\definecolor{wwttqq}{rgb}{0.4,0.2,0}
\definecolor{uuuuuu}{rgb}{0.27,0.27,0.27}
\definecolor{qqzzff}{rgb}{0,0.6,1}
\definecolor{xdxdff}{rgb}{0.49,0.49,1}
\definecolor{ccwwqq}{rgb}{0.8,0.4,0}

\definecolor{ttqqqq}{rgb}{0.2,0,0}

\definecolor{qqzzcc}{rgb}{0,0.6,0.8}

\newcommand{\diff}[2]{\frac{\partial {#1} }{\partial {#2} } }
\newcommand{\eref}[1]{$(\ref{#1})$}
\newcommand{\Ni}{ N_i}
\newcommand{\Nj}{ N_j}
\newcommand{\halb}{\frac{1}{2}} 
\newcommand{\TT}{\mbox{\boldmath$T$}}
\newcommand{\QQ}{\mbox{\boldmath$R$}}
\newcommand{\R}{\mathbb{R}}
\newcommand{\B}{\mathcal{B}}
\newcommand{\E}{\mathbf{E}}
\newcommand{\nv}{\vec{n}}
\newcommand{\p}{\mathbb{\wp}}
\newcommand{\st}{ {st}}
\newcommand{\xx}{ \mathbf{x}}
\newcommand{\bphi}{ \bm{\phi}}
\newcommand{\bpsi}{ \bm{\psi}}
\newcommand{\dx}{ d\xx}
\newcommand{\dxt}{ d\xx \, dt \,}

\newcommand{\D}{\bm{\mathcal{D}}}
\newcommand{\Q}{\bm{\mathcal{Q}}}

\newcommand{\Ss}{ \bm{\mathcal{S}}}

\newcommand{\M}{\bm{M}}
\newcommand{\nextud}{\vec{n}_{i,j}}
\newcommand{\nstd}{\vec{n}_{j}}

\renewcommand{\vec}{\mathbf} 

\newtheorem{proof}{Proof}
\newtheorem{theorem}{Theorem}

\usepackage{amssymb}

\journal{Journal of Computational Physics}

\begin{document}

\begin{frontmatter}



\title{Arbitrary high order accurate space-time discontinuous Galerkin finite element schemes on staggered unstructured meshes for linear elasticity}   

\author[1]{Maurizio Tavelli\fnref{label1}}
\author[1]{Michael Dumbser \corref{corr1} \fnref{label2}}
\address[1]{Department of Civil, Environmental and Mechanical Engineering, University of Trento, Via Mesiano 77, I-38123 Trento, Italy}

\fntext[label1]{\tt m.tavelli@unitn.it (M.~Tavelli)}
\fntext[label2]{\tt michael.dumbser@unitn.it (M.~Dumbser)}
\begin{abstract}
 In this paper we propose a new high order accurate space-time discontinuous Galerkin (DG) finite element scheme for the solution of the linear elastic wave equations in first order velocity-stress  formulation in two and three-space dimensions on \textit{staggered} unstructured triangular and tetrahedral meshes. The method reaches arbitrary high order of accuracy in both space and time via the use  of space-time basis and test functions. Within the staggered mesh formulation, we define the discrete velocity field in the control volumes of a primary mesh, while the discrete stress tensor is defined on a face-based staggered dual mesh. The space-time DG formulation leads to an \textit{implicit} scheme that requires the solution of a linear system for the unknown degrees of freedom at the new time level. 
The number of unknowns is reduced at the aid of the Schur complement, so that in the end only a linear system for the degrees of freedom of the velocity  field needs to be solved, rather than a system  that involves both stress and velocity. Thanks to the use of a spatially staggered mesh, the stencil of the final velocity system involves only the element and its direct neighbors and the linear system can be efficiently solved via matrix-free iterative methods. Despite the necessity to solve a linear system, the numerical scheme is still computationally efficient. The chosen discretization and the  linear  nature of the governing PDE system lead to an unconditionally stable scheme, which allows large time steps even for low quality meshes that contain so-called sliver elements.  
The fully discrete staggered space-time DG method is proven to be \textit{energy stable} for any order of accuracy, for any mesh and for any time step size. 
For the particular case of a simple Crank-Nicolson time discretization and homogeneous material, the final velocity system can be proven to be symmetric and positive definite and in this 
case the scheme is also exactly \textit{energy preserving}. The new scheme is applied to several test problems in two and three space dimensions, providing also a comparison with high order explicit 
ADER-DG schemes. 
\end{abstract}


\end{frontmatter}


\section{Introduction}
\label{sec.intro} 
Even nowadays the accurate simulation of elastic wave propagation in heterogeneous media involving complex geometries is a very challenging task. 
In the past several numerical methods have been developed in order to solve the linear elasticity equations. Some classical finite difference methods can be found in \cite{Madariaga1976, Virieux1984, Virieux1986} and further extensions, see e.g. \cite{Levander1988, Mora1989, Moczo2002, Igel1995, Tessmer1995, Moczo2015, Moczo2017}. Concerning the class of pseudo-spectral methods we refer  the reader to \cite{Tessmer1994, Igel1999}. The spectral finite element method, originally introduced by Patera in \cite{Patera1984}, was applied to linear elastic wave propagation 
in a well-known series of papers, see e.g. \cite{Priolo1994,Komatitsch1998,Seriani1998,Komatitsch1999,Komatitsch2002} and references therein. 

A major challenge in the numerical simulation of linear elastic waves is the ability of the numerical scheme to accurately propagate complex wave patterns over long distances and 
for very long times. Therefore, the use of high order schemes in both space and time is necessary. For a quantitative accuracy analysis of high order schemes applied to elastic wave 
propagation, see e.g. \cite{Hermann2008,Moczo2010a}. The analysis is based on the misfit criteria developed in \cite{Moczo2006,Moczo2009}. For an alternative study of high order
DG schemes applied to wave propagation problems, see \cite{RemacleWave}. 

Another challenge is the discretization of complex geometries including both, complex surface topography as well as complex sub-surface fault structures. In this case, the use of 
unstructured simplex meshes composed of triangles or tetrahedra seems to be beneficial concerning the problem of mesh generation in complex geometries. Concerning high order explicit 
discontinuous Galerkin (DG) finite element schemes for linear elastic wave propagation on unstructured general meshes the reader is referred to \cite{gij1, gij2, gij3, gij4, gij5}
and to \cite{GroteDG, Antonietti1, Antonietti2}. 
However, since the previous methods are \textit{explicit}, they are only stable under a CFL-type stability condition on the time step that depends on the mesh quality as well as the 
polynomial approximation 
degree used. In particular, unstructured simplex meshes for complex 3D geometries may contain so-called sliver elements, which are tiny elements with very bad aspect ratio and which look
like needles or thin plates. In the case of explicit time discretizations, such elements can be efficiently treated only at the aid of time-accurate local time stepping (LTS), see e.g.
\cite{gij5,Taub2009,GroteLTS1,GroteLTS2}. In this paper, we try to solve this problem in a different way using an efficient high order accurate \textit{implicit} time discretization. 

Our work is inspired by a new class of high order accurate semi-implicit discontinuous Galerkin finite element schemes on staggered meshes recently introduced in 
\cite{DumbserCasulli,2DSIUSW,2STINS,3DSIINS,Fambri2016,3DSICNS,AMRDGSI} for the numerical solution of the shallow water equations, the incompressible and the compressible Navier-Stokes equations. 
Being semi-implicit, the previous methods allow large time steps. Furthermore, the use of an edge-based staggered grid allows to connect the discrete divergence operator with 
the discrete gradient operator. This leads to some interesting properties of the final pressure system that needs to be solved, which becomes symmetric and positive 
definite. The use of staggered meshes is state of the art for many finite difference schemes used in computational fluid dynamics \cite{markerandcell,chorin1,chorin2,Bell1989,CasulliCheng1992,patankar,vanKan,BalsaraSpicer1999,Balsara2001b,Casulli2014,DumbserCasulli2016} as well as for seismic wave propagation \cite{Moczo2002,Moczo2010b,Puente3,Puente4}. 
However, at present staggered meshes are still almost unknown in the context of high order discontinuous Galerkin finite element methods for wave propagation. 
Apart from the above-mentioned references on semi-implicit staggered DG schemes 
\cite{DumbserCasulli,2DSIUSW,2STINS,3DSIINS,Fambri2016,3DSICNS,AMRDGSI}, the authors are only aware of \cite{StaggeredDG,StaggeredDG2,StaggeredDGCE1,StaggeredDGCE2,StaggeredDGCE3} and references 
therein concerning high order DG schemes for wave propagation using edge-based staggered grids. For central DG schemes, which use a vertex-based grid staggering, the reader is referred to 
\cite{CentralDG1,CentralDG2}. However, none of those references uses space-time discontinuous Galerkin finite elements, where the basis and test functions depend not only on space, but both
on space and time. The concept of space-time DG schemes was introduced by van der Vegt et al. for computational fluid dynamics in \cite{spacetimedg1,spacetimedg2,KlaijVanDerVegt,Rhebergen2012,Rhebergen2013} and has been subsequently analyzed e.g. in \cite{Balazsova1,Balazsova2}. The first application of space-time DG schemes to elastodynamics on collocated grids has been reported in \cite{Antonietti3,Antonietti4}, but to the best of our knowledge there exists no space-time DG scheme for the linear elastic wave equations on staggered grids so far. It is the aim of this paper to design and
analyze the properties of such methods. 

More precisely, in this paper we extend the idea of staggered semi-implicit space-time discontinuous Galerkin methods for the Navier-Stokes equations \cite{2STINS,3DSIINS,3DSICNS,Fambri2016} 
to linear elasticity. 
While the velocity field is discretized on the main grid, the stress tensor is defined  
on a face-based staggered \textit{dual} mesh. The governing PDE system is linear and all terms are taken implicitly. Inserting the discrete evolution equations for the stress tensor into the discrete momentum equation leads to one single linear system for the velocity field via the application of the Schur complement. Once the velocity field at the new time is known, one can readily update the 
stress tensor using an explicit formula. The good properties of the main system already observed in \cite{3DSIINS, 3DSICNS} are achieved also in this case. The resulting numerical scheme is shown to be  \textit{energy stable} for any polynomial degree in space and time. A remarkable particular case can be obtained by using arbitrary high order polynomials in space combined with a second order Crank-Nicolson time discretization. For this special case the method becomes exactly energy preserving and the main system becomes symmetric and positive definite. We also present a simple and efficient 
physics-based preconditioner that is useful in the presence of sliver elements. 

The rest of this paper is organized as follows: in Section \ref{sec.pde} we present the governing PDE system and in Section \ref{sec.grid} we introduce the staggered grid that is used in our approach, as well as the chosen basis functions. In Section \ref{sec.scheme} we present the numerical scheme and analyze its properties in Section \ref{sec_5}. In Section \ref{sec.results} we show numerical results for several test problems in two and three space dimensions. We compare all numerical results obtained with our new high order staggered space-time DG scheme with those obtained by a high order explicit ADER-DG scheme on unstructured meshes. The paper closes with some concluding remarks and an outlook to future work in Section \ref{sec.concl}.  

\section{Governing equations}
\label{sec.pde}
Based on the theory of linear elasticity, see e.g. \cite{BedfordDrumheller}, the governing partial differential equations for the wave propagation in a linear elastic medium without 
attenuation can be written in compact first order velocity-stress formulation based on the Hooke law and the momentum conservation law. They read 
\begin{eqnarray}
\diff{\bm{\sigma}}{t}- \E \cdot \nabla \vec{v} = \bm{S}_\sigma, \label{eq:1.1} \\
\diff{\rho\vec{v}}{t}-\nabla \cdot \bm{\sigma} = \rho \bm{S}_v,	\label{eq:1.2} 
\end{eqnarray}
where $\rho$ is the mass density, $\bm{\sigma}=\bm{\sigma}^\top$ is the symmetric stress tensor, $\vec{{v}}=(u,v,w)$ is the velocity field, $\bm{S}_v$ and $\bm{S}_\sigma$ 
are volume sources and $\E$ denotes the usual rank 4 stiffness tensor representing the linear material behavior according to the Hooke law $ \sigma_{ij} = E_{ijkl} \epsilon_{kl}$, 
where $\epsilon_{kl}=\epsilon_{lk}$ is the symmetric strain tensor. 
The connection between the strain \textit{rate} tensor and the velocity gradient is $\partial_t \epsilon_{ij} = \dot{\epsilon}_{ij} = \halb \left( \partial_j v_i + \partial_i v_j \right)$. 
It is well-known that the stiffness tensor $\E$ has the following so-called \textit{minor symmetries} $E_{ijkl} = E_{jikl} = E_{ijlk}$, due to the symmetries of the stress and the strain 
tensor, and the \textit{major symmetry} $E_{ijkl} = E_{klij}$, hence it can have at most 21 independent components, and not 81. From the minor symmetries of $\E$ follows that  
$E_{ijkl} \partial_t \epsilon_{kl} = \halb E_{ijkl} \partial_l v_k + \halb E_{ijlk} \partial_k v_l = E_{ijkl} \partial_l v_k = \E \cdot \nabla \vec{v}$. 
Throughout the paper we use the Einstein summation convention over repeated indices. The symmetric stress tensor $\bm{\sigma}$ is 
\begin{eqnarray}
	\bm{\sigma} = \left(
	\begin{array}{ccc}
		\sigma_{xx} & \sigma_{xy} & \sigma_{xz} \\
		\sigma_{xy} & \sigma_{yy} & \sigma_{yz} \\
		\sigma_{xz} & \sigma_{yz} & \sigma_{zz} \\ 
	\end{array} \right) = \bm{\sigma}^T.
\label{eq:2}
\end{eqnarray}
The normal stress components along the $x$, $y$ and $z$ directions are given by $\sigma_{xx}$, $\sigma_{yy}$ and $\sigma_{zz}$, while the shear stresses are represented by $\sigma_{xy}$, $\sigma_{xz}$ 
and $\sigma_{yz}$. Due to its symmetry the stress tensor $\bm{\sigma}$ can be written as a vector in terms of its six independent components as $\tilde{\bm{\sigma}}=(\sigma_{xx},\sigma_{yy},\sigma_{zz},\sigma_{yz},\sigma_{xz},\sigma_{xy})$, where we use the tilde symbol when we refer to the vector of the six independent components of the stress tensor $\bm{\sigma}$. The same notation is also used
for the 6 independent components of the strain tensor, i.e. $\tilde{\bm{\epsilon}} = ({\epsilon}_{xx}, {\epsilon}_{yy}, {\epsilon}_{zz}, {\epsilon}_{yz}, {\epsilon}_{xz}, {\epsilon}_{xy} )$, 
so that the stress-strain relationship can be also written as $\tilde{\bm{\sigma}} = \tilde{\E} \, \tilde{\bm{\epsilon}}$. In this paper we assume $\tilde{\E}$ to be \textit{invertible} so that 
the strain can be computed from the stress as $\tilde{\bm{\epsilon}} = \tilde{\E}^{-1} \, \tilde{\bm{\sigma}}$. From $\tilde{\E}^{-1}$ we define a tensorial object $\E^{-1}=E^{-1}_{ijkl}$ with the 
same symmetries as $E_{ijkl}$ and the property $ E^{-1}_{ijpq} E_{pqkl} = \delta_{ijkl}$. The object $\delta_{ijkl}$ has again the same symmetries as $\E$ and furthermore 
it satisfies $\delta_{ijkl} \sigma_{kl} = \sigma_{ij}$ and thus also $\delta_{ijkl} \sigma_{ij} = \sigma_{kl}$. The entries of $E^{-1}_{ijkl}$ are given by those of $\tilde{\E}^{-1}$ or are scaled by one half, 
and the object $\delta_{ijkl}$ contains only zeros, ones and $\halb$. Their construction is immediate once the inverse $\tilde{\E}^{-1}$ has been computed. 
For \textit{isotropic material}, equation \eref{eq:1.1} can be rewritten in terms of the two Lam\'e constants $\lambda$ and $\mu$ simply as  
\begin{equation}
\partial_t {\bm{\sigma}}  - \lambda \left( \nabla \cdot \vec{{v}} \right) \mathbf{I} - \mu \left( \nabla \vec{{v}} + \nabla \vec{{v}}^T \right) = \bm{S}_\sigma, \label{eq:1.1b} \\
\end{equation} 
with the identity matrix $\mathbf{I}$, or in terms of the vector $\tilde{\bm{\sigma}}$ and the independent components of the strain rate tensor as 
\begin{equation}
	\partial_t {\tilde{\bm{\sigma}}}  - \tilde{\E} \cdot \partial_t \tilde{\bm{\epsilon}}  = \bm{S}_{\tilde{\sigma}},  
\label{eq:3}
\end{equation}
with $\partial_t \tilde{\bm{\epsilon}} = \left( \partial_x u, \partial_y v, \partial_z w, \halb( \partial_z v + \partial_y w ), \halb ( \partial_z u + \partial_x w ), \halb( \partial_y u + \partial_x v) \right)$ and where for isotropic material 
\begin{equation}
\tilde{\E}=\left( 
\begin{array}{cccccc}
	\lambda+2\mu & \lambda & \lambda & 0 & 0 & 0 \\
	\lambda & \lambda+2\mu & \lambda & 0 & 0 & 0 \\ 
	\lambda & \lambda & \lambda+2\mu & 0 & 0 & 0 \\  
	 0 & 0 & 0 & 2 \mu & 0 & 0 \\
	 0 & 0 & 0 & 0 & 2 \mu & 0 \\
	 0 & 0 & 0 & 0 & 0 & 2 \mu 
\end{array}
\right),  \quad   
\tilde{\E}^{-1} = \frac{1}{2 \mu \alpha} 
\left( 
\begin{array}{cccccc}
	2 (\lambda + \mu) & -\lambda & -\lambda & 0 & 0 & 0 \\
	-\lambda & 2 (\lambda + \mu) & -\lambda & 0 & 0 & 0 \\ 
	-\lambda & -\lambda & 2 (\lambda + \mu) & 0 & 0 & 0 \\  
	 0 & 0 & 0 & \alpha & 0 & 0 \\
	 0 & 0 & 0 & 0 & \alpha & 0 \\
	 0 & 0 & 0 & 0 & 0 & \alpha   
\end{array}
\right).  
\label{eq:5}
\end{equation}
with $\alpha = 3 \lambda + 2 \mu $. For a homogeneous material we can assume $\E$ to be a constant in space and time. For non-homogeneous media 
we have $\E=\E(\vec{x})$, which,  however, is still assumed to be a constant in time. 

\section{Staggered unstructured grid and basis functions}
\label{sec.grid}
Throughout this paper we use the same unstructured spatially staggered mesh as the one used in \cite{SINS,2STINS,3DSIINS} for the two and three-dimensional case, respectively. 
In the following section we briefly summarize the grid construction and the main notation for the two dimensional triangular grid. After that, the primary and dual spatial elements 
are extended to the three dimensional case and also to the case of space-time control volumes.

\paragraph{Two space dimensions}
In the two-dimensional case the spatial computational domain $\Omega \subset \mathbb{R}^2$ is covered with a set of $\Ni$ non-overlapping triangular elements $\TT_i$ with $i=1 \ldots \Ni$. By denoting with $\Nj$ the total number of edges, the $j-$th edge will be called $\Gamma_j$. $\B(\Omega)$ denotes the set of indices $j$ corresponding to boundary edges.
The three edges of each triangle $\TT_i$ constitute the set $S_i$ defined by $S_i=\{j \in [1,\Nj] \,\, | \,\, \Gamma_j \mbox{ is an edge of }\TT_i \}$. For every $j\in [1\ldots \Nj]-\B(\Omega)$ there exist two triangles $i_1$ and $i_2$ that share $\Gamma_j$. We assign arbitrarily a left and a right triangle called respectively $\ell(j)$ and $r(j)$ for any $j\in [1\ldots \Nj]-\B(\Omega)$. The standard positive direction is assumed to be from left to right. $\nv_{j}$ denotes the unit normal vector defined on the edge $j$ and oriented with respect to the positive direction according to the previous definition. For every triangular element $i$ and edge $j \in S_i$, the index of the neighbor triangle of element $\TT_i$ that shares the edge $\Gamma_j$ is denoted by $\p(i,j)$.
\par For every $j\in [1, \Nj]-\B(\Omega)$ the quadrilateral dual element associated to $\Gamma_j$ is called $\QQ_j$ and it is defined, in general, by the two barycenter of $\TT_{\ell(j)}$ and $\TT_{r(j)}$ and the two nodes of $\Gamma_j$, see also \cite{Bermudez1998,Bermudez2014,Busto2018,USFORCE,2DSIUSW,StaggeredDG}. We denote by $\TT_{i,j}=\QQ_j \cap \TT_i$ the intersection element for every $i$ and $j \in S_i$. Figure $\ref{fig.1}$ summarizes the used notation, the primal triangular mesh and the dual quadrilateral grid.  
\begin{figure}
    \begin{center}
    \input{./tec_ugrid.tex}
    \caption{Example of a triangular mesh element with its three neighbors and the associated staggered edge-based dual control volumes, together with the notation
    used throughout the paper.}
    \label{fig.1}
		\end{center}
\end{figure}
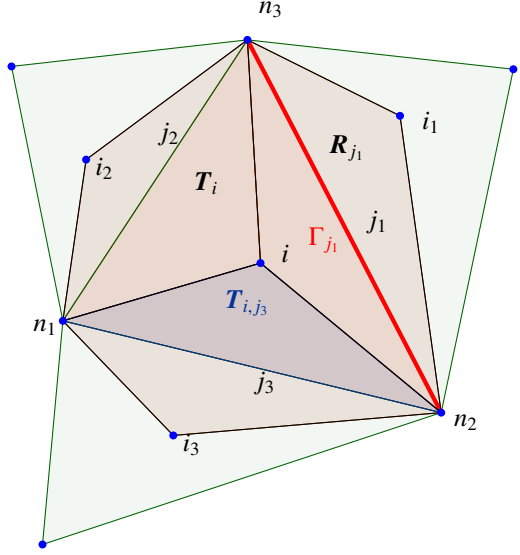
According to \cite{2STINS}, we will call the mesh of triangular elements $\{\TT_i \}_{i \in [1, \Ni]}$ the \textit{main grid} and the quadrilateral grid $\{\QQ_j \}_{j \in [1, N_d]}$ is termed the \textit{dual grid}. 


\paragraph{Three space dimensions}
The definitions given above are then readily extended to three space dimensions with the domain $\Omega \subset \mathbb{R}^3$. 
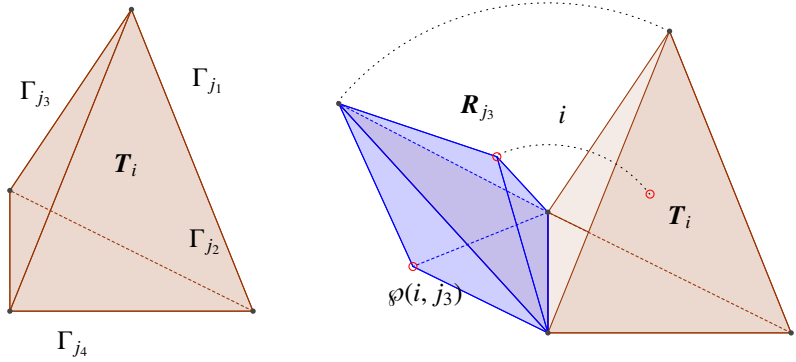
\begin{figure}
    \begin{center}
    \input{./tec_3d_omega_i.tex}
		\input{./tec_3d_omega_j.tex}
    \caption{An example of a tetrahedral element of the primary mesh with $S_i=\{j_1, j_2, j_3, j_4\}$ (left) a non-standard dual face-based hexahedral element associated to the face $j_3$ (right).}
    \label{fig.MESH_1}
		\end{center}
\end{figure}
An example of the resulting main and dual grid in three space dimensions is reported in Figure \ref{fig.MESH_1}. The main grid consists of tetrahedral simplex elements, and the face-based dual
elements contain the three vertices of the common triangular face of two tetrahedra (a left and a right one), and the two barycenters of the two tetrahedra that share the same face. 
In three space dimensions the dual grid therefore consists of non-standard five-point hexahedral elements. The same face-based staggered dual mesh has also been used in 
\cite{Bermudez2014,Bermudez2014,Busto2018,USFORCE,USFORCE2}.  


\paragraph{Space-time extension}
In the time direction we cover the time interval $[0,T]$ with a sequence of times $0=t^0<t^1<t^2 \ldots <t^N<t^{N+1}=T$. We denote the time step by $\Delta t^{n+1} = t^{n+1}-t^{n} $ and 
the corresponding time interval by $T^{n+1}=[t^{n}, t^{n+1}]$ for $n=0 \ldots N$. In order to ease notation, sometimes we will use the abbreviation $\Delta t= \Delta t^{n+1}$. 
The generic space-time element defined in the time interval $[t^n, t^{n+1}]$ is given by $\TT_i^\st = \TT_i \times T^{n+1}$ for the main grid, and $\QQ_j^\st=\QQ_j \times T^{n+1}$ for the dual grid.

\paragraph{Space-time basis functions}
\label{sec222}
According to \cite{2DSIUSW,2STINS,3DSIINS} we proceed as follows: in the two dimensional case, we first construct the polynomial basis up to a generic polynomial degree $p$ on some triangular and  quadrilateral reference elements. In particular, we take $T_{std}=\{(\xi,\eta) \in \R^{2} \,\, | \,\,  0 \leq \xi \leq 1, \,\,  0 \leq \eta \leq 1-\xi \}$ as the reference triangle. Using the standard nodal approach of conforming continuous finite elements, we obtain $N_\phi=\frac{(p+1)(p+2)}{2}$ basis functions $\{\phi_k \}_{k \in [1,N_\phi]}$ on $T_{std}$ and $N_{\psi}=(p+1)^2$ nodal basis 
functions on the unit square $R_{std}=[0,1]^2$ that can be obtained using the tensor product of one dimensional basis functions defined of the unit interval $[0,1]$. 
The connection between the reference coordinates $\boldsymbol{\xi}=(\xi,\eta)$ and the physical coordinates $\xx=(x,y)$ is obtained using either sub-parametric or iso-parametric maps, see e.g.  
\cite{2DSIUSW} for more details. 

Regarding the basis functions in three space dimensions, we use the unit tetrahedron $T_{std}=\{(\xi,\eta,\zeta) \in \R^{3} \,\, | \,\,  0 \leq \xi \leq 1, \,\,  0 \leq \eta \leq 1-\xi, \,\, 0 \leq \zeta \leq 1-\xi-\eta \}$ to construct the basis polynomials for the main grid. We use again the standard nodal basis functions of conforming finite elements based on the reference element $T_{std}$ and then using either a sub-parametric or an iso-parametric map to connect the reference space $\boldsymbol{\xi}=(\xi,\eta,\zeta)$ to the physical space $\xx=(x,y,z)$ and vice-versa.   
For the non-standard five-point hexahedral elements of the dual mesh, we define the polynomial basis directly in the physical space via the rescaled monomials of a Taylor series, as defined in 
\cite{3DSIINS}. We thus obtain $N_\phi=N_\psi=\frac{(p+1)(p+2)(p+3)}{6}$ basis functions per element for both, the main grid and the dual mesh. 

Finally, we construct the time basis functions on a reference interval $I_{std}=[0,1]$ for polynomials of degree $p_\gamma$ by taken the Lagrange interpolation polynomials passing through the Gauss-Legendre quadrature points for the unit interval $I_{std}$. In this case the resulting $N_\gamma=p_\gamma+1$ basis functions in time are called $\{\gamma_k\}_{k \in [1, N_\gamma]}$. In this manner, the 
nodal basis in time is an orthogonal basis. 
For every time  interval $[t^n, t^{n+1}]$, the map between the reference interval and the physical one is simply given by $t=t^n+\tau \Delta t^{n+1} \,\, \forall \tau \in [0,1]$. 
Using the tensor product we can finally construct the basis functions on the space-time elements $\TT_i^\st$ and $\QQ_j^\st$ as $\tilde{\phi}(\boldsymbol{\xi},\tau)=\phi(\boldsymbol{\xi}) \cdot \gamma(\tau)$ and $\tilde{\psi}(\boldsymbol{\xi},\tau)=\psi\boldsymbol{\xi}) \cdot \gamma(\tau)$. The total number of basis functions becomes $N_\phi^\st=N_\phi \cdot N_\gamma$ and $N_\psi^\st=N_\psi \cdot N_\gamma$. 

\section{Numerical scheme}
\label{sec.scheme} 
The discrete velocity field $\vec{v}_h$ is now defined on the main grid, while the discrete stress tensor $\bm{\sigma}_h$ is defined on the face-based staggered dual grid, namely $\vec{v}_i(\xx,t)=\vec{v}_h(\xx,t)|_{\TT_i^\st}$ and $\bm{\sigma}_j(\xx,t)=\bm{\sigma}_h(\xx,t)|_{\QQ_j^\st}$. For a heterogeneous material also the material parameters $\lambda$, $\mu$ and $\rho$ have to be discretized using piecewise high order polynomials. The discrete material density $\rho_h$ is defined on the main grid, while the discrete material tensor $\E_h$ is defined on the dual grid, namely 
$\rho_i(\xx)=\rho_h(\xx)|_{\TT_i^\st}$ and $\E_j(\xx)=\E_h(\xx)|_{\QQ_j^\st}$. The numerical solution of \eref{eq:1.1}-\eref{eq:1.2}, as well as the  discrete material parameters are represented inside the space-time control volumes of the main and the dual grid and for a time slice $T^{n+1}$ by piecewise space-time polynomials as follows: 
\begin{eqnarray}
	\vec{v}_i(\xx,t)      & = & \sum\limits_{l=1}^{N_\phi^\st} \tilde{\phi}_l^{(i)}(\xx,t) \hat{\vec{v}}_{l,i}^{n+1}=:\tilde{\bphi}^{(i)}(\xx,t)\hat{\mathbf{\vec{v}}}_i^{n+1}, \nonumber \\
	\rho_i(\xx,t)      & = & \sum\limits_{l=1}^{N_\phi^\st} \tilde{\phi}_l^{(i)}(\xx,t) \hat{\vec{v}}_{l,i}^{n+1}=:\tilde{\bphi}^{(i)}(\xx,t)\hat{\mathbf{\rho}}_i^{n+1}, \nonumber \\
		\bm{\sigma}_j(\xx,t)      & = & \sum\limits_{l=1}^{N_\psi^\st} \tilde{\psi}_l^{(j)}(\xx,t) \hat{\sigma}_{l,j}^{n+1}=:\tilde{\bpsi}^{(j)}(\xx,t)\hat{\bm{\sigma}}_j^{n+1}, \nonumber \\
			\E_j(\xx)       & = & \sum\limits_{l=1}^{N_\psi^\st} \tilde{\psi}_l^{(j)}(\xx) \hat{\E}_{l,j}=:\tilde{\bpsi}^{(j)}(\xx)\hat{\E}_j.  
\label{eq:NM1}
\end{eqnarray}
Note that the discrete velocity is allowed to jump at the element boundaries of the main grid, while the discrete stress tensor jumps only at the boundaries of the dual grid and is therefore 
\textit{continuous} across the boundaries of the main grid. This property is essential for our staggered DG method, since it completely avoids the necessity of Riemann solvers or numerical flux functions
at the element boundaries.

Multiplication of the momentum equation \eref{eq:1.2} by a test function $\tilde{\phi}_k$, for $k=1 \ldots N_\phi^\st$, and integration over a primary space-time control volume $\TT_\st$, leads to
\begin{eqnarray}
\int\limits_{\TT_i^\st}{\tilde{\phi}_k^{(i)} \diff{\rho\vec{v}}{t}\dxt}-\int\limits_{\TT_i^\st}{\tilde{\phi}_k^{(i)} \nabla \cdot \bm{\sigma}\dxt} 
=\int\limits_{\TT_i^\st}{\tilde{\phi}_k^{(i)} \rho \bm{S}_v \dxt}. 
\label{eq:NM6}
\end{eqnarray}
Using integration by parts Eqn. \eref{eq:NM6} yields 
\begin{eqnarray}
\int\limits_{\TT_i^\st}{\tilde{\phi}_k^{(i)} \diff{\rho\vec{v}}{t}\dxt} 
-\left( \, \int\limits_{\partial \TT_i^\st}{\tilde{\phi}_k^{(i)} \bm{\sigma} \cdot \vec{n}_{i} dS\, dt} -\int\limits_{\TT_i^\st}{\nabla\tilde{\phi}_k^{(i)} \cdot \bm{\sigma} \dxt}  \right)
=\int\limits_{\TT_i^\st}{\tilde{\phi}_k^{(i)}\rho \bm{S}_v \dxt},
\label{eq:NM7}
\end{eqnarray}
where $\vec{n}_{i}$ indicates the outward unit normal vector with respect to $\TT_i$. Multiplication of equation \eref{eq:1.1} by a test function $\tilde{\psi}_k$, for $k=1 \ldots N_\psi^\st$ and integration over a dual space-time control volume $\QQ_j^\st$ leads to
\begin{eqnarray}
	\int\limits_{\QQ_j^\st}{\tilde{\psi}_k^{(j)} \diff{\bm{\sigma}}{t} \dxt} -  {\int\limits_{\QQ_j^\st}{\tilde{\psi}_k^{(j)} \E \cdot  \nabla \vec{v} \dxt}} 
	=\int\limits_{\QQ_j^\st}{\tilde{\psi}_k^{(j)} \bm{S}_\sigma \dxt}. 
\label{eq:NM8}
\end{eqnarray}
Due to the discontinuous discretization of our numerical quantities we have to split equations \eref{eq:NM7} and \eref{eq:NM8} as follows:
\begin{eqnarray}
\int\limits_{\TT_i^\st}{\tilde{\phi}_k^{(i)} \diff{(\rho \vec{v})_i}{t}\dxt} 
-\sum\limits_{j \in S_i}\left(\int\limits_{\Gamma_j^\st}{\tilde{\phi}_k^{(i)} \bm{\sigma}_j \cdot \vec{n}_{i,j} dS\, dt} -\int\limits_{\TT_{i,j}^\st}{\nabla\tilde{\phi}_k^{(i)} \cdot \bm{\sigma}_j \dxt}  \right)
=\int\limits_{\TT_i^\st}{\tilde{\phi}_k^{(i)} \rho \bm{S}_v \, \dxt},
\label{eq:NM10}
\end{eqnarray}
\begin{equation}
	\int\limits_{\QQ_j^\st}{\tilde{\psi}_k^{(j)} \diff{\bm{\sigma}_j}{t} \dxt} - 
	  \!\! \int\limits_{\TT_{\ell(j),j}^\st} \!\!\! {\tilde{\psi}_k^{(j)} \E_j \cdot \nabla \vec{v}_{\ell(j)} \dxt}  
	- \!\! \int\limits_{\TT_{r(j),j}^\st}    \!\!\! {\tilde{\psi}_k^{(j)} \E_j \cdot \nabla \vec{v}_{r(j)} \dxt}  
											  - \int\limits_{\Gamma_{j}^\st}{\tilde{\psi}_k^{(j)} \E_j \cdot (\vec{v}_{r(j)}-\vec{v}_{\ell(j)}) \otimes \vec{n}_j dS\, dt} 
	 = \! \int\limits_{\QQ_j^\st}{\tilde{\psi}_k^{(j)} \bm{S}_\sigma \, \dxt}. 
\label{eq:NM11}
\end{equation}
With $\vec{n}_{i,j}$ we denote the outward pointing unit normal vector of element $\TT_i^\st$ on its face $\Gamma_j^\st$. 
Note that a jump contribution is necessary in Eq. \eref{eq:NM11}, since the gradient of the velocity needs to be integrated in the sense of distributions. 
However, since the stress tensor $\bm{\sigma}_j$ is defined on the staggered dual mesh and therefore is continuous across primary element interfaces, no Riemann solver 
(numerical flux function) is needed in our approach, which is a particular feature of the chosen staggered mesh. 
Following the ideas used in \cite{3DSIINS,2STINS} we integrate the terms including the time derivatives in \eref{eq:NM10}-\eref{eq:NM11} by parts in time and hence obtain  
\begin{eqnarray}
\int\limits_{\TT_i^\st}{\tilde{\phi}_k^{(i)} \diff{(\rho\vec{{v}})_i}{t}\dxt}&=&	  \int\limits_{\TT_i}{\tilde{\phi}_k^{(i)}(\xx,t^{n+1,-}) \rho\vec{{v}}_i(\xx,t^{n+1,-}) \dx}  
																																					- \int\limits_{\TT_i}{\tilde{\phi}_k^{(i)}(\xx,t^{n,+}) \rho\vec{{v}}_i(\xx,t^{n,-}) \dx} 
																																					- \int\limits_{\TT_i^\st}{\diff{\tilde{\phi}_k^{(i)}}{t} (\rho\vec{{v}})_i\dxt}, \nonumber \\
\int\limits_{\QQ_j^\st}{\tilde{\psi}_k^{(j)} \diff{\bm{\sigma}_j}{t}\dxt}&=&	 \int\limits_{\QQ_j}{\tilde{\psi}_k^{(j)}(\xx,t^{n+1,-}) \bm{\sigma}_j(\xx,t^{n+1,-}) \dx} 
                                                                             - \int\limits_{\QQ_j}{\tilde{\psi}_k^{(j)}(\xx,t^{n,+}) \bm{\sigma}_j(\xx,t^{n,-}) \dx}  
																																			- \int\limits_{\QQ_j^\st}{\diff{\tilde{\psi}_k^{(j)}}{t} \sigma_j\dxt},																															
\label{eq:NM13}
\end{eqnarray}
where $t^{n,-}$ indicates the boundary-extrapolated value from a lower time slice and thus corresponds to upwinding in time, due to the causality principle. 
Using the definitions \eref{eq:NM1} and rewriting the contribution of the time derivative as specified in \eref{eq:NM13} we obtain from the previous equations 
\begin{eqnarray}
\left( \, \int\limits_{\TT_i}{\tilde{\phi}_k^{(i)}(\xx,t^{n+1,-}) \tilde{\phi}_m^{(i)}(\xx,t^{n+1,-})\dx} \right. 
\left.- \int\limits_{\TT_i^\st}{\diff{\tilde{\phi}_k^{(i)}}{t} \tilde{\phi}_m^{(i)} \dxt} \right) \, \hat{(\rho\vec{v})}_{m,i}^{n+1} 
-\int\limits_{\TT_i}{\tilde{\phi}_k^{(i)}(\xx,t^{n,+}) \tilde{\phi}_m^{(i)}(\xx,t^{n,-})\dx} \, \hat{(\rho\vec{v})}_{m,i}^{n} \nonumber \\
-\sum\limits_{j \in S_i}\left(\int\limits_{\Gamma_j^\st}{\tilde{\phi}_k^{(i)} \tilde{\psi}_m^{(j)} \vec{n}_{i,j} dS\, dt} -\int\limits_{\TT_{i,j}^\st}{\nabla\tilde{\phi}_k^{(i)}  \tilde{\psi}_m^{(j)} \dxt}  \right) \cdot \hat{\bm{\sigma}}_{m,j}^{n+1}  
=\int\limits_{\TT_i^\st} {\tilde{\phi}_k^{(i)} \rho \bm{S}_v } \, \dxt \nonumber \\
\label{eq:NM14}
\end{eqnarray}
and
\begin{eqnarray}
&& \left( \int\limits_{\QQ_j}{\tilde{\psi}_k^{(i)}(\xx,t^{n+1,-}) \tilde{\psi}_m^{(i)}(\xx,t^{n+1,-})\dx} \right. 
\left.- \int\limits_{\QQ_j^\st}{\diff{\tilde{\psi}_k^{(i)}}{t} \tilde{\psi}_m^{(i)} \dxt} \right) \hat{\bm{\sigma}}_{m,j}^{n+1} 
-\int\limits_{\QQ_i}{\tilde{\psi}_k^{(i)}(\xx,t^{n,+}) \tilde{\psi}_m^{(i)}(\xx,t^{n,-})\dx} \, \hat{\bm{\sigma}}_{m,j}^{n}  \nonumber \\
&&	- \hat{\E}_{q,j}\cdot \left( \int\limits_{\TT_{\ell(j),j}^\st}{\tilde{\psi}_k^{(j)} \nabla \tilde{\phi}_m^{(\ell(j))} \tilde{\psi}_q^{(j)} \dxt}  
	 -   \int\limits_{\Gamma_{j}^\st}{\tilde{\psi}_k^{(j)} \tilde{\phi}_m^{(\ell(j))} \tilde{\psi}_q^{(j)}\vec{n}_j dS dt} \right) \hat{\vec{v}}_{m,\ell(j)}^{n+1}  \nonumber \\
&&	- \hat{\E}_{q,j}\cdot \left( \int\limits_{\TT_{r(j),j}^\st}{\tilde{\psi}_k^{(j)}\nabla \tilde{\phi}_m^{(r(j))} \tilde{\psi}_q^{(j)}  \dxt}  
	  +  \int\limits_{\Gamma_{j}^\st}{\tilde{\psi}_k^{(j)} \tilde{\phi}_m^{(r(j))} \tilde{\psi}_q^{(j)}\vec{n}_j dS dt} \right) \hat{\vec{v}}_{m,r(j)}^{n+1} 
	=\int\limits_{\QQ_j^\st}{\tilde{\psi}_k^{(j)} \bm{S}_{\sigma} \dxt}.   
\label{eq:NM15}
\end{eqnarray}
where the quantity $\hat{(\rho\vec{v})}_{m,i}^{n+1}$ is simply defined using a pointwise evaluation, namely $\hat{(\rho\vec{v})}_{m,i}^{n+1} = \hat{\rho}_{m,i}^{n+1} \hat{\vec{v}}_{m,i}^{n+1}$ (here, no
summation over repeated indices is used). 
In order to ease the notation we introduce the following matrix and tensor definitions, according to \cite{3DSIINS,2STINS}: 
\begin{equation}
	  \M_j^+       = \int\limits_{\QQ_j}{\tilde{\psi}_k^{(j)}(\xx,t^{n+1,-})\tilde{\psi}_m^{(j)}(\xx,t^{n+1,-})  \, d\xx }, \qquad 
		\bar{\M}_i^+ = \int\limits_{\TT_i}{\tilde{\phi}_k^{(i)}(\xx,t^{n+1,-})\tilde{\phi}_m^{(i)}(\xx,t^{n+1,-})  \, d\xx }, \label{eq:MD_2} 
\end{equation} 		
\begin{equation}
    \M_j^-       = \int\limits_{\QQ_j}{\tilde{\psi}_k^{(j)}(\xx,t^{n,+})\tilde{\psi}_m^{(j)}(\xx,t^{n,-})  \, d\xx }, \qquad
		\bar{\M}_i^- = \int\limits_{\TT_i}{\tilde{\phi}_k^{(i)}(\xx,t^{n,+})\tilde{\phi}_m^{(i)}(\xx,t^{n,-})  \, d\xx }, \label{eq:MD_2_1} \\
\end{equation} 		
\begin{equation}
    \M_j^\circ       = \int\limits_{\QQ_j^\st}{\diff{\tilde{\psi}_k^{(j)}}{t} \tilde{\psi}_m^{(j)} \, d\xx  dt}, \qquad
    \bar{\M}_i^\circ = \int\limits_{\TT_i^\st}{\diff{\tilde{\phi}_k^{(i)}}{t} \tilde{\phi}_m^{(i)} \, d\xx  dt},  \label{eq:MD_2_2} \\
\end{equation} 		
\begin{equation}
		\M_j = \M_j^+ -\M_j^\circ, \qquad \bar{\M}_i = \bar{\M}_i^+ - \bar{\M}_i^\circ, \label{eq:MD_2_3}  
\end{equation} 		

\begin{equation}
	     \Ss_j = \int\limits_{\QQ_j^\st}{\tilde{\psi}_k^{(j)} \mathbf{S}_{\sigma} \, d\xx  dt}, \qquad 
	(\rho \Ss)_i = \int\limits_{\TT_i^\st}{\tilde{\phi}_k^{(i)} \rho \mathbf{S}_v \, d\xx  dt}
\label{eq:MD_5_2}
\end{equation}

\begin{equation}
	\D_{i,j}=\int\limits_{\Gamma_j^\st}{\tilde{\phi}_k^{(i)}\tilde{\psi}_m^{(j)}\nextud dS dt}-\int\limits_{\TT_{i,j}^\st}{\nabla \tilde{\phi}_k^{(i)}\tilde{\psi}_m^{(j)}\, d\xx  dt},
\label{eq:MD_3}
\end{equation}

\begin{equation}
	\Q_{i,j} = \int\limits_{\TT_{i,j}^\st}{\tilde{\psi}_k^{(j)} \nabla \tilde{\phi}_{m}^{(i)} \tilde{\psi}_q^{(j)}  \, d\xx  dt} -\int\limits_{\Gamma_j^\st}{\tilde{\psi}_k^{(j)} \tilde{\phi}_{m}^{(i)} \tilde{\psi}_q^{(j)}s_{i,j} \nstd dS dt},
\label{eq:MD_6}
\end{equation}
where $s_{i,j}$ is a sign function defined by
\begin{equation}
	s_{i,j}=\frac{r(j)-2i+\ell(j)}{r(j)-\ell(j)}.
\label{eq:SD_1}
\end{equation}
Equations \eref{eq:NM14} and \eref{eq:NM15} are then rewritten in a compact form as
\begin{equation}
\bar{\M}_i \hat{(\rho \vec{{v}})}_i^{n+1}=\bar{\M}_i^- \hat{(\rho \vec{{v}})}_i^{n}+\sum\limits_{j \in S_i}\D_{i,j} \cdot \hat{\bm{\sigma}}^{n+1}_j+(\rho\Ss)_i,
\label{eq:NM25}
\end{equation}
\begin{equation}
\M_j \hat{\bm{\sigma}}_j^{n+1} = \M_j^- \hat{\bm{\sigma}}_j^{n}
+ \hat{\E}_j \cdot \Q_{\ell(j),j} \hat{\vec{{v}}}^{n+1}_{\ell(j)}+ \hat{\E}_j \cdot  \Q_{r(j),j} \hat{\vec{{v}}}^{n+1}_{r(j)} 
+\Ss_j.
\label{eq:NM26}
\end{equation}
Formal substitution of the discrete PDE for the stress tensor \eqref{eq:NM26} into the discrete momentum equation \eref{eq:NM25}, i.e. application of the Schur complement, yields a 
linear system that corresponds to a \textit{discrete second order wave equation} for all degrees of freedom of the velocity vector field $\vec{v}_h$ and which reads 
\begin{equation}
\bar{\M}_i \, \hat{\rho}_i \hat{ \vec{{v}}}_i^{n+1} - \sum\limits_{j \in S_i}\D_{i,j} \cdot \M_j^{-1} 
\left( \hat{\E}_j \cdot \Q_{\ell(j),j}  \hat{\vec{{v}}}^{n+1}_{\ell(j)} + \hat{\E}_j \cdot \Q_{r(j),j}  \hat{\vec{{v}}}^{n+1}_{r(j)} \right)  
= \bar{\M}_i^- \hat{\rho}_i \hat{ \vec{{v}}}_i^{n}+\sum\limits_{j \in S_i}\D_{i,j} \cdot \M_j^{-1} \left( \M_j^- \hat{\bm{\sigma}}^{n}_j + \Ss_j \right) + (\rho\Ss)_i. 
\label{eq:velocity.sys}
\end{equation}
The shape of this system can be rather complex if explicitly expressed in terms of all components of $\mathbf{v}_h$ and $\mathbf{E}_h$. 
For anisotropic materials, the system has exactly the same formal structure as given in \eqref{eq:velocity.sys}, just with a more complex tensor $\hat{\E}_j $ compared to simple isotropic material. 
In any case, the system involves only the velocity  field of the direct neighbors of each element and thus becomes a $4$-point block system in two space dimensions and a $5$-point block system in 
three space dimensions. For the  particular case of $p_\gamma=0$ (piecewise constant polynomials in time, i.e. $\M_j^{\circ}=\bar{\M}_i^{\circ}=0$, $\M_j^+ = \M_j^- = \M_j$, $\bar{\M}_i^+ = \bar{\M}_i^- = \bar{\M}_i$), second order of accuracy in time can be easily achieved with the Crank-Nicolson scheme. In this 
setting, equations \eqref{eq:NM25} and \eqref{eq:NM26} read 
\begin{equation}
\bar{\M}_i \hat{(\rho \vec{{v}})}_i^{n+1}=\bar{\M}_i \hat{(\rho \vec{{v}})}_i^{n}+\sum\limits_{j \in S_i}\D_{i,j} \cdot \hat{\bm{\sigma}}^{n+\halb}_j+(\rho\Ss)_i,
\label{eq:NM25.cn}
\end{equation}
\begin{equation}
\M_j \hat{\bm{\sigma}}_j^{n+1} = \M_j \hat{\bm{\sigma}}_j^{n}
+ \hat{\E}_j \cdot \Q_{\ell(j),j} \hat{\vec{{v}}}^{n+\halb}_{\ell(j)} + \hat{\E}_j \cdot \Q_{r(j),j}  \hat{\vec{{v}}}^{n+\halb}_{r(j)} 
+\Ss_j,
\label{eq:NM26.cn}
\end{equation}
with 
$\hat{\bm{\sigma}}^{n+\halb}_j = \halb \left( \hat{\bm{\sigma}}^{n}_j  + \hat{\bm{\sigma}}^{n+1}_j \right) $ and 
$\hat{\vec{{v}}}^{n+\halb}_i= \halb \left( \hat{\vec{{v}}}^{n}_i + \hat{\vec{{v}}}^{n+1}_i \right)$. 
In this case the final velocity system reads 
\begin{eqnarray}
\bar{\M}_i \, \hat{\rho}_i \hat{ \vec{{v}}}_i^{n+1} &-& \frac{1}{4} \sum\limits_{j \in S_i}\D_{i,j}  \cdot \M_j^{-1} 
\left( \hat{\E}_j \cdot \Q_{i,j}  \hat{\vec{{v}}}^{n+1}_{i} + \hat{\E}_j \cdot \Q_{\p(i,j),j}  \hat{\vec{{v}}}^{n+1}_{\p(i,j)} \right)  
=   \nonumber \\ 
\bar{\M}_i \hat{\rho}_i \hat{ \vec{{v}}}_i^{n} 
+ \sum\limits_{j \in S_i}\D_{i,j} \cdot \left( \hat{\bm{\sigma}}^{n}_j + \halb \M_j^{-1}  \Ss_j \right) 
&+& \frac{1}{4} \sum\limits_{j \in S_i}\D_{i,j} \cdot \M_j^{-1} 
\left( \hat{\E}_j \cdot \Q_{i,j} \hat{\vec{{v}}}^{n}_{i} + \hat{\E}_j \cdot \Q_{\p(i,j),j}  \hat{\vec{{v}}}^{n}_{\p(i,j)} \right) 
 + (\rho\Ss)_i.   
\label{eq:velocity.sys.cn}
\end{eqnarray}
It can be shown to be symmetric and positive definite for homogeneous materials. The proof of those properties is reported in Section \ref{sec_5} for the  homogeneous case. Thanks to those properties we  are able, for this special choice, to use a very fast linear solver such as the conjugate gradient (CG) method. For $p_\gamma>0$ the system is not symmetric anymore and since the time derivatives appear in both  equations the symmetrization strategy adopted in \cite{Fambri2016} for the incompressible Navier-Stokes equations is not possible any more. In any case we can still solve the system using a matrix-free GMRES algorithm \cite{GMRES} in order to obtain the degrees of freedom $\hat{ \vec{{v}}}_i^{n+1}$ 
of the velocity field at the new time slice. Once the new velocity field is known, we can then readily update the stress tensor at the aid of \eqref{eq:NM26} for $p_\gamma > 0$ or via
\eqref{eq:NM26.cn} for $p_\gamma=0$. 
This closes the description of the numerical method, which is analyzed in the subsequent section.

\section{Properties of the staggered space-time DG schemes for linear elasticity}
\label{sec_5}
In this section we report some details about the main matrix for the velocity system that needs to be solved in each time step, as well as some theoretical results about the energy 
stability of the numerical method. 

\subsection{Symmetry and positive definiteness for the special case of a Crank-Nicolson scheme in time} 
For homogeneous material ($\rho=const.$, $\mathbf{E}=const. $) and for $p_\gamma=0$ combined with the Crank-Nicolson scheme in time, the linear system \eqref{eq:velocity.sys.cn} reduces to  
\begin{equation}
\rho \bar{\M}_i \, \hat{ \vec{{v}}}_i^{n+1} - \frac{1}{4} \sum\limits_{j \in S_i}\D_{i,j} \cdot \M_j^{-1} 
\E \cdot \left( \tilde{\Q}_{i,j} \cdot \hat{\vec{{v}}}^{n+1}_{i} + \tilde{\Q}_{\p(i,j),j}  \cdot \hat{\vec{{v}}}^{n+1}_{\p(i,j)} \right) = \mathbf{b}_i^n, 
\label{eqn.A1} 
\end{equation} 
with the known right hand side $\mathbf{b}_i^n$ and the matrix 
\begin{equation}
	\tilde{\Q}_{i,j} = \left( \tilde{Q}_{i,j} \right)_l^{\kappa \mu} = \int\limits_{\TT_{i,j}^\st}{\tilde{\psi}_{\kappa}^{(j)} \partial_l \tilde{\phi}_{\mu}^{(i)}  \, d\xx  dt} -\int\limits_{\Gamma_j^\st}{\tilde{\psi}_\kappa^{(j)} \tilde{\phi}_{\mu}^{(i)} s_{i,j} (n_j)_l dS dt}.
	\label{eqn.qtilde} 
\end{equation} 
Note that the rank 3 tensor ${\Q}_{i,j}$ can be simplified to $\tilde{\Q}_{i,j}$ in the case of constant material properties. In this section, we use Greek upper indices for the basis and test functions in the objects $\tilde{\Q}_{i,j}$ and $\D_{i,j}$, and Latin lower indices for spatial vectors and tensors. The indices $i$ and $j$ are reserved for the numbers of the element and the face.  
\begin{theorem}
In the homogeneous isotropic case and for $p_\gamma=0$, the matrix of system \eref{eqn.A1} is symmetric. 
\end{theorem}
\begin{proof}
Since the material is assumed to be homogeneous, $\rho$ and $\E$ are constant in space and time. Due to the symmetry of the stress tensor $\sigma_{ij}=\sigma_{ji}$ and the strain tensor 
$\epsilon_{kl}=\epsilon_{lk}$, we also have $E_{ijkl} = E_{jikl} = E_{jilk}$, which are the so-called \emph{minor symmetries} of $\E$. The so-called \emph{major symmetries} of $\E$  
imply also that $\E_{ijkl}=\E_{klij}$. All these symmetries of $\E$ are summarized in the shorthand notation $\E = \E^\top$. Furthermore, from the definitions \eqref{eqn.qtilde} and 
\eqref{eq:MD_3} it is obvious to see that $\tilde{\Q}_{i,j} = - \D_{i,j}^\top$, see also \cite{3DSIINS}. 
From \eqref{eq:MD_2}-\eqref{eq:MD_2_3} one obtains that $ \bar{\M}_i =  \bar{\M}_i^\top$ for $p_\gamma=0$. The diagonal block in \eqref{eqn.A1} then reads 
\begin{equation}
 \mathbf{D}_i = \rho \bar{\M}_i  + \frac{1}{4} \sum\limits_{j \in S_i} \D_{i,j}  \M_j^{-1} \E \cdot \D_{i,j}^\top,   
\end{equation} 
or, more conveniently in index notation (Greek upper indices refer to basis and test functions, Latin lower indices to spatial vectors and tensors) 
\begin{equation} 
\mathbf{D}_i = (D_i)_{kl}^{\mu \nu} = \rho \left( \bar{M}_i \right)^{ \mu \nu } \delta_{kl} + \frac{1}{4} \sum\limits_{j \in S_i} \left( D_{i,j} \right)^{\mu \kappa}_p \left( \M_j^{-1} \right)^{\kappa \alpha} E_{kplm} \left( D_{i,j} \right)^{\nu \alpha}_m. 
\end{equation} 
and it is easy to see that its transpose verifies 
\begin{equation}
 \mathbf{D}_i^\top = \rho \bar{\M}_i^\top  + \frac{1}{4} \sum\limits_{j \in S_i} \D_{i,j}  \M_j^{-\top} \E^\top \cdot \D_{i,j}^\top = \mathbf{D}_i,   
\end{equation} 
or, more conveniently in index notation 
\begin{eqnarray}
\mathbf{D}_i^\top = (D_i)_{lk}^{\nu \mu} &=& \rho \left( \bar{M}_i \right)^{ \nu \mu } \delta_{lk} + \frac{1}{4} \sum\limits_{j \in S_i} \left( D_{i,j} \right)^{\nu \kappa}_p \left( \M_j^{-1} \right)^{\kappa \alpha} E_{lpkm} \left( D_{i,j} \right)^{\mu \alpha}_m = \nonumber \\  
&=& \rho \left( \bar{M}_i \right)^{ \mu \nu} \delta_{kl} + \frac{1}{4} \sum\limits_{j \in S_i}  \left( D_{i,j} \right)^{\mu \alpha}_m \left( \M_j^{-1} \right)^{\alpha \kappa} E_{kmlp} \left( D_{i,j} \right)^{\nu \kappa}_p  = (D_i)_{kl}^{\mu \nu} = \mathbf{D}_i, 
\end{eqnarray} 
where we have used the major symmetry of $E_{ijkl}$, the symmetries of the mass matrix and of the Kronecker delta $\delta_{kl}$ and the simple renaming of contracted indices. 

The off-diagonal blocks involving the neighbor elements $\p(i,j)$ of element $i$ read 
\begin{equation}
\mathbf{N}_{i,\p(i,j)} = - \frac{1}{4} \D_{i,j}  \M_j^{-1} \E \cdot \tilde{\Q}_{\p(i,j),j}=\frac{1}{4} \D_{i,j}  \M_j^{-1} \E \cdot \D_{\p(i,j),j}^\top.  
\end{equation} 
We write now the previous contribution in terms of edges $j \in [1,N_j]$ so that $\mathbf{N}_{\ell(j),r(j)}$ and $\mathbf{N}_{r(j),\ell(j)}$ are the off-diagonal blocks involving the contribution of $r(j)$ to $\ell(j)$ and vice-versa. So we have to show that $\mathbf{N}_{\ell(j),r(j)}=\mathbf{N}_{r(j),\ell(j)}^\top$, but
\begin{equation}
	\mathbf{N}_{r(j),\ell(j)}^\top= \frac{1}{4} \left(\D_{r(j),j}  \M_j^{-1} \E \cdot \D_{\ell(j),j}^\top \right)^\top = \frac{1}{4} \D_{\ell(j),j} \M_j^{-\top} \E^\top \D_{r(j),j}^\top=\mathbf{N}_{\ell(j),r(j)},
\end{equation}
or, using again the index notation,
\begin{eqnarray}
	\mathbf{N}_{r(j),\ell(j)}^\top=\left( N_{r(j),\ell(j)} \right)_{lk}^{\nu \mu}
	&=&  \frac{1}{4} \left( D_{r(j),j} \right)^{\nu \kappa}_p \left( \M_j^{-1} \right)^{\kappa \alpha} E_{lpkm} \left( D_{\ell(j),j} \right)^{\mu \alpha}_m = \nonumber \\
	&=& \frac{1}{4} \left( D_{\ell(j),j} \right)^{\mu \alpha}_m \left( \M_j^{-1} \right)^{\alpha \kappa} E_{kmlp} \left( D_{r(j),j} \right)^{\nu \kappa}_p  = \left(N_{\ell(j),r(j)}\right)_{kl}^{\mu \nu} = \mathbf{N}_{\ell(j),r(j)}
\end{eqnarray}
from the symmetries of $\E$ and $\M_j$.    
\end{proof}
\begin{theorem}
In the homogeneous case and $p_\gamma=0$, the matrix of system \eref{eqn.A1}  is positive definite.
\end{theorem}
\begin{proof}
We can follow the same reasoning as in \cite{SINS}, since $\bar{\M}_i = \bar{\M}_i^\top > 0$ and $\E=\E^\top > 0$. With these properties and from the results of \cite{SINS} we obtain  that the system  matrix of \eqref{eqn.A1} without the term $\rho \M_i$ is at least positive semi-definite. If we add the contribution of the positive definite mass matrix $\rho \bar{\M}_i > 0$, then the resulting 
system matrix in \eqref{eqn.A1} is positive definite. 
\end{proof}
Numerical evidence shows that also the non-homogeneous case seems to have the same properties, but unfortunately a rigorous mathematical proof is still missing for the general non-homogeneous case. 

\subsection{Stability analysis}
In this section we prove some stability results for the proposed scheme in the energy norm. 
In particular we will demonstrate that the semi-discrete scheme is \textit{energy preserving} and that the fully discrete staggered space-time DG scheme is \textit{energy stable}. 
A particular case is given by $p_\gamma=0$ combined with the Crank-Nicolson time discretization, for which the fully discrete scheme is \textit{exactly energy preserving}. 
\begin{theorem}
\label{thm_P1}
For homogeneous material with $\rho>0$, $\E=\E^\top > 0$ and in the absence of volume source terms the semi-discrete form of the proposed staggered DG scheme is energy preserving.  
\end{theorem}
\begin{proof}
Since $\E=\E^\top > 0$ one also has $\E^{-1} = \E^{-\top} > 0$. The semi-discrete form of the scheme with no volume source terms is given by 
\begin{eqnarray}
\int\limits_{\TT_i}{{\phi}^{(i)} \diff{(\rho\vec{v})_i}{t}\dx}= 
\sum\limits_{j \in S_i}\left(\int\limits_{\Gamma_j}{{\phi}^{(i)} \bm{\sigma}_j \cdot \vec{n}_{i,j} dS} -\int\limits_{\TT_{i,j}}{\nabla {\phi}^{(i)} \bm{\sigma}_j \dx}  \right),
\label{eq:PNM10}
\end{eqnarray}
\begin{eqnarray}
	 \int\limits_{\QQ_j}{ {\psi}^{(j)} \diff{\bm{\sigma}_j}{t} \dx}= 
													 \int\limits_{\TT_{\ell(j),j}}{ {\psi}^{(j)} \E_j \cdot \nabla \vec{v}_{\ell(j)} \dx} 
												 + \int\limits_{\TT_{r(j),j}}{    {\psi}^{(j)} \E_j \cdot \nabla \vec{v}_{r(j)}    \dx} 
												 + \int\limits_{\Gamma_{j}}{ {\psi}^{(j)} \E_j \cdot \left( \vec{v}_{r(j)}-\vec{v}_{\ell(j)} \right) \otimes \vec{n}_j dS }. 
\label{eq:PNM11}
\end{eqnarray}
Since the material is assumed to be homogeneous, we can take $ {\psi}^{(j)}=\E_j^{-1} \cdot \bm{\sigma}_j$ and $ {\phi}^{(i)}= \vec{v}_i$ as test functions, 
sum up all contributions (we use the index contraction $\bm{\sigma} : \bm{B} = \sigma_{ij} B_{ij}$ and the identity 
$\E^{-1} \cdot \bm{\sigma} : \E \cdot \mathbf{B} = E^{-1}_{ijmn} \sigma_{mn} E_{ijkl} B_{kl} = E^{-1}_{mnij} E_{ijkl} \sigma_{mn}  B_{kl} = 
\delta_{mnkl} \sigma_{mn} B_{kl} = \sigma_{kl} B_{kl} = \bm{\sigma} : \bm{B} $) and thus obtain the two \emph{scalar} relations  
\begin{equation}
\int\limits_{\TT_i}{ \vec{v}_i \cdot \diff{(\rho\vec{v})_i}{t}\dx}=
\sum\limits_{j \in S_i}\left(\int\limits_{\Gamma_j}{ \vec{v}_i \cdot \left( \bm{\sigma}_j \cdot \vec{n}_{i,j} \right) dS } - 
\int\limits_{\TT_{i,j}}{  \nabla\vec{v}_i : \bm{\sigma}_j \dx}  \right),
\label{eq:PNM102}
\end{equation}
\begin{eqnarray}
	\int\limits_{\QQ_j}{ \left( \E_j^{-1} \cdot \bm{\sigma}_j \right) : \diff{\bm{\sigma}_j}{t} \dx}=
												  \int\limits_{\TT_{\ell(j),j}}{\bm{\sigma}_j :  \nabla \vec{v}_{\ell(j)}  \dx} 
												+ \int\limits_{\TT_{r(j),j}}{\bm{\sigma}_j :     \nabla \vec{v}_{r(j)}     \dx} 
												+ \int\limits_{\Gamma_{j}}{\bm{\sigma}_j :  \left( \vec{v}_{r(j)}-\vec{v}_{\ell(j)} \right) \otimes \vec{n}_j  dS }. 
\label{eq:PNM112}
\end{eqnarray}
Summing over the entire domain yields 
\begin{equation}
\sum_{i=1}^{N_i}\int\limits_{\TT_i}{ \vec{v}_i \cdot \diff{(\rho\vec{v})_i}{t}\dx}= 
\sum_{i=1}^{N_i}\sum\limits_{j \in S_i}\left(\int\limits_{\Gamma_j}{ \vec{v}_i \cdot \left( \bm{\sigma}_j \cdot \vec{n}_{i,j} \right) dS } -\int\limits_{\TT_{i,j}}{ \nabla\vec{v}_i : \bm{\sigma}_j \dx}  \right),
\label{eq:PNM103}
\end{equation}
\begin{equation}
	\sum_{j=1}^{N_j}\int\limits_{\QQ_j}{ \left( \E_j^{-1} \cdot \bm{\sigma}_j \right)  : \diff{\bm{\sigma}_j}{t} \dx}=
													\sum_{j=1}^{N_j}\left( \int\limits_{\TT_{\ell(j),j}}{\bm{\sigma}_j : \nabla \vec{v}_{\ell(j)} \dx} \right.
												+\int\limits_{\TT_{r(j),j}}{\bm{\sigma}_j : \nabla \vec{v}_{r(j)} \dx}
												\left.+\int\limits_{\Gamma_{j}}{\bm{\sigma}_j : (\vec{v}_{r(j)}-\vec{v}_{\ell(j)}) \otimes \vec{n}_j dS }\right). 
\label{eq:PNM113}
\end{equation}
With $\bm{\sigma}=\bm{\sigma}^\top$, $\E=\E^\top$ and therefore $\E^{-1}=\E^{-\top}$ we can rewrite the time derivative terms as 
\begin{equation}
		\sum_{i=1}^{N_i}\int\limits_{\TT_i}{\vec{v}_i \cdot \diff{(\rho\vec{v})_i}{t}\dx}=  \frac{1}{2}\int\limits_{\Omega}{ \diff{\rho_h \vec{v}_h^2}{t}\dx},
\label{eq:SP1}
\end{equation}
\begin{equation}
	\sum_{j=1}^{N_j}\int\limits_{\QQ_j}{\left( \E_j^{-1} \cdot \bm{\sigma}_j \right) : \diff{\bm{\sigma}_j}{t} \dx}=\frac{1}{2} \int\limits_{\Omega}{\frac{\partial}{\partial t} {\left( \bm{\sigma}_h : \E_h^{-1} \cdot \bm{\sigma}_h \right)}  \dx},
\label{eq:SP2}
\end{equation}
 and since $\boldsymbol{\sigma}_j$ is continuous across  ${\Gamma_{j}}$ the right hand side of \eref{eq:PNM103} can be written in terms of the faces ${\Gamma_{j}}$ as
\begin{eqnarray}
	\sum_{i=1}^{N_i}\sum\limits_{j \in S_i}\left(\int\limits_{\Gamma_j}{\vec{v}_i \cdot \left( \bm{\sigma}_j \cdot \vec{n}_{i,j} \right) dS } -\int\limits_{\TT_{i,j}}{ \nabla \vec{v}_i : \bm{\sigma}_j \dx}   \right)=
	\sum_{j=1}^{N_j}\left( \int\limits_{\Gamma_j}{ \bm{\sigma}_j : \left( \vec{v}_{\ell(j)}-\vec{v}_{r(j)} \right) \otimes \vec{n}_{j}  dS }\right. \nonumber \\
	\left.-\int\limits_{\TT_{r(j),j}}{ \nabla\vec{v}_{r(j)} : \bm{\sigma}_j \dxt}-\int\limits_{\TT_{\ell(j),j} }{ \nabla\vec{v}_{\ell(j)} : \bm{\sigma}_j \dx} \right). 
\label{eq:SP3}
\end{eqnarray}
Summing Eqs. \eref{eq:PNM103}-\eref{eq:PNM113} and making use of Eqs. \eref{eq:SP1} and \eref{eq:SP2} and since the right hand sides of \eqref{eq:PNM103} and \eqref{eq:PNM113}  
add up to zero due to  \eref{eq:SP3}, one finally obtains 
\begin{eqnarray}
	\frac{1}{2} \int\limits_{\Omega}{\frac{\partial}{\partial t} {\left( \bm{\sigma}_h : \E_h^{-1} \cdot \bm{\sigma}_h  + \rho_h \vec{v}_h^2 \right)}  \dx}  = 0. 
\label{eq:SPF}
\end{eqnarray}
This means that the total energy, which is the sum of the kinetic energy and the mechanical energy, is conserved for the semi-discrete scheme. 
\end{proof}
We show now similar results for the fully discrete forms. The first result can be seen as a simple extension of the previous theorem using the ideas presented in \cite{DumbserFacchini,2STINS}. 
\begin{theorem}
	For homogeneous material with $\rho > 0$, $\E = \E^\top>0$ and in the absence of volume source terms, the staggered space-time DG scheme \eqref{eq:NM10} and \eqref{eq:NM11} with \eqref{eq:NM13} is  energy stable for $p_{\gamma} \geq 0$ for arbitrary meshes and for arbitrary time step size $\Delta t$. 
\end{theorem}
\begin{proof}
The fully-discrete staggered space-time DG method \eqref{eq:NM10} and \eqref{eq:NM11} with \eqref{eq:NM13} in the absence of volume source terms reads 
\begin{eqnarray}
\int\limits_{\TT_i}{\rho_i  \tilde{\phi}^{(i)}(\xx,t^{n+1,-})  \vec{v}_i(\xx,t^{n+1,-}) \dx} - \int\limits_{\TT_i}{\rho_i  \tilde{\phi}^{(i)}(\xx,t^{n,+})  \vec{v}_i(\xx,t^{n,-}) \dx} -  
\int\limits_{\TT_i^{st}}{ \rho_i \diff{\tilde{\phi}^{(i)}}{t}  \vec{v}_i \dxt} = \nonumber \\
\sum\limits_{j \in S_i}\left(\int\limits_{\Gamma_j^{st}}{\tilde{\phi}^{(i)} \bm{\sigma}_j \cdot \vec{n}_{i,j} dS \, dt} -\int\limits_{\TT_{i,j}^{st}}{\nabla \tilde{\phi}^{(i)} \bm{\sigma}_j \dxt}  \right),
\label{eq:PNM10st}
\end{eqnarray}
\begin{eqnarray}
\int\limits_{\QQ_j}{ \tilde{\psi}^{(i)}(\xx,t^{n+1,-})  \bm{\sigma}_j(\xx,t^{n+1,-}) \dx} - \int\limits_{\QQ_j}{  \tilde{\psi}^{(i)}(\xx,t^{n,+}) \bm{\sigma}_j(\xx,t^{n,-}) \dx} -  
\int\limits_{\QQ_j^{st}}{   \diff{\tilde{\psi}^{(i)}}{t}  \bm{\sigma}_j \dxt} = \nonumber \\
													\int\limits_{\TT_{\ell(j),j}^{st}}{ \tilde{\psi}^{(j)} \E_j \cdot \nabla \vec{v}_{\ell(j)} \dxt} 
												+\int\limits_{\TT_{r(j),j}^{st}}{ \tilde{\psi}^{(j)} \E_j \cdot \nabla \vec{v}_{r(j)} \dxt}
												+\int\limits_{\Gamma_{j}^{st}}{ \tilde{\psi}^{(j)} \E_j \cdot (\vec{v}_{r(j)}-\vec{v}_{\ell(j)}) \otimes \vec{n}_j dS \, dt }. 
\label{eq:PNM11st}
\end{eqnarray}
Taking $ \tilde{\psi}^{(j)}=\E_j^{-1} \cdot \bm{\sigma}_j$ and $ \tilde{\phi}^{(i)}= \vec{v}_i$ as test functions, summing up all contributions and proceeding in the same manner as in the proof of the  previous theorem, we arrive at the following intermediate scalar expression (also here the right hand side terms add again up to zero, for the same reason as before): 
\begin{eqnarray}  
\int\limits_{\Omega}{\rho_h    \vec{v}_h(\xx,t^{n+1,-}) \cdot \vec{v}_h(\xx,t^{n+1,-}) \dx} - \int\limits_{\Omega}{\rho_h  \vec{v}_h(\xx,t^{n,+}) \cdot  \vec{v}_h(\xx,t^{n,-}) \dx} -  
\halb \int \limits_{t^n}^{t^{n+1}} \int\limits_{\Omega}{  \diff{\rho_h \vec{v}_h^2 }{t} \dxt} + \nonumber \\
\int\limits_{\Omega}{ \bm{\sigma}_h(\xx,t^{n+1,-}) :  \E_h^{-1} \cdot \bm{\sigma}_h(\xx,t^{n+1,-})      \dx} - \int\limits_{\Omega}{  \bm{\sigma}_h(\xx,t^{n,-}) : \E_h^{-1} \cdot \bm{\sigma}_h(\xx,t^{n,+})    \dx} -  
\halb \int \limits_{t^n}^{t^{n+1}} \int\limits_{\Omega}{ \diff{}{t} \left( \bm{\sigma}_h : \E_h^{-1} \cdot \bm{\sigma}_h \right) \dxt} = 0. \nonumber 
\label{eq:stsum1}
\end{eqnarray}  
The terms containing the time derivatives can be integrated by parts in time and thus one obtains:  
\begin{eqnarray}  
  \halb \! \int\limits_{\Omega}{ \! \rho_h    \vec{v}^2_h(\xx,t^{n+1,-}) \dx} 
 - \int\limits_{\Omega}{ \! \rho_h  \vec{v}_h(\xx,t^{n,+}) \cdot  \vec{v}_h(\xx,t^{n,-}) \dx}   
+ \halb \! \int\limits_{\Omega}{ \! \rho_h    \vec{v}^2_h(\xx,t^{n,+}) \dx}  + \nonumber \\  
\halb \! \int\limits_{\Omega}{ \!\! \bm{\sigma}_h(\xx,t^{n+1,-}) :  \E_h^{-1} \cdot \bm{\sigma}_h(\xx,t^{n+1,-})      \dx} 
- \int\limits_{\Omega}{  \!\! \bm{\sigma}_h(\xx,t^{n,-}) : \E_h^{-1} \cdot \bm{\sigma}_h(\xx,t^{n,+})    \dx} 
+ \halb \! \int\limits_{\Omega}{ \!\! \bm{\sigma}_h(\xx,t^{n,+}) :  \E_h^{-1} \cdot \bm{\sigma}_h(\xx,t^{n,+})      \dx}   = 0. \nonumber 
\label{eq:stsum2}
\end{eqnarray}  
Adding and immediately subtracting again $\halb \! \int\limits_{\Omega}{ \! \rho_h    \vec{v}^2_h(\xx,t^{n,-}) \dx}$ and $\halb \! \int\limits_{\Omega}{ \!\! \bm{\sigma}_h(\xx,t^{n,-}) :  \E_h^{-1} \cdot \bm{\sigma}_h(\xx,t^{n,-})      \dx}$   yields 
\begin{eqnarray}  
  \halb \! \int\limits_{\Omega}{ \! \left( \rho_h    \vec{v}^2_h(\xx,t^{n+1,-}) + \bm{\sigma}_h(\xx,t^{n+1,-}) :  \E_h^{-1} \cdot \bm{\sigma}_h(\xx,t^{n+1,-}) \right) \dx} -
  \halb \! \int\limits_{\Omega}{ \! \left( \rho_h    \vec{v}^2_h(\xx,t^{n,-}) + \bm{\sigma}_h(\xx,t^{n,-}) :  \E_h^{-1} \cdot \bm{\sigma}_h(\xx,t^{n,-}) \right) \dx} \nonumber \\ 
+ \halb \! \int\limits_{\Omega}{ \! \rho_h    \vec{v}^2_h(\xx,t^{n,-}) \dx} 	
- \int\limits_{\Omega}{ \! \rho_h  \vec{v}_h(\xx,t^{n,+}) \cdot  \vec{v}_h(\xx,t^{n,-}) \dx} 	
+ \halb \! \int\limits_{\Omega}{ \! \rho_h    \vec{v}^2_h(\xx,t^{n,+}) \dx}  + \nonumber \\  
+ \halb \! \int\limits_{\Omega}{ \!\! \bm{\sigma}_h(\xx,t^{n,-}) :  \E_h^{-1} \cdot \bm{\sigma}_h(\xx,t^{n,-})      \dx}
- \int\limits_{\Omega}{  \!\! \bm{\sigma}_h(\xx,t^{n,-}) : \E_h^{-1} \cdot \bm{\sigma}_h(\xx,t^{n,+})    \dx} 
+ \halb \! \int\limits_{\Omega}{ \!\! \bm{\sigma}_h(\xx,t^{n,+}) :  \E_h^{-1} \cdot \bm{\sigma}_h(\xx,t^{n,+})      \dx}   = 0. \nonumber 
\label{eq:stsum3}
\end{eqnarray}  
The quadratic forms in the expressions above can be easily recognized, hence 
\begin{eqnarray}  
  \halb \! \int\limits_{\Omega}{ \! \left( \rho_h    \vec{v}^2_h(\xx,t^{n+1,-}) + \bm{\sigma}_h(\xx,t^{n+1,-}) :  \E_h^{-1} \cdot \bm{\sigma}_h(\xx,t^{n+1,-}) \right) \dx} -
  \halb \! \int\limits_{\Omega}{ \! \left( \rho_h    \vec{v}^2_h(\xx,t^{n,-}) + \bm{\sigma}_h(\xx,t^{n,-}) :  \E_h^{-1} \cdot \bm{\sigma}_h(\xx,t^{n,-}) \right) \dx} \nonumber \\ 
+ \halb \! \int\limits_{\Omega}{ \! \rho_h    \left( \vec{v}_h(\xx,t^{n,+}) - \vec{v}_h(\xx,t^{n,-}) \right)^2 \dx}  
+ \halb \! \int\limits_{\Omega}{ \!\! \left(  \bm{\sigma}_h(\xx,t^{n,+}) - \bm{\sigma}_h(\xx,t^{n,-}) \right) :  \E_h^{-1} \cdot \left(  \bm{\sigma}_h(\xx,t^{n,+}) - \bm{\sigma}_h(\xx,t^{n,-})  \right)    \dx} = 0. \nonumber \\ 
\label{eq:stsum4}
\end{eqnarray}  
Since $\rho_h > 0$ and $\E_h > 0$ and thus the jump terms at time $t^n$ are non-negative, 
\begin{equation}
\halb \! \int\limits_{\Omega}{ \! \rho_h    \left( \vec{v}_h(\xx,t^{n,+}) - \vec{v}_h(\xx,t^{n,-}) \right)^2 \dx}  
+ \halb \! \int\limits_{\Omega}{ \!\! \left(  \bm{\sigma}_h(\xx,t^{n,+}) - \bm{\sigma}_h(\xx,t^{n,-}) \right) :  \E_h^{-1} \cdot \left(  \bm{\sigma}_h(\xx,t^{n,+}) - \bm{\sigma}_h(\xx,t^{n,-})  \right) }  
\geq 0,  
\label{eqn:inequality} 
\end{equation} 
we finally obtain from \eqref{eq:stsum4} and \eqref{eqn:inequality} the sought result which relates the total energy at the new time level with the total energy at the old time level as 
\begin{equation}  
  \halb \! \int\limits_{\Omega}{ \! \left( \rho_h    \vec{v}^2_h(\xx,t^{n+1,-}) + \bm{\sigma}_h(\xx,t^{n+1,-}) :  \E_h^{-1} \cdot \bm{\sigma}_h(\xx,t^{n+1,-}) \right) \dx} 
	\leq 
  \halb \! \int\limits_{\Omega}{ \! \left( \rho_h    \vec{v}^2_h(\xx,t^{n,-}) + \bm{\sigma}_h(\xx,t^{n,-}) :  \E_h^{-1} \cdot \bm{\sigma}_h(\xx,t^{n,-}) \right) \dx},  
\label{eq:stsum5}
\end{equation}   
from which we can conclude that our new staggered space-time DG scheme for the linear elasticity equations is \emph{energy stable} for arbitrary polynomial approximation degree, 
general meshes and arbitrary time step size $\Delta t$. 
\end{proof}
The previous theorem shows that the method is energy stable and that the rate of energy loss is proportional to the jump in the discrete solution at the interface between two time slices. 
This rises the almost natural question on what happens if we employ a second order time discretization using the classical Crank-Nicolson scheme. The following theorem give us an interesting result: 
\begin{theorem}
For homogeneous material with $\rho>0$, $\E=\E^\top > 0$ and in the absence of volume source terms the fully-discrete staggered DG scheme with $p_\gamma=0$ and Crank-Nicolson time discretization 
is exactly energy preserving. 
\end{theorem}
\begin{proof}
Starting from the semi-discrete form \eref{eq:PNM10} and \eref{eq:PNM11}, inserting the standard Crank-Nicolson time discretization  
and using as test functions $\psi^{(j)}=\E_j^{-1} \cdot \bm{\sigma}_j^{n+\frac{1}{2}}$ and $\phi^{(i)}=\vec{v}_j^{n+\frac{1}{2}}$, one obtains 
\begin{eqnarray}
\int\limits_{\TT_i}{ \rho_i \vec{v}_i^{n+\frac{1}{2}} \cdot \frac{\vec{v}_i^{n+1}-\vec{v}_i^{n}}{\Delta t}\dx}= 
\sum\limits_{j \in S_i}\left(\int\limits_{\Gamma_j}{ \vec{v}_i^{n+\frac{1}{2}} \cdot \bm{\sigma}_j^{n+\frac{1}{2}} \cdot \vec{n}_{i,j} dS} -\int\limits_{\TT_{i,j}}{\nabla\vec{v}_i^{n+\frac{1}{2}} \cdot \bm{\sigma}_j^{n+\frac{1}{2}} \dx}  \right), \nonumber  
\label{eq:P3NM102}
\end{eqnarray}
\begin{eqnarray}
	\int\limits_{\QQ_j}{ \E_j^{-1} \cdot \bm{\sigma}_j^{n+\frac{1}{2}} : \frac{\bm{\sigma}_j^{n+1}-\bm{\sigma}_j^{n}}{\Delta t} \dx}=
													\int\limits_{\TT_{\ell(j),j}}{\bm{\sigma}_j^{n+\frac{1}{2}} : \nabla \vec{v}_{\ell(j)}^{n+\frac{1}{2}} \dx} 
												+\int\limits_{\TT_{r(j),j}}{\bm{\sigma}_j^{n+\frac{1}{2}} : \nabla \vec{v}_{r(j)}^{n+\frac{1}{2}} \dx} 
												+\int\limits_{\Gamma_{j}}{\bm{\sigma}_j^{n+\frac{1}{2}} : (\vec{v}_{r(j)}^{n+\frac{1}{2}}-\vec{v}_{\ell(j)}^{n+\frac{1}{2}}) \otimes \vec{n}_j dS}.  \nonumber \\ 
\label{eq:P3NM112}
\end{eqnarray} 
The right hand sides add again up to zero from the proof of Theorem \ref{thm_P1}, while for the discrete time derivatives we get from the definition of $\bm{\sigma}_j^{n+\frac{1}{2}}=\frac{1}{2} \left( \bm{\sigma}_j^{n}+\bm{\sigma}_j^{n+1}\right)$ and $\vec{v}_i^{n+\frac{1}{2}}=\frac{1}{2} \left( \vec{v}_i^{n}+\vec{v}_i^{n+1}\right)$ that 
\begin{eqnarray}
\int\limits_{\Omega}{ \rho_h \vec{v}_h^{n+\frac{1}{2}} \frac{\vec{v}_h^{n+1}-\vec{v}_h^{n}}{\Delta t}\dx}
=\frac{1}{\Delta t}\frac{1}{2}  \int\limits_{\Omega}{ \rho_h \left( \left(\vec{v}_h^{n+1}\right)^2-\left(\vec{v}_h^{n}\right)^2 \right) \dx},
\label{eq:SP31}
\end{eqnarray}
and a similar result for $\bm{\sigma}$. Using the same reasoning of Theorem \eref{thm_P1} we finally obtain
\begin{eqnarray}
	\frac{1}{2} \int\limits_{\Omega}{ \left( \rho_h \left(\vec{v}_h^{n+1}\right)^2 + \bm{\sigma}_h^{n+1} : \E_j^{-1} \cdot \bm{\sigma}_h^{n+1}  \right) \dx}  =  
	\frac{1}{2} \int\limits_{\Omega}{ \left( \rho_h \left(\vec{v}_h^{n  }\right)^2 + \bm{\sigma}_h^{n  } : \E_j^{-1} \cdot \bm{\sigma}_h^{n  }  \right) \dx}.    
\label{eq:SP3F}
\end{eqnarray}
and so the staggered DG scheme with the simple Crank-Nicolson time discretization is \emph{exactly energy preserving}.
\end{proof}

\section{Numerical tests}
\label{sec.results}
All test problems in this section assume isotropic material. For the definition of the initial conditions, we also make use of the state vector 
$\mathbf{U} = \left( \sigma_{xx}, \sigma_{yy}, \sigma_{xy}, u, v \right)$  in 2D and 
$\mathbf{U} = \left( \sigma_{xx}, \sigma_{yy}, \sigma_{zz}, \sigma_{xy}, \sigma_{yz}, \sigma_{xz}, u, v, w \right)$ in 3D. 
\subsection{Scattering of a plane wave on a circular cavity}
\label{sec_2dpw}
In this test case we consider a simple $p$-wave traveling in the $x$-direction and hitting a circular cavity. The computational domain is $\Omega=[-2.5,2.5]^2 - C_{0.25}$, where $C_r$ indicates the circle of radius $r$. The initial condition is  
\begin{eqnarray}
	\mathbf{U}(\xx,0) = 0.1\cdot(-2,0,4,2,0)\sin(2 \pi x),
\label{eq:NT1.1}
\end{eqnarray}
and the boundary conditions are set to be periodic on the external boundary and free surface boundary ($\bm{\sigma} \cdot \vec{n}=0$) on the circular cavity. The material parameters are homogeneous
and are chosen as $\rho=1$, $\lambda=2$ and $\mu=1$. The computational domain is 
discretized using $\Ni=5644$ triangles of characteristic mesh size {$h=0.11$}. We use a polynomial approximation degree of $p=5$ in space and $p_\gamma=1$ in time. 
The time step size is chosen as {$\Delta t = 0.01$}. 
We compare our new staggered space-time DG scheme with a well established explicit high order ADER-DG method that is the basis of the \texttt{SeisSol} code published in 
\cite{gij1,gij2,gij3,gij4,gij5,SeisSol1,SeisSol2} and its generalization under the framework of $P_NP_M$ schemes achieved in \cite{Dumbser2008}. 
For the reference solution, we use $N=M=2$ and a very fine mesh of $\Ni=563280$ triangular elements. In both cases we run the simulation up to $t_{end}=1.0$. 
A comparison of the resulting stress component $\sigma_{xx}$, colored with $\sigma_{yy}$ is shown in Figure \ref{fig.NT1.2}. 
Figure \ref{fig.NT1.3} shows the time series of all variables in $\vec{x}_1=(0.5,0.5)$ and $\vec{x}_2=(1.0,0.0)$. A very good agreement can be observed in all cases. 
Furthermore, we emphasize that the use of high order isoparametric elements is important for properly representing the curvilinear geometry of this test case. 

\begin{figure*}%
\centering 
\begin{tabular}{lr} 
\includegraphics[width=0.45\columnwidth]{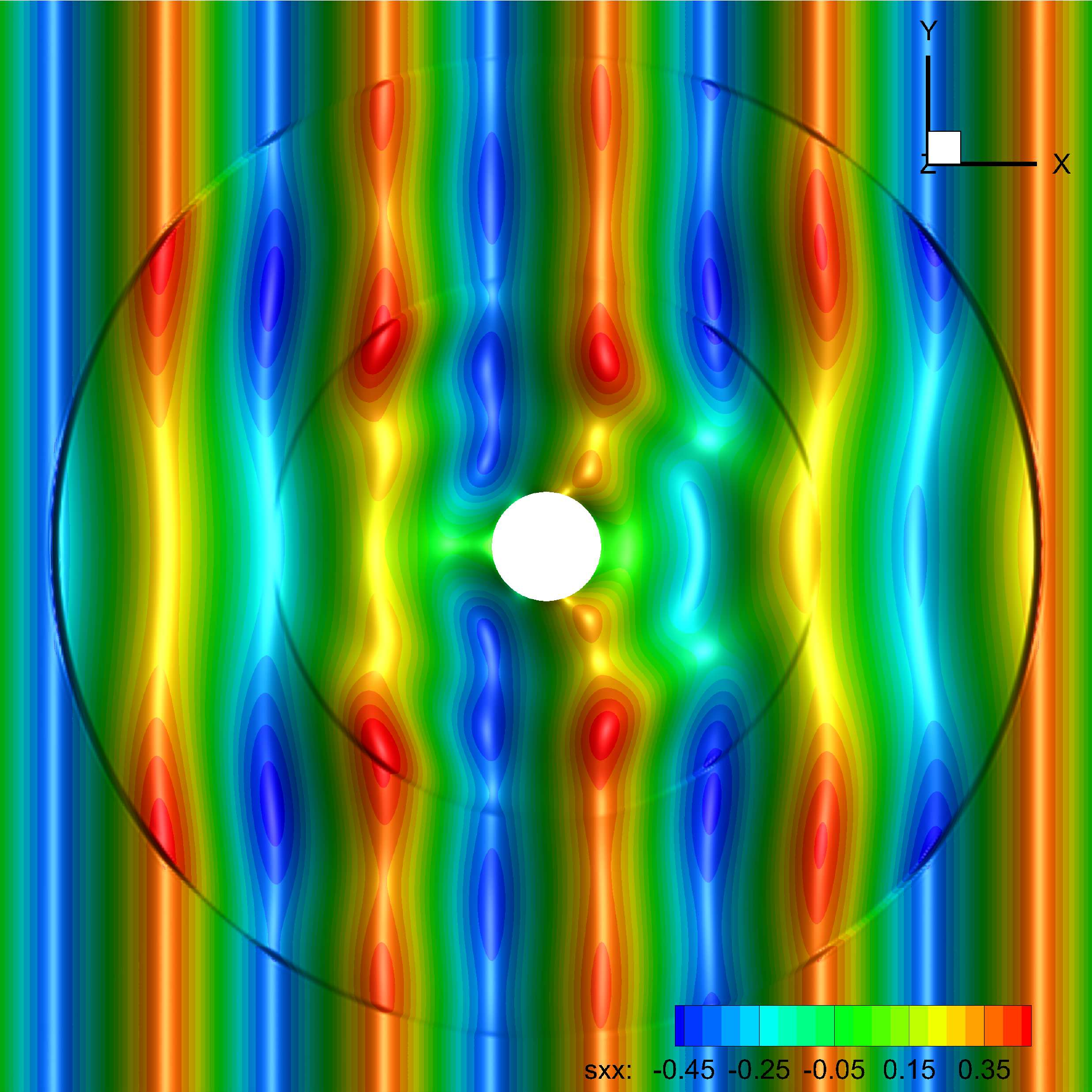} & 
\includegraphics[width=0.45\columnwidth]{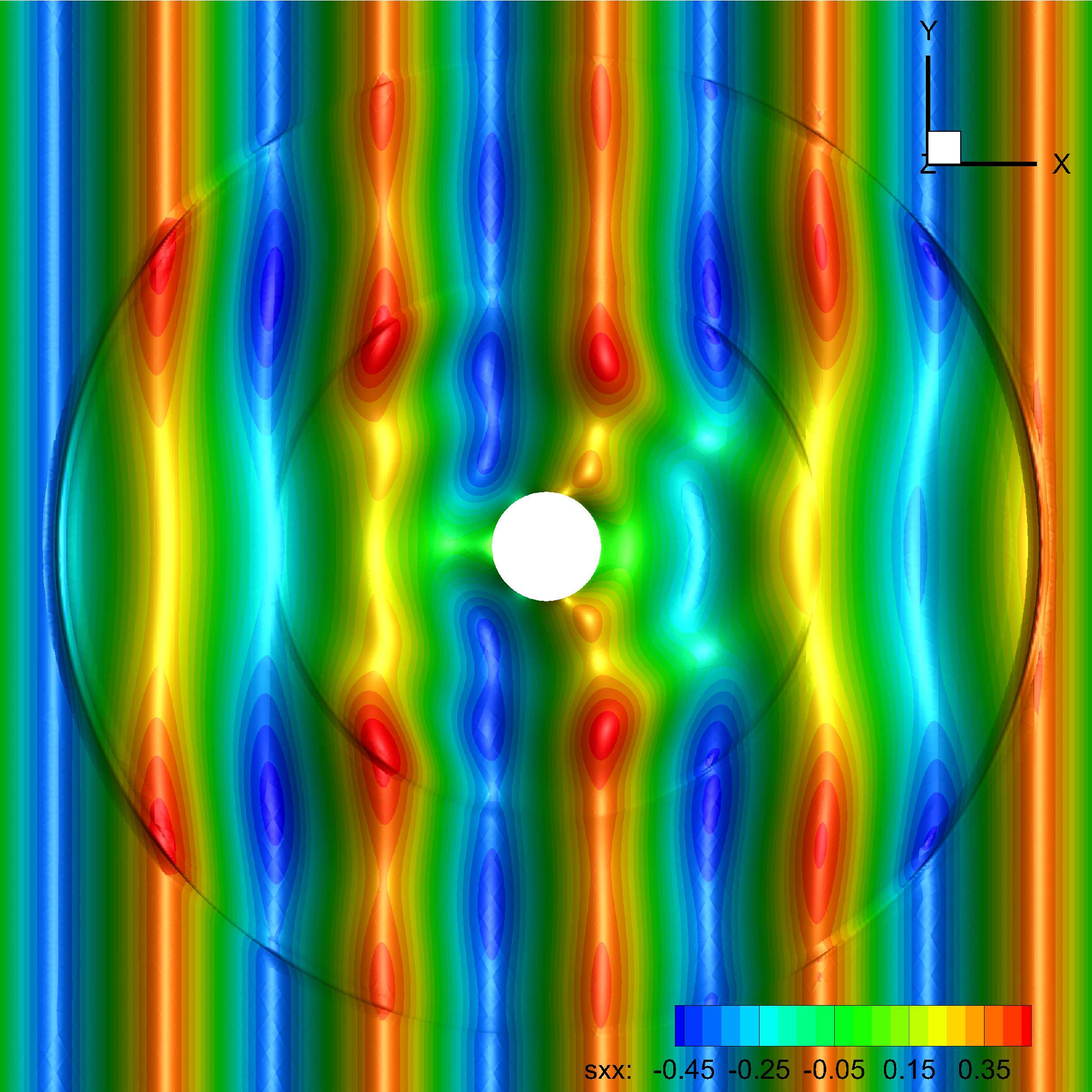} 
\end{tabular} 
\caption{Plane wave scattering on a circular cavity. Comparison of the isocontours of the stress tensor component $\sigma_{xx}$ between the reference solution given by an explicit ADER-DG scheme (left) and our new staggered space-time DG scheme (right).}%
\label{fig.NT1.2}%
\end{figure*}
\begin{figure*}%
\includegraphics[width=0.49\columnwidth]{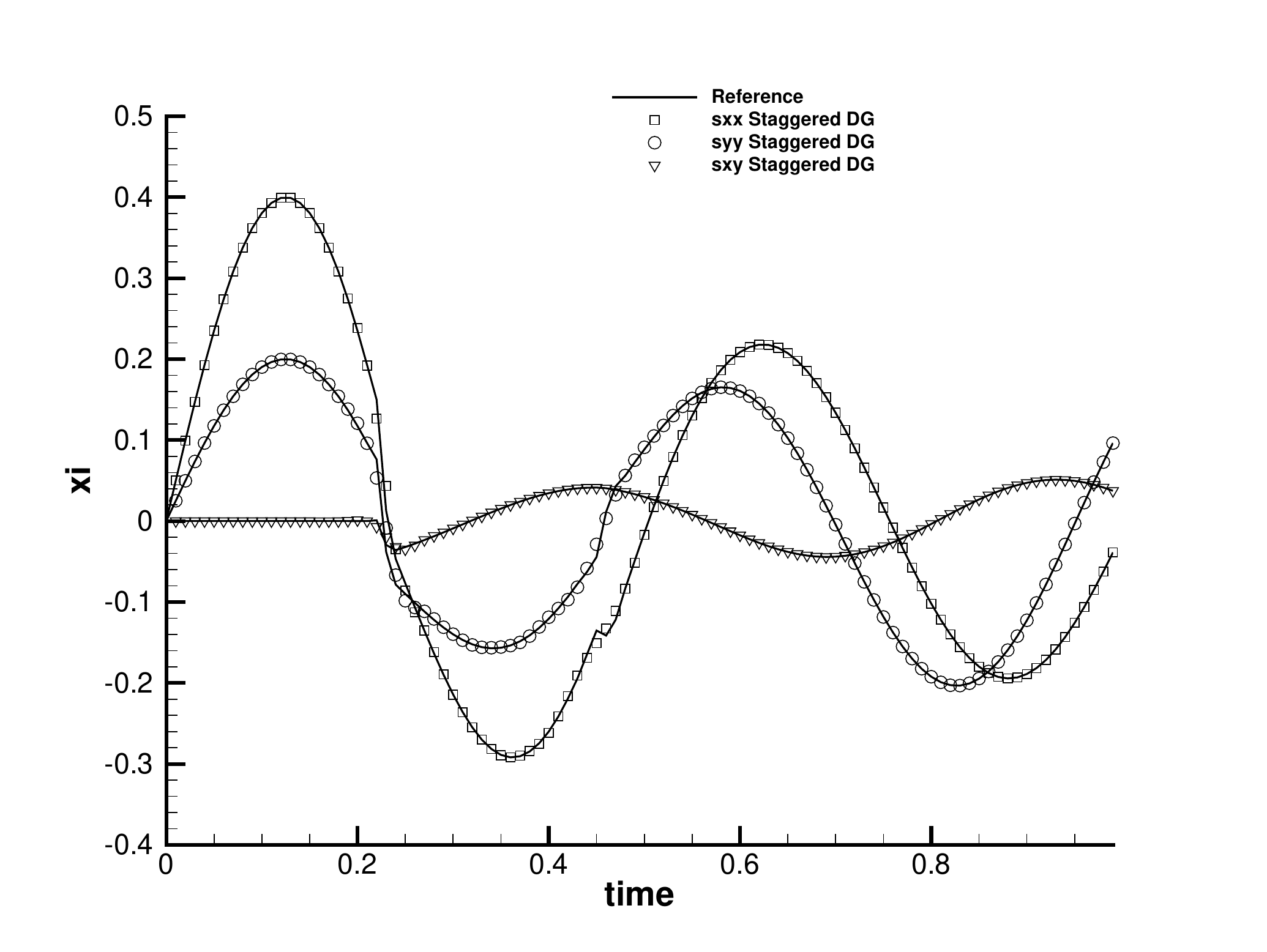} 
\includegraphics[width=0.49\columnwidth]{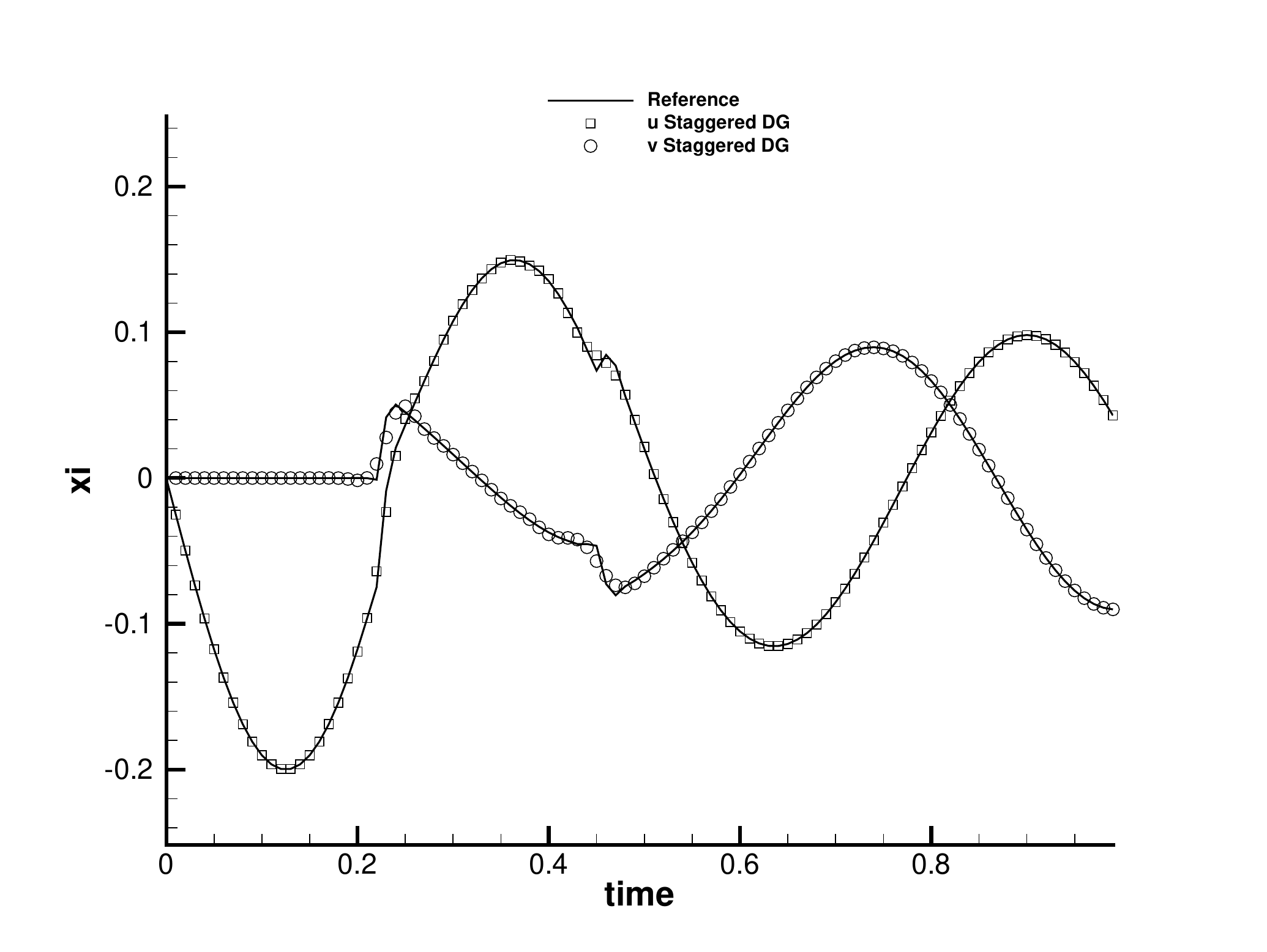} \\
\includegraphics[width=0.49\columnwidth]{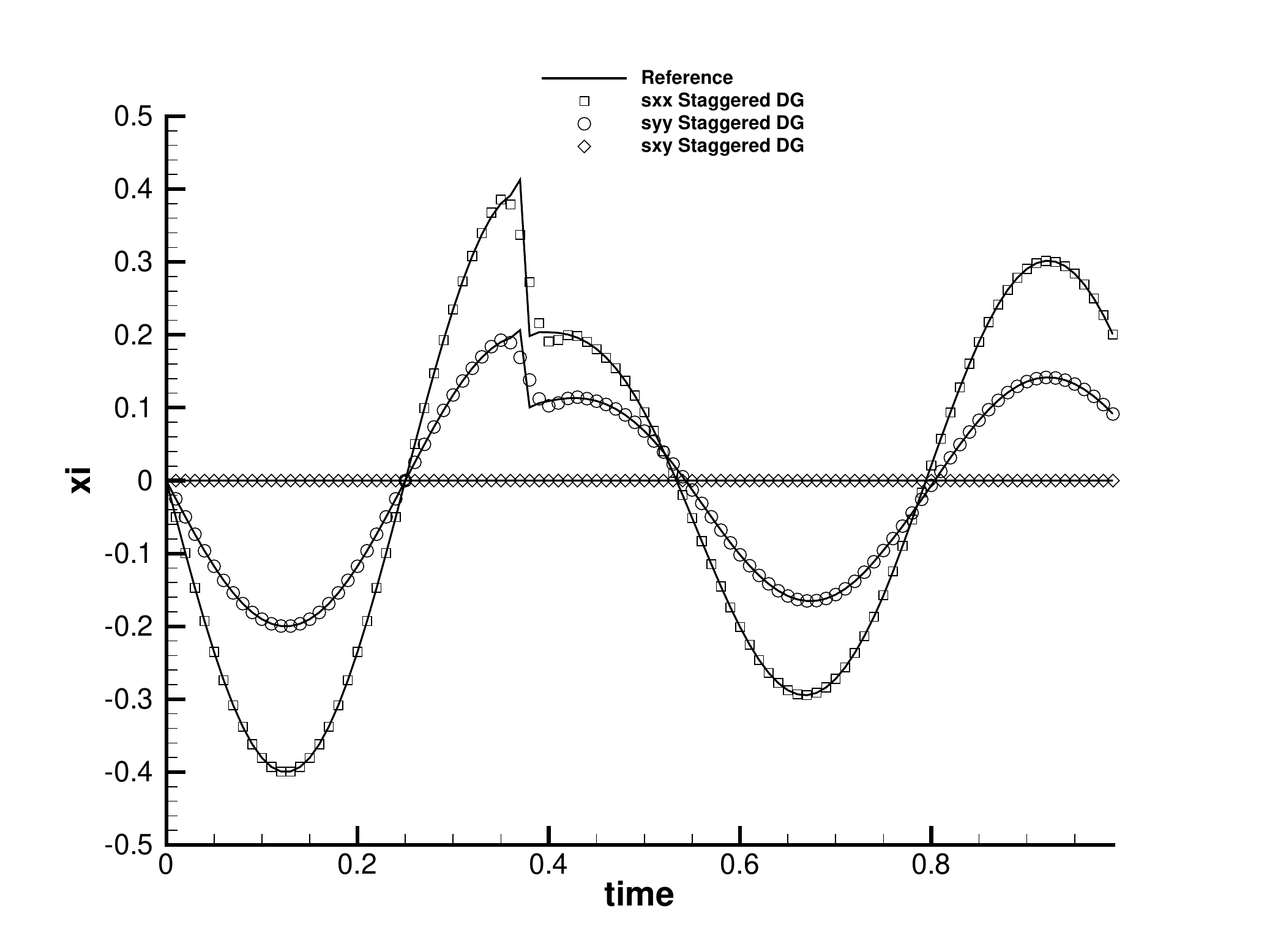} 
\includegraphics[width=0.49\columnwidth]{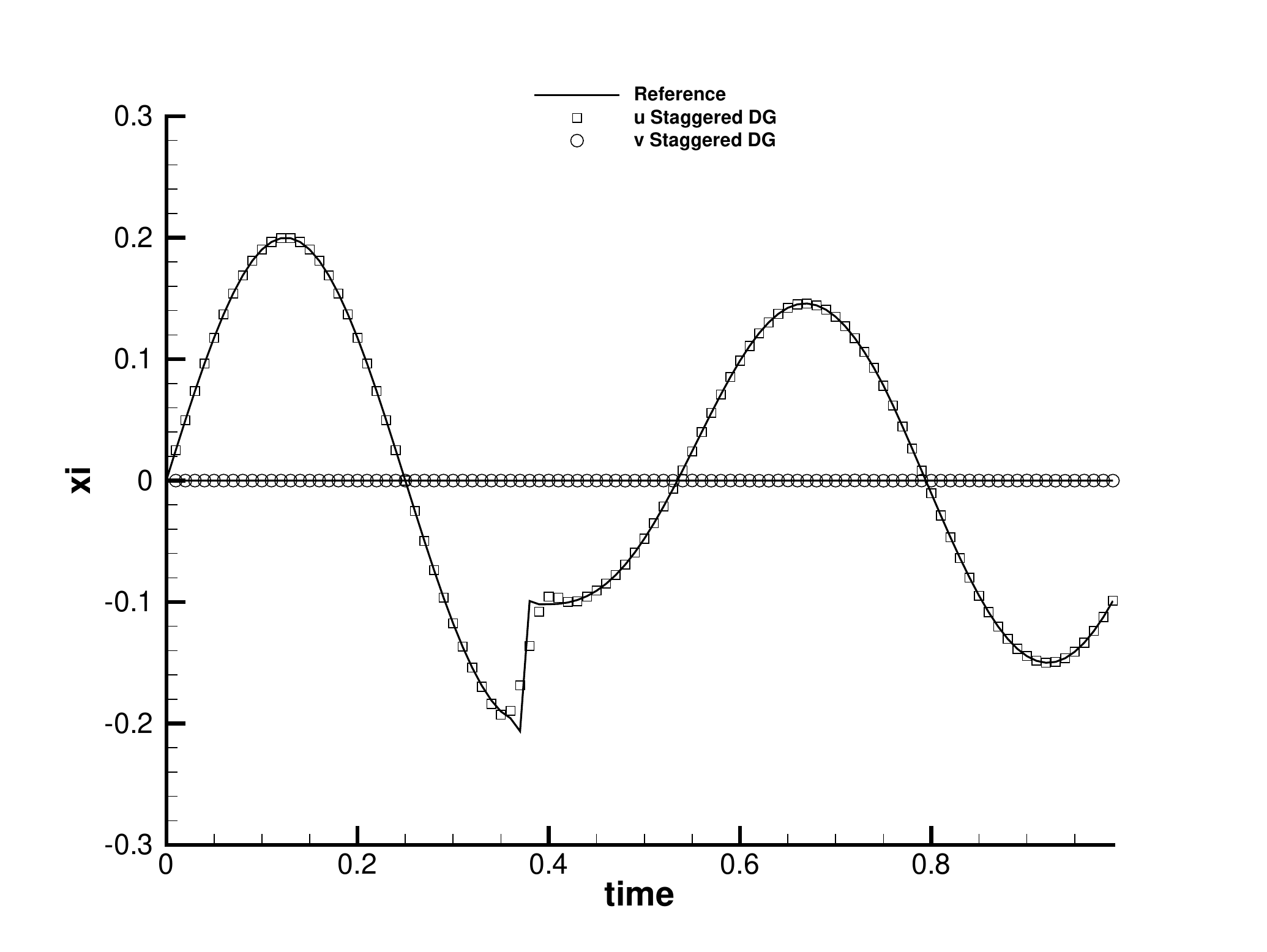}
\caption{Comparison of the stress tensor components $\sigma_{xx}$, $\sigma_{yy}$, $\sigma_{xy}$ (left) and velocity components $u,v$ (right) in the receiver point $\vec{x}_1=(0.5,0.5)$ (top) and $\vec{x}_2=(1.0,0.0)$ (bottom).}%
\label{fig.NT1.3}%
\end{figure*}
%
%
\subsection{Numerical convergence test}
\label{sec.convtest}
In this test we verify the order of accuracy and the computational efficiency of our new staggered space-time DG schemes for linear elasticity. Following \cite{gij1} we consider a combination of a $p-$ and an $s-$wave in a  square domain $\Omega=[-1.5,1.5]^2$ extended with periodic boundaries everywhere. As initial state we take 
\begin{eqnarray}
	\mathbf{U}(\xx,0) = \alpha \mathbf{r}_{p} \sin(k \cdot \vec{x})+\alpha \mathbf{r}_{s} \sin( \mathbf{k} \cdot \vec{x}),
\label{eq:NT2.1}
\end{eqnarray}
where $\alpha=0.1$; $\mathbf{k} = 2\pi \, \vec{n}$; $\vec{n}=(n_x,n_y)=(1,1)$; $\mathbf{r}_{p}$ and $\mathbf{r}_{s}$ are the eigenvectors associated with the $p-$ and $s-$ wave: 
\begin{equation}
	\mathbf{r}_{s}=\left( -2\mu n_x n_y, 2\mu n_x n_y, \mu(n_x^2-n_y^2), n_y c_s, -n_x c_s \right), \qquad 
	\mathbf{r}_{p}=\left( \lambda+2\mu n_x^2, \lambda+2\mu n_y^2, 2\mu n_x n_y, - n_x c_p, -n_y c_p \right),
\label{eq:NT2.2}
\end{equation}
with the $p-$wave speed $c_p = \sqrt{(\lambda + 2\mu)/\rho}$ and the $s-$wave speed $c_s=\sqrt{\mu / \rho}$. We set $(\lambda, \mu, \rho)=(2,1,1)$. The final time is $t_{end}=3\sqrt{2}$ so that the resulting exact solution has to be the the same as the initial one 
i.e. $\mathbf{U}(\xx,t_{end})=\mathbf{U}(\xx,0)$. In Table \ref{tab:NT1} we report the resulting $L_2$ error norms for the  entire state vector $\mathbf{U}$ 
 and the order of convergence for different polynomial approximation degrees $p=p_{\gamma}$ on a sequence of successively refined meshes of characteristic size {$h=(0.1264, 0.0842, 0.0842, 0.0505, 0.0421)$}. 
The time step size has been chosen proportional to the mesh spacing $h$ as {$\Delta t = K h$, with $K=0.112$, independent of the polynomial degree $p$}.  We also report the 
wall clock times $T_{CPU}$ {measured on $20$ cores of an Intel Xeon E5 CPU with 2.5 GHz clock speed and 128 GB of RAM.} 
From Table \ref{tab:NT1} the optimal order of convergence can be observed for all variables.  
\begin{table*}%
\begin{tabular}{cccccccccccccc}
	\hline
	$p$ & $N_i$ 	& $u$ 			& 						& $v$ 			& 						& $\sigma_{xx}$ & 						& $\sigma_{yy}$ & 						& $\sigma_{xy}$ & 						& $T_{CPU}$ \\
	\hline
	1		&	1760	& 1.253E-01 & {\small    }& 2.675E-01 & {\small    }& 5.111E-01 		& {\small    }& 3.003E-01 		& {\small    }& 1.343E-01 		& {\small    }& 4.4\\
	1		&	3960	& 4.609E-02 & {\small 2.5}& 1.284E-01 & {\small 1.8}& 2.428E-01 		& {\small 1.8}& 1.248E-01 		& {\small 2.2}& 6.143E-02 		& {\small 1.9}& 11.7\\
	1		&	7040	& 2.479E-02 & {\small 2.2}& 7.356E-02 & {\small 1.9}& 1.387E-01 		& {\small 1.9}& 6.938E-02 		& {\small 2.0}& 3.481E-02 		& {\small 2.0}& 24.3\\
	1		&	11000	& 1.567E-02 & {\small 2.1}& 4.741E-02 & {\small 2.0}& 8.931E-02 		& {\small 2.0}& 4.430E-02 		& {\small 2.0}& 2.235E-02 		& {\small 2.0}& 42.7\\
	\hline
\end{tabular}
\begin{tabular}{cccccccccccccc}
	$p$ & $N_i$ & $u$ 			& 						& $v$ 			& 						& $\sigma_{xx}$ & 						& $\sigma_{yy}$ & 						& $\sigma_{xy}$ & 						& $T_{CPU}$ \\
	\hline
	2		&	1760	& 1.512E-03 & {\small    }& 3.249E-03 & {\small    }& 6.081E-03 		& {\small    }& 3.156E-03 		& {\small    }& 1.574E-03 		& {\small    }& 27.6\\
	2		&	3960	& 3.697E-04 & {\small 3.5}& 6.568E-04 & {\small 3.9}& 1.218E-03 		& {\small 4.0}& 6.411E-04 		& {\small 3.9}& 3.186E-04 		& {\small 3.9}& 90.1\\
	2		&	7040	& 1.416E-04 & {\small 3.3}& 2.118E-04 & {\small 3.9}& 3.882E-04 		& {\small 4.0}& 2.086E-04 		& {\small 3.9}& 1.031E-04 		& {\small 3.9}& 198.1\\
	2		&	11000	& 6.901E-05 & {\small 3.2}& 8.872E-05 & {\small 3.9}& 1.601E-04 		& {\small 4.0}& 8.835E-05 		& {\small 3.9}& 4.324E-05 		& {\small 3.9}& 364.4\\
	\hline
\end{tabular}
\begin{tabular}{cccccccccccccc}
	$p$ & $N_i$ 	& $u$ 			& 						& $v$ 			& 						& $\sigma_{xx}$ & 						& $\sigma_{yy}$ & 						& $\sigma_{xy}$ & 						& $T_{CPU}$ \\
	\hline
	3		&	1760	& 5.522E-05 & {\small    }& 3.323E-05 & {\small    }& 4.781E-05 		& {\small    }& 3.835E-05 		& {\small    }& 1.919E-05 		& {\small    }& 153.3\\
	3		&	3960	& 1.079E-05 & {\small 4.0}& 5.544E-06 & {\small 4.4}& 6.534E-06 		& {\small 4.9}& 6.313E-06 		& {\small 4.4}& 3.250E-06 		& {\small 4.4}& 450.5\\
	3		&	7040	& 3.414E-06 & {\small 4.0}& 1.677E-06 & {\small 4.2}& 1.824E-06 		& {\small 4.4}& 1.906E-06 		& {\small 4.2}& 9.790E-07 		& {\small 4.2}& 998.0\\
	3		&	11000	& 1.396E-06 & {\small 4.0}& 6.827E-07 & {\small 4.0}& 7.183E-07 		& {\small 4.2}& 7.668E-07 		& {\small 4.1}& 3.983E-07 		& {\small 4.0}& 1811.5\\
	\hline
\end{tabular}
\begin{tabular}{cccccccccccccc}
	$p$ & $N_i$ 	& $u$ 			& 						& $v$ 			& 						& $\sigma_{xx}$ & 						& $\sigma_{yy}$ & 						& $\sigma_{xy}$ & 						& $T_{CPU}$ \\
	\hline
	4		&	1760	& 2.480E-06 & {\small    }& 1.216E-06 & {\small    }& 1.400E-06 		& {\small    }& 1.434E-06 		& {\small    }& 6.596E-07 		& {\small    }&183.0\\
	4		&	3960	& 3.270E-07 & {\small 5.0}& 1.582E-07 & {\small 5.0}& 1.820E-07 		& {\small 5.0}& 1.869E-07 		& {\small 5.0}& 8.319E-08 		& {\small 5.1}&984.6\\
	4		&	7040	& 7.724E-08 & {\small 5.0}& 3.733E-08 & {\small 5.0}& 4.292E-08 		& {\small 5.0}& 4.418E-08 		& {\small 5.0}& 1.933E-08 		& {\small 5.1}&2476.2\\
	4		&	11000	& 2.532E-08 & {\small 5.0}& 1.218E-08 & {\small 5.0}& 1.402E-08 		& {\small 5.0}& 1.442E-08 		& {\small 5.0}& 6.278E-09 		& {\small 5.0}&9466.8\\
	\hline
\end{tabular}
\caption{Numerical convergence test: $L_2$ error norm, numerical convergence rates and CPU time $T_{CPU}$ for all variables for $p=p_\gamma=1 \ldots 4$.}
\label{tab:NT1}
\end{table*}

\subsection{$2D$ tilted Lamb problem}
\label{NS:2DLamb}
In this test case we study the two dimensional tilted Lamb problem, as suggested in \cite{Komatitsch1998,gij1}. The computational domain $\Omega=\{(x,y)\in \R^+ \, | \, 0 \leq x \leq 4000 \,\, , \,\, 0 \leq y \leq 2000+ x \tan \theta \}$ consists in a free surface with a tilt angle of $\theta=10^{\circ}$. The chosen $p-$ and $s-$wave velocities are set to $c_p=3200$ and $c_s=1847.5$, respectively. The mass density is taken as $\rho=2200$ so that the resulting Lam\'e constants are $\lambda=7.5096725\cdot 10^{9}$ and $\mu=7.50916375\cdot 10^{9}$. The initial condition is $\mathbf{U}=0$ everywhere in $\Omega$. The waves are generated by a  directional point source located in $\xx_s=(1720.0,2303.18)$. We place a receiver in $\xx_p=(2694.96, 2475.08)$, at a distance of 900 length units from the source. As reference solution we use the well  established ADER-DG method proposed in \cite{gij1,gij2,Dumbser2008} with $N=M=4$ and $\Ni=844560$. The numerical parameters of the new staggered space-time DG scheme are $p=4$, $p_\gamma=2$, $\Delta t=10^{-3}$ and $\Ni=33952$. The point source  
$$\mathbf{S}_v(\xx,t)=\frac{1}{\rho}\vec{\bm{d}} \, \delta(\xx-\xx_s)\mathcal{S}(t),$$ 
is characterized by a Dirac delta distribution in space located in $\xx_s$ and a temporal part, which is a Ricker wavelet defined as
\begin{eqnarray}
	\mathcal{S}(t)=a_1\left( 0.5+a_2(t-t_D)^2 \right),
\label{eq:NTLamb2D_1}
\end{eqnarray}
where $t_D=0.08s$ is the source delay time; $a_1=-2000$; $a_2=-(\pi f_c)^2$; and $f_c=14.5$.
Finally the vector $\vec{\bm{d}}=(-\sin\theta, \cos\theta ,0,0,0)^\top$ determines the direction of the source and depends on the tilt angle $\theta$. 
A comparison of the velocity component $v$ at $t=0.6$ is reported in Figure \ref{fig.NT3.3}. 
Figure \ref{fig.NT3.4} shows the comparison of the recorded seismograms in the receiver location $\xx_p$. An excellent agreement with the reference solution can be observed 
also in this case. 
\begin{figure*}%
\includegraphics[width=0.49\columnwidth]{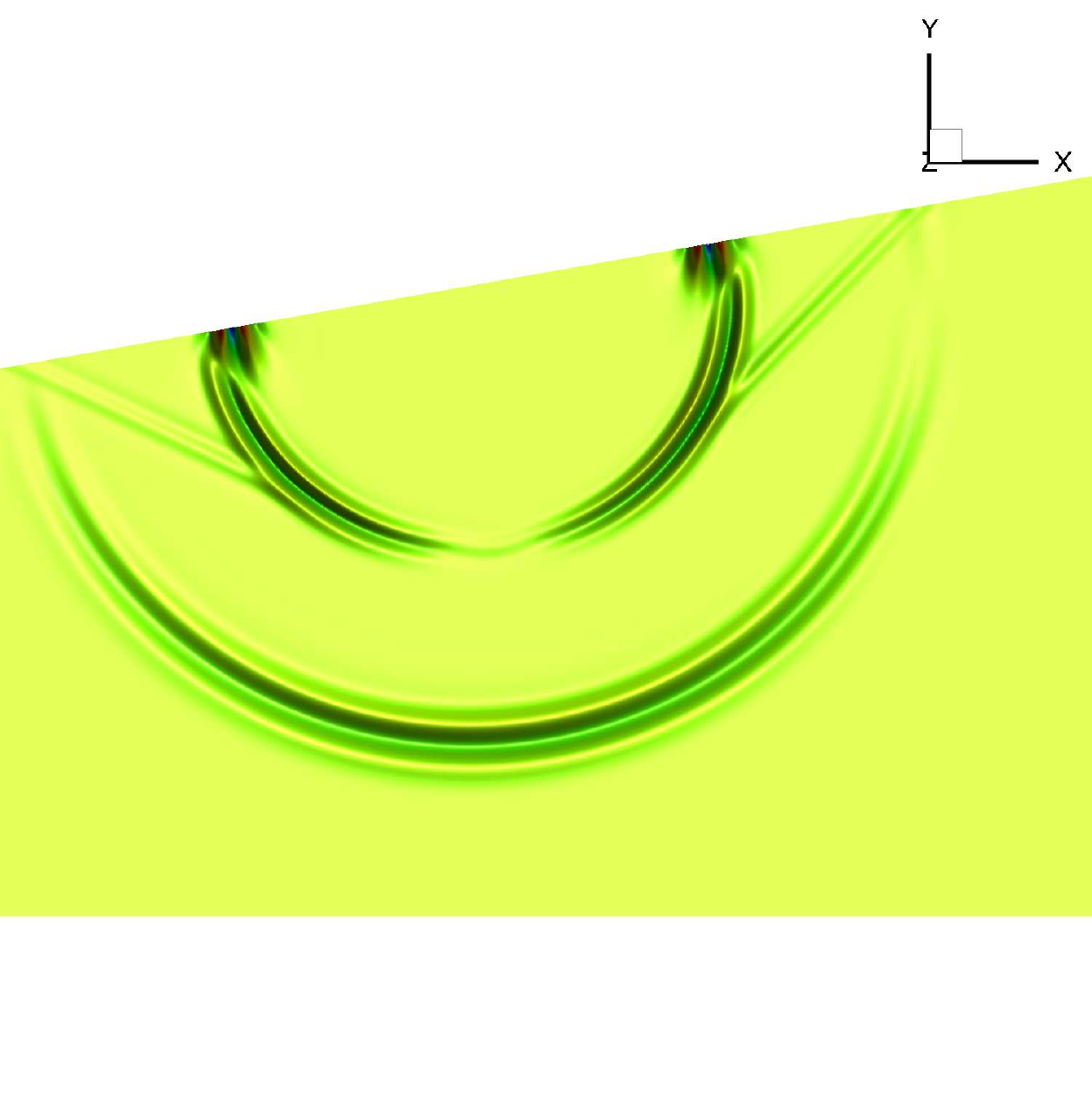} 
\includegraphics[width=0.49\columnwidth]{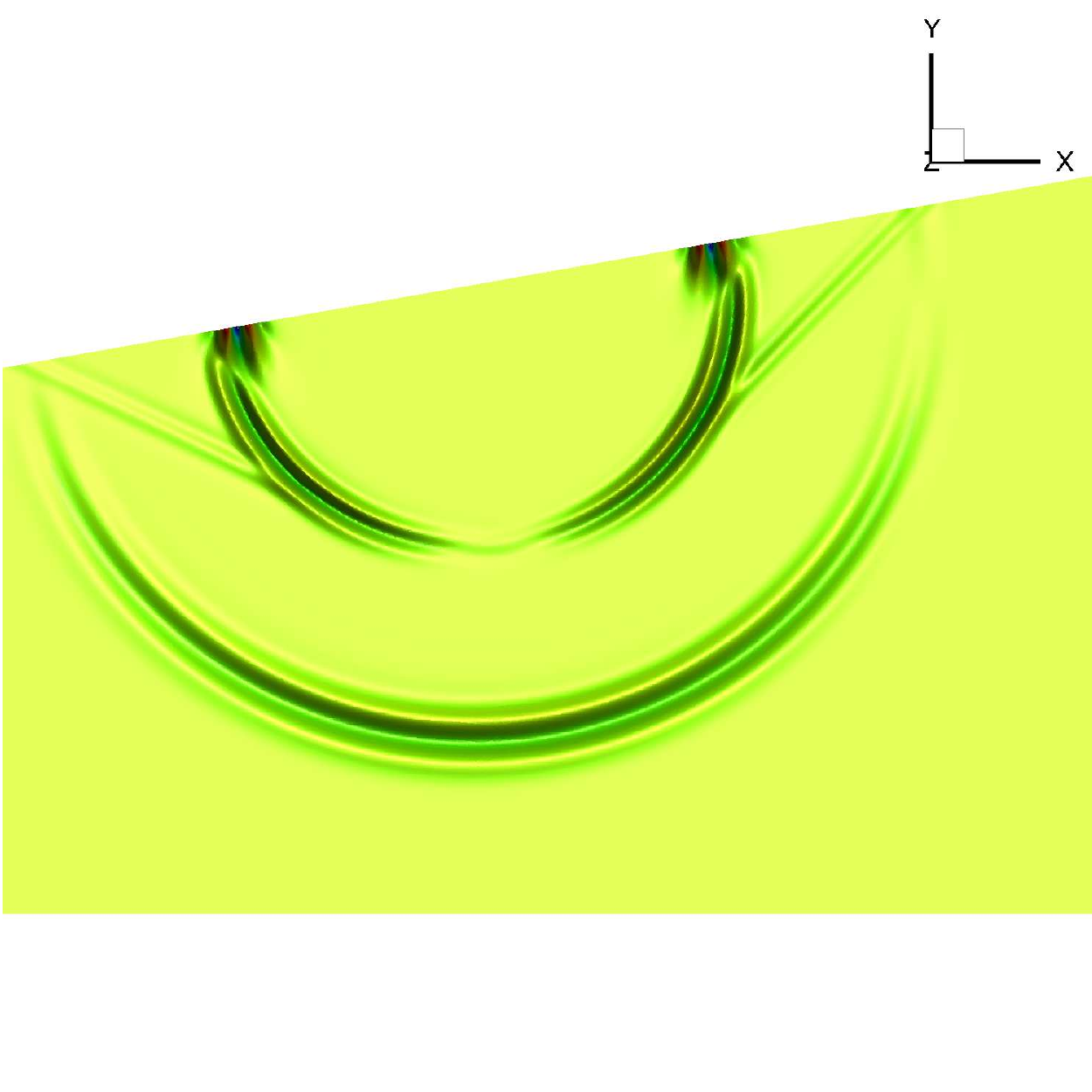} \\
\caption{Contours of the velocity component $v$ at time $t=0.6$ obtained with an explicit ADER-DG scheme (left) as reference and the new staggered space-time DG scheme (right).}%
\label{fig.NT3.3}%
\end{figure*}
\begin{figure*}%
\includegraphics[width=0.49\columnwidth]{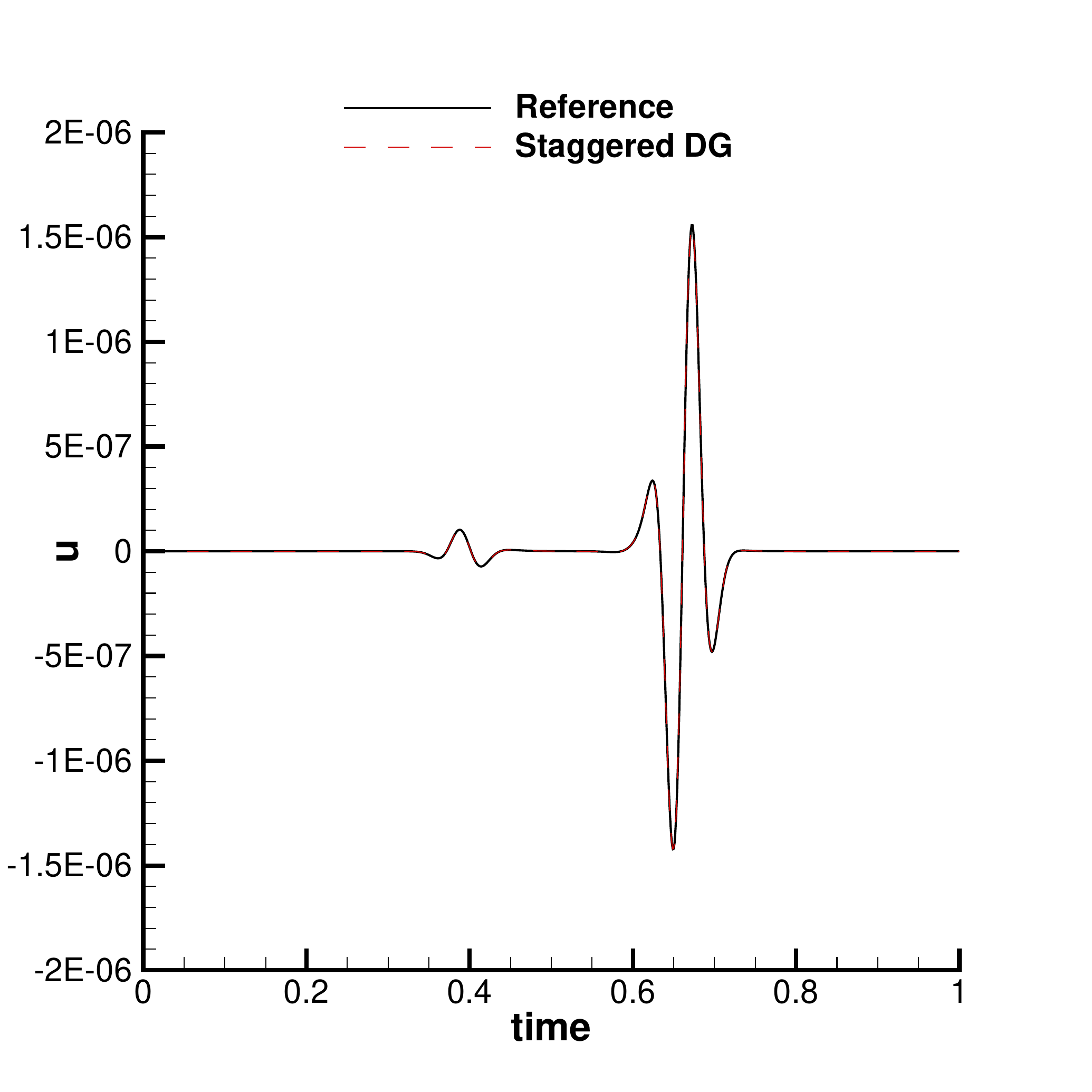} 
\includegraphics[width=0.49\columnwidth]{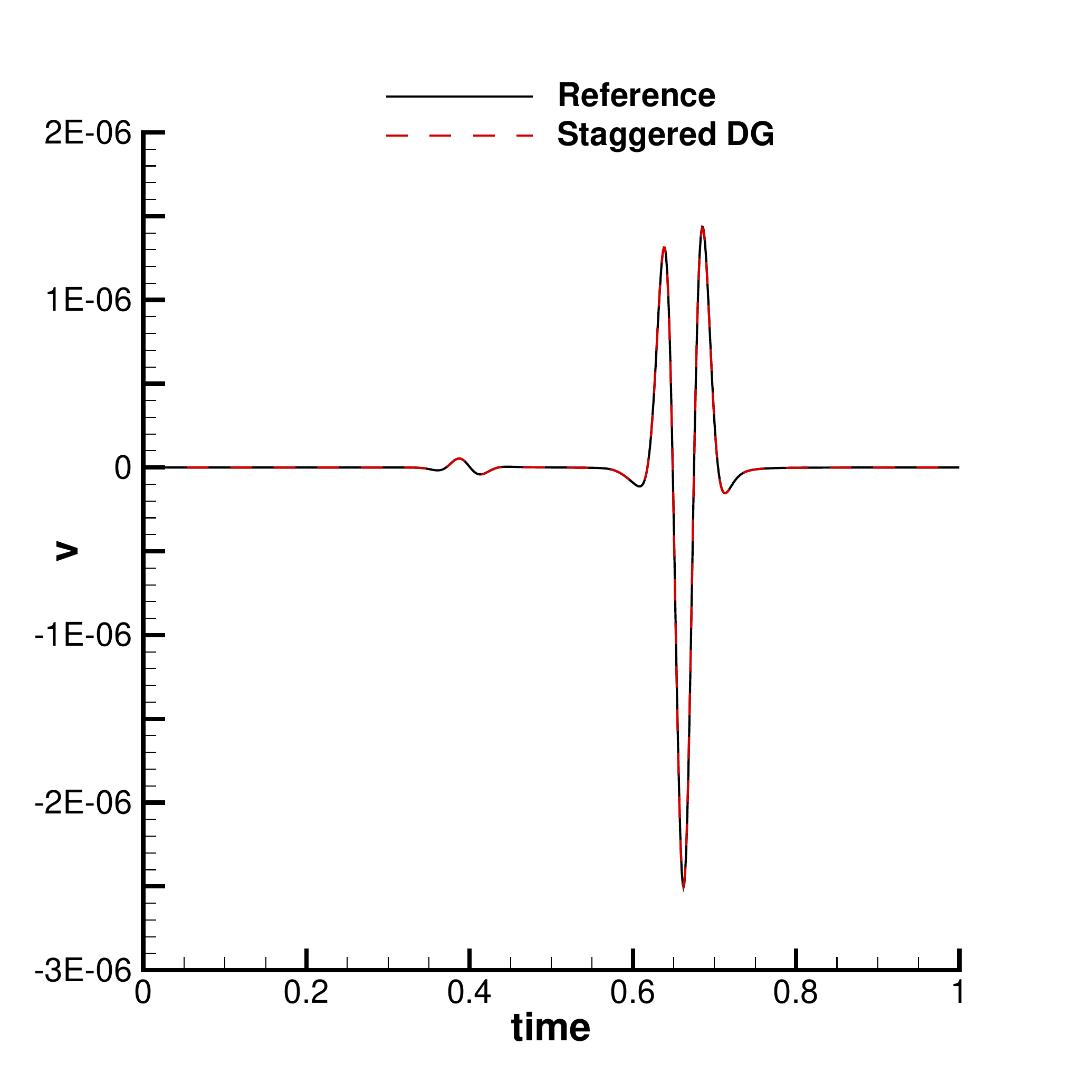} \\
\caption{Comparison of the numerical solution obtained with the new staggered space-time DG scheme and the reference solution for the velocity components $u$ and $v$ in the 
receiver point $\xx=\xx_p$ up to $t=1.0$.}%
\label{fig.NT3.4}%
\end{figure*}

\subsection{Wave propagation in complex geometry} 
This test case is very similar to the previous tilted Lamb problem, but in a non-trivial domain and using a heterogeneous medium. The computational domain is $\Omega=\{(x,y) \,\, | \,\, x \in [0,4000] \,\, y \in [0,f(x)]\}$ where the location of the free surface boundary is defined by the function $f(x)=2000+100\left( \sin{(\frac{3}{200}x)}+ \sin{(\frac{2}{200}x)} \right)$.  
The material is heterogeneous and consists in two layers with different material properties. The first layer is placed in $\{y>1500-\frac{x}{2}\}$ with $c_p=3200$ and $c_s=1847.5$, 
while the second layer covers the rest of the domain with $c_p=2262.74$ and $c_s=1306.38$. We use free surface boundary conditions everywhere. The same point source as  
described in the previous Section \ref{NS:2DLamb} is used (with $\theta=10^\circ$ as before) and is located in $\xx_s=(3000,1500.18)$. We place three seismogram recorders 
in $\xx_1=(893.80,1994.83)$, $\xx_2=(1790.0,880.0)$ and $\xx_3=(1000.0,500.0)$. The computational domain, the position of the source point and the position of the receivers 
are depicted in Figure \ref{fig.NT4.1}. 
\begin{figure}%
\includegraphics[width=1.0\columnwidth]{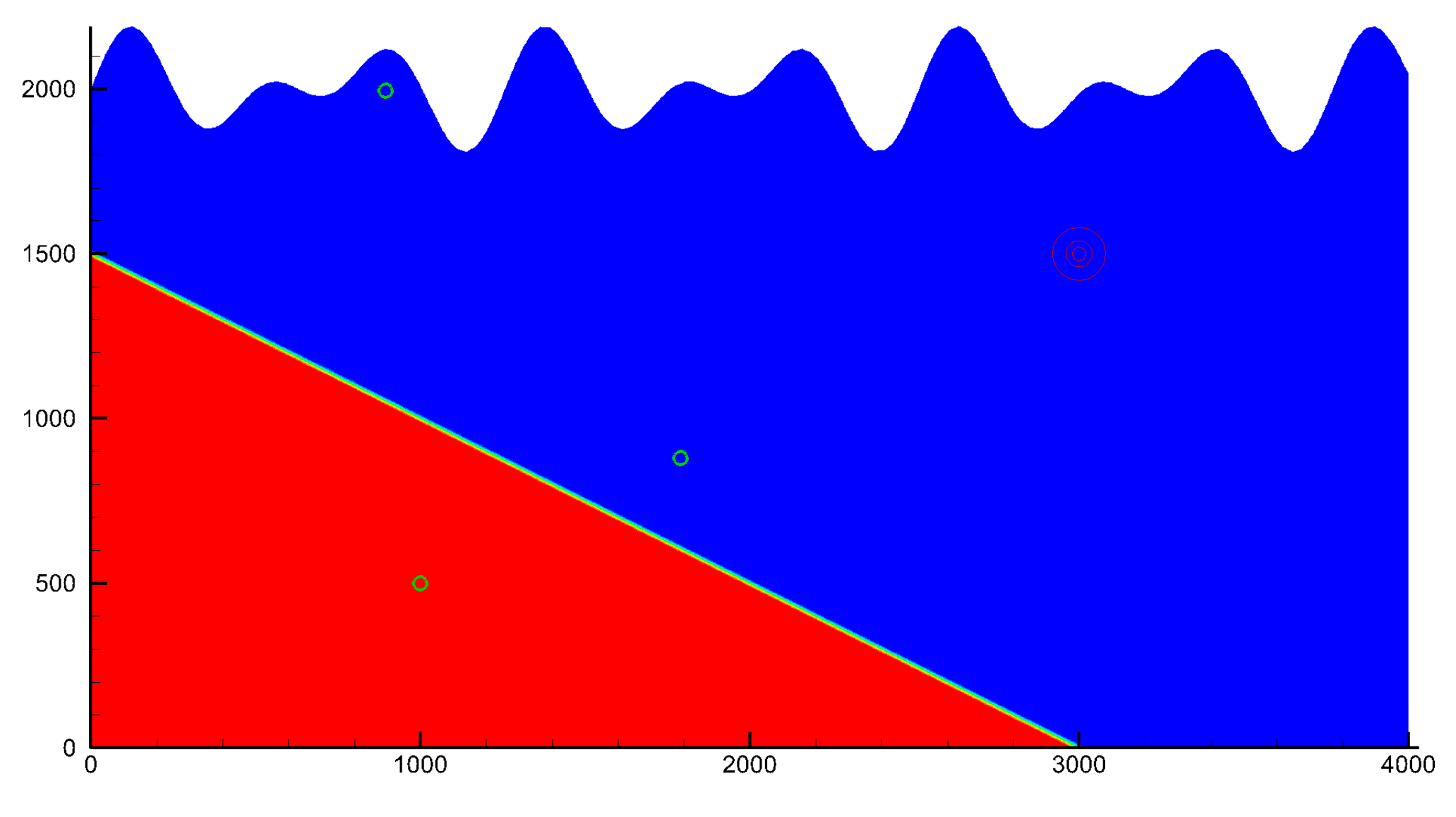} 
\caption{Wave propagation in complex geometry. Computational domain with the point source highlighted in red and the receivers in green.}%
\label{fig.NT4.1}%
\end{figure}
The computational domain is discretized using only $N_i=7352$ triangles of characteristic mesh spacing {$h=58.50$} and the polynomial approximation degrees are chosen as $p=4$ in space and $p_\gamma=2$ in time. We run the simulation up to $t=5$ and we set $\Delta t=10^{-3}$. We compare our numerical solution again with the well established ADER-DG method proposed in \cite{gij1,gij2,Dumbser2008} with $N=M=4$ on the same spatial mesh. A comparison of the numerical solution with the reference solution is reported at several times in Figure \ref{fig.NT4.2}, while the time  series of the velocity component $v$ in the three receiver points is reported in Figure \ref{fig.NT4.3}. In all cases we can observe a very good agreement with the reference solution. 
\begin{figure*}%
\includegraphics[width=0.49\columnwidth]{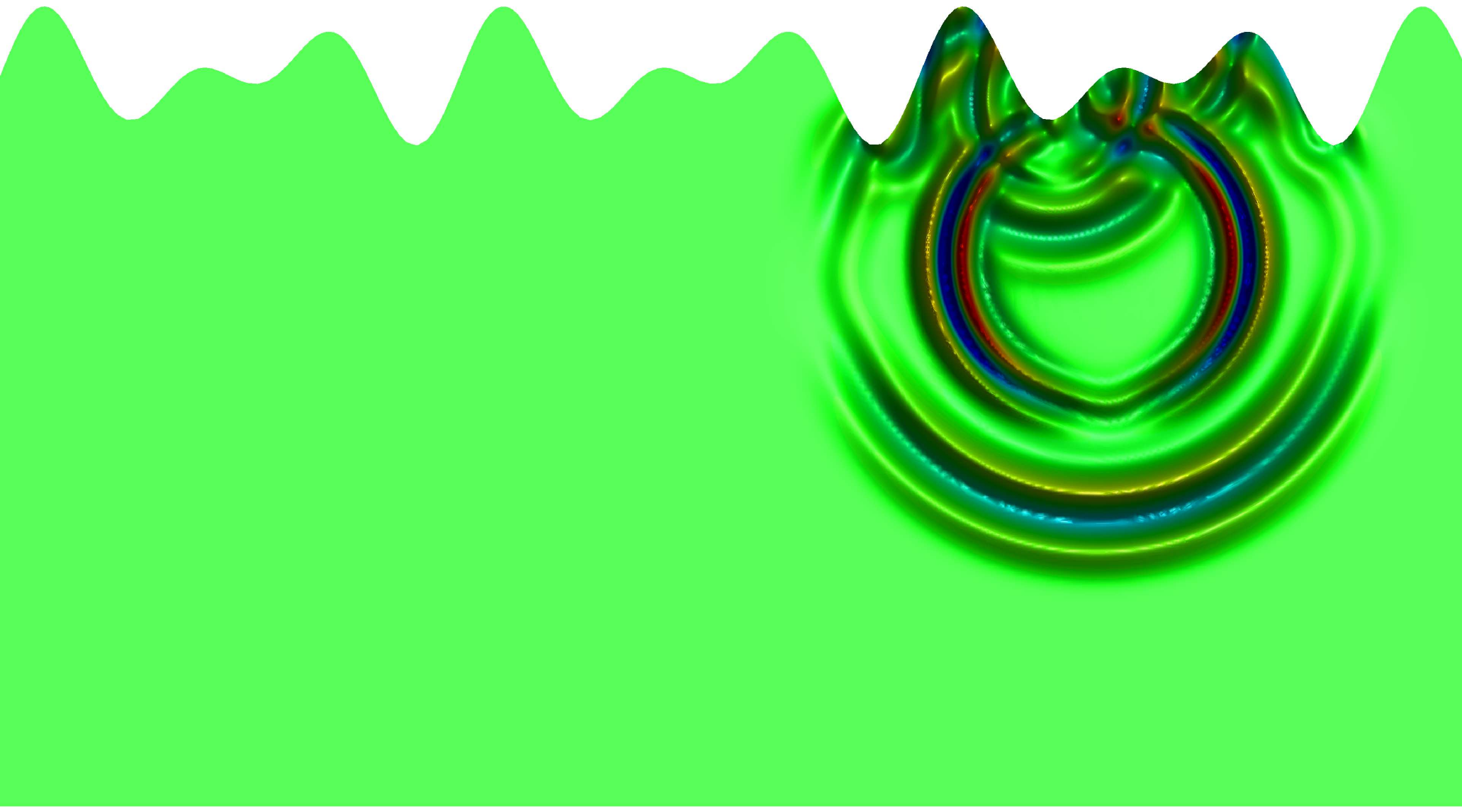} 
\includegraphics[width=0.49\columnwidth]{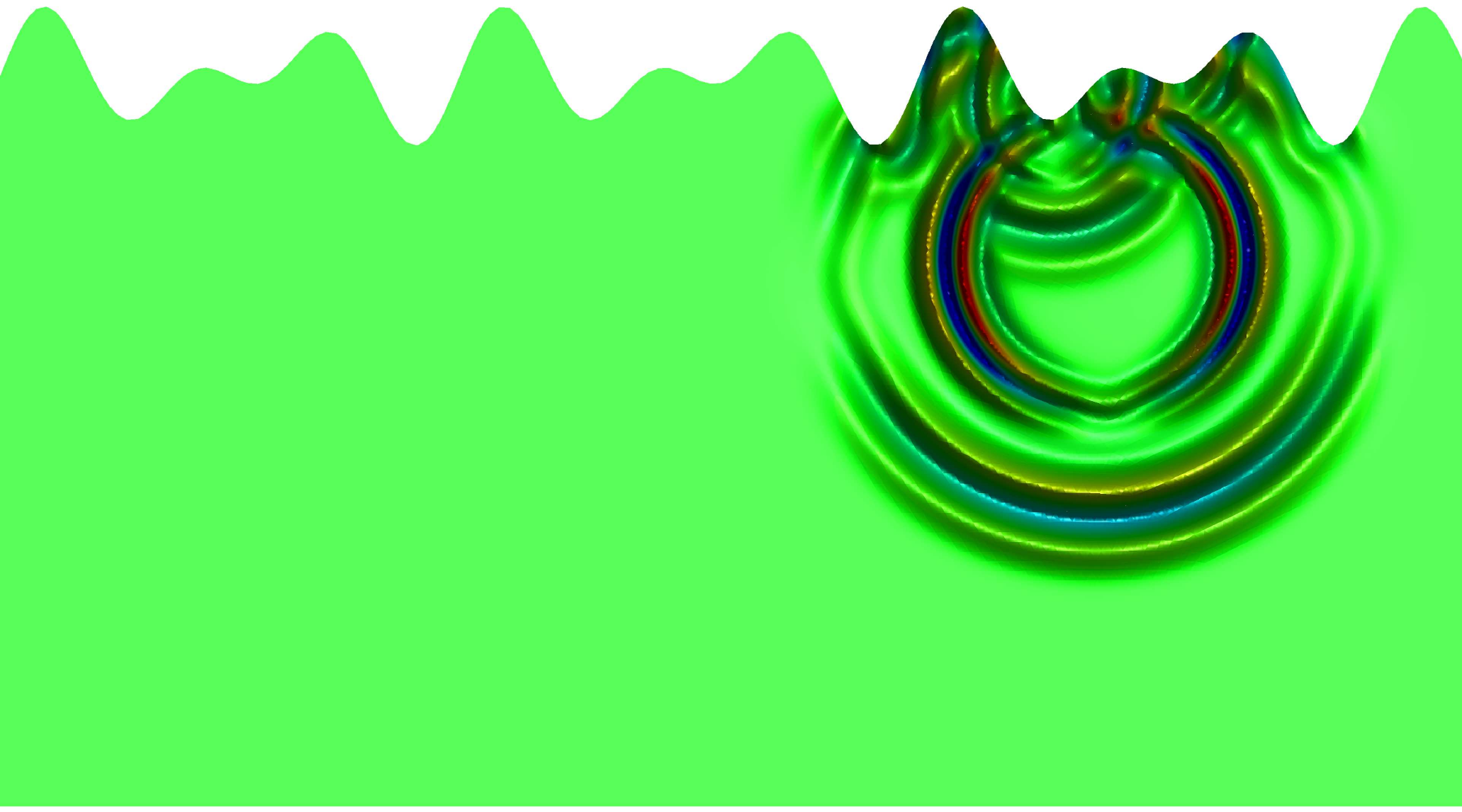} \\
\includegraphics[width=0.49\columnwidth]{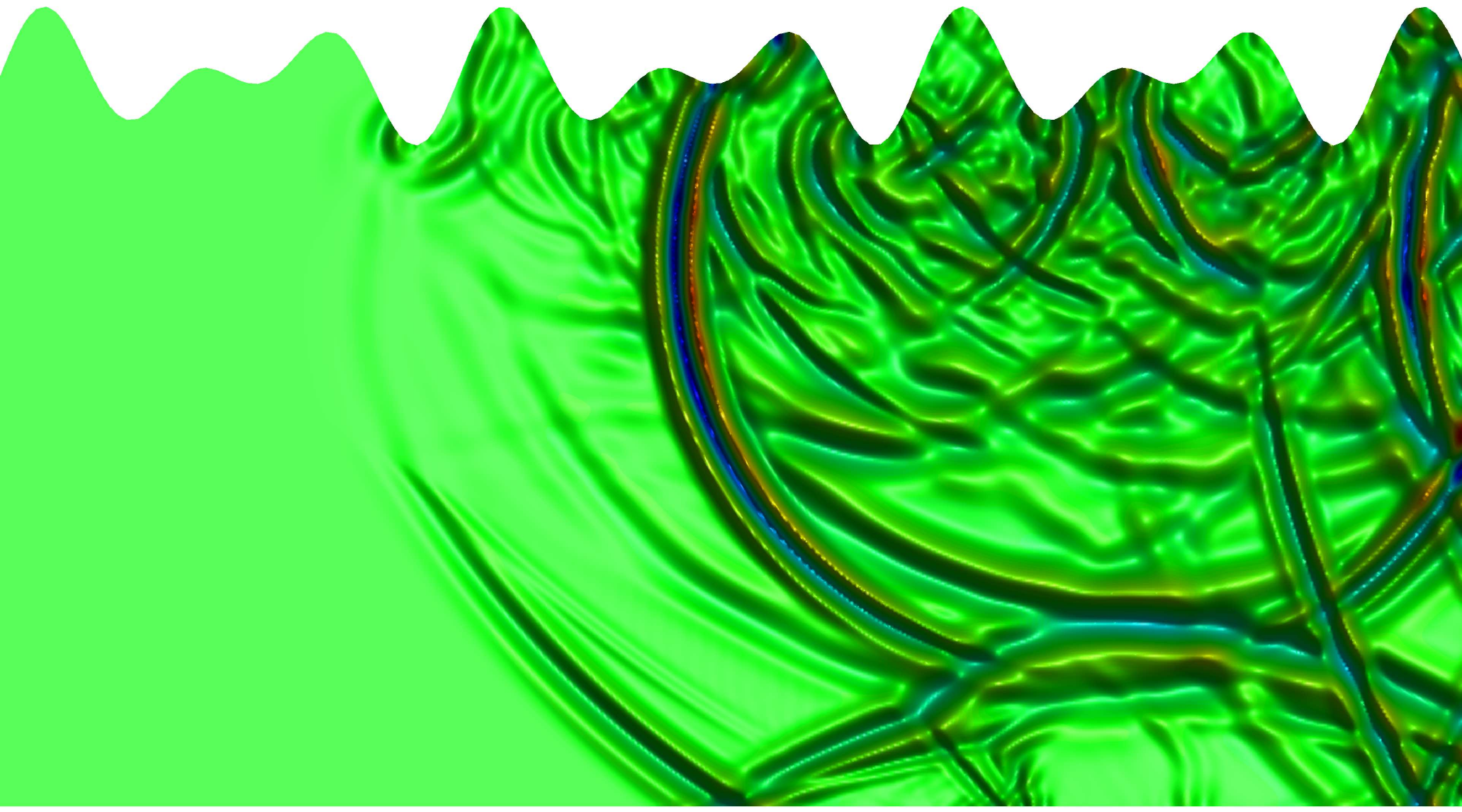} 
\includegraphics[width=0.49\columnwidth]{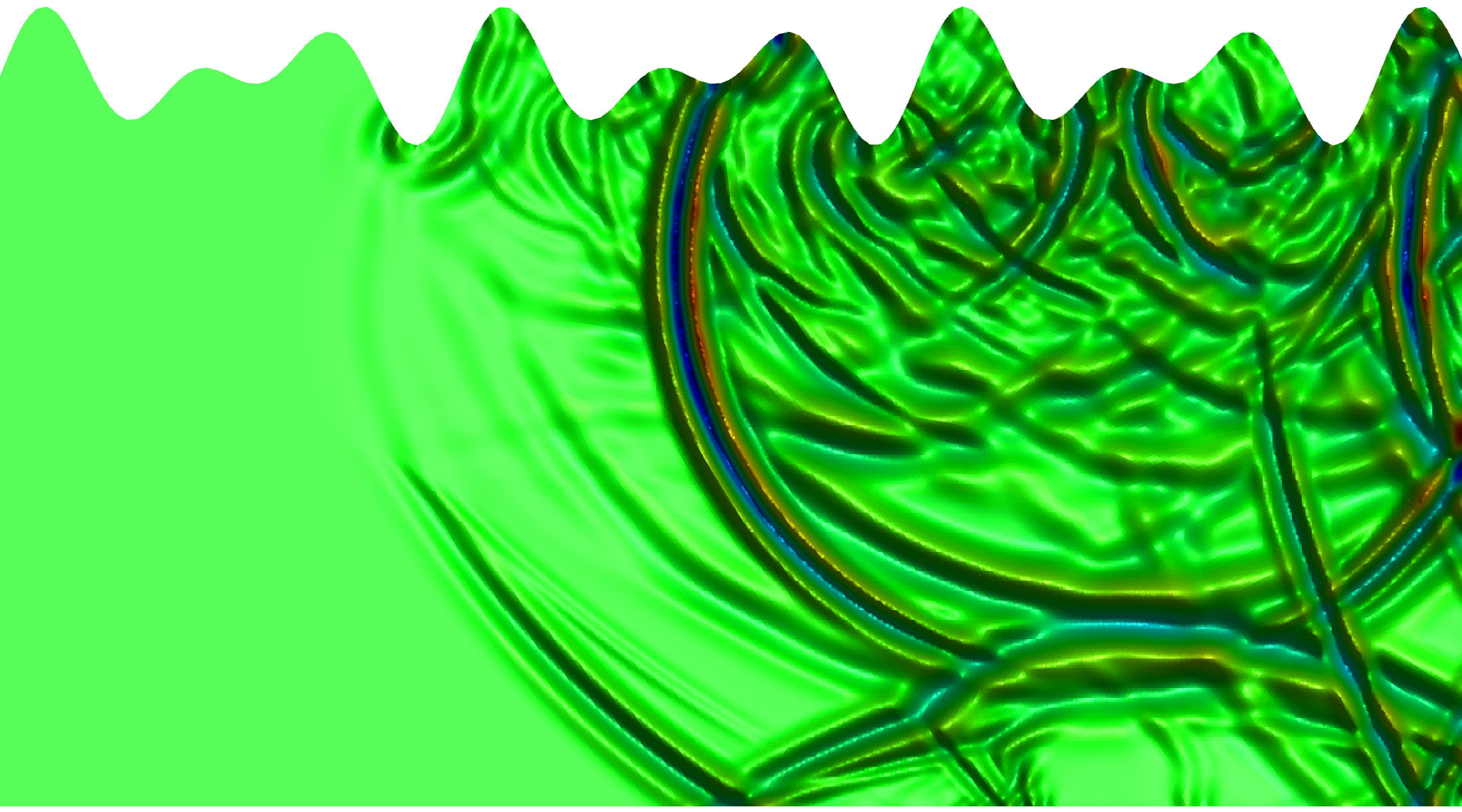} \\
\includegraphics[width=0.49\columnwidth]{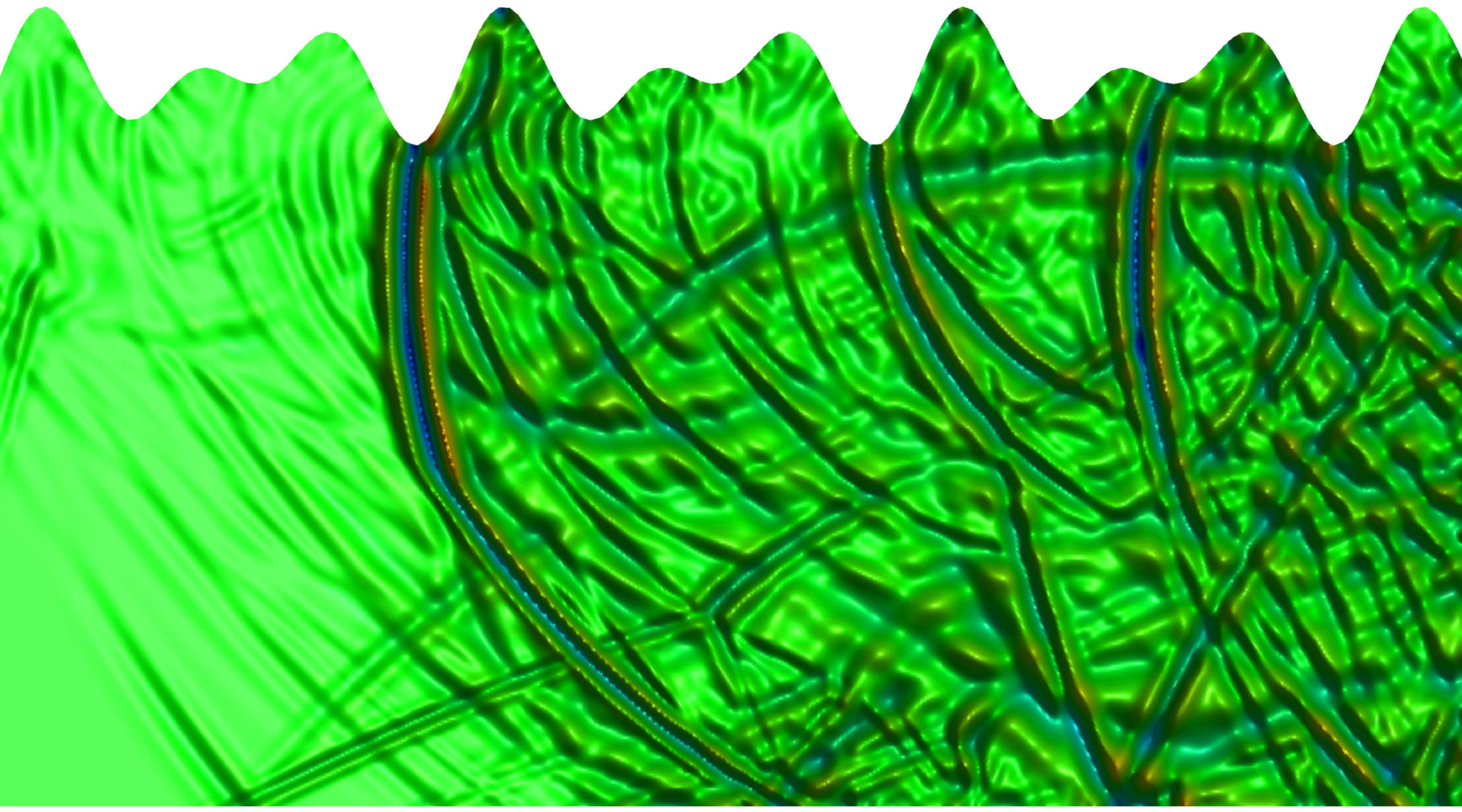} 
\includegraphics[width=0.49\columnwidth]{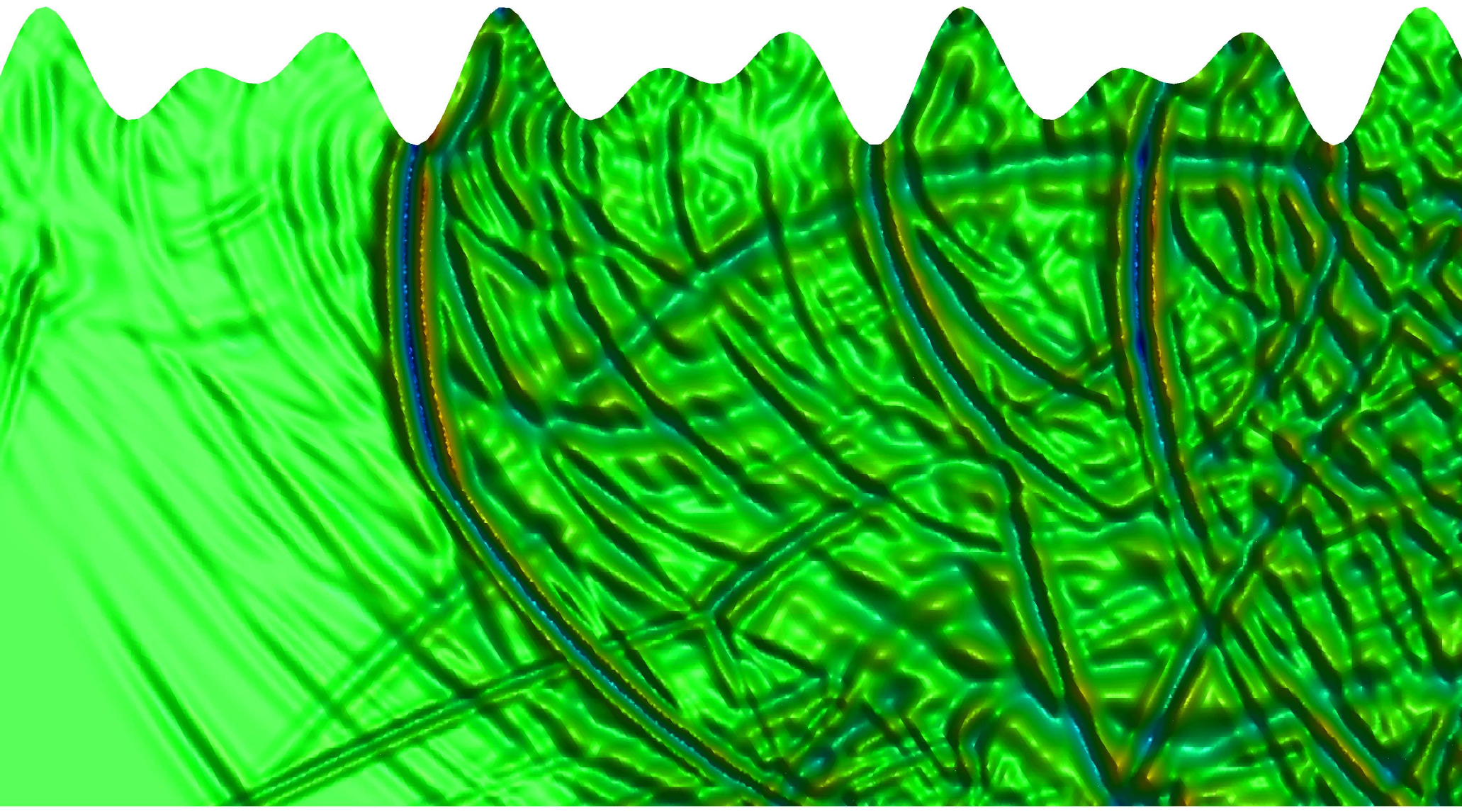} \\
\caption{Wave propagation in complex geometry. Isocontours of the vertical velocity $v$ at times $t=0.3, 0.7, 1.1$ from top to bottom for the the reference solution (left) and the 
new staggered space-time DG scheme (right).}%
\label{fig.NT4.2}%
\end{figure*}
\begin{figure*}%
\includegraphics[width=0.9\columnwidth]{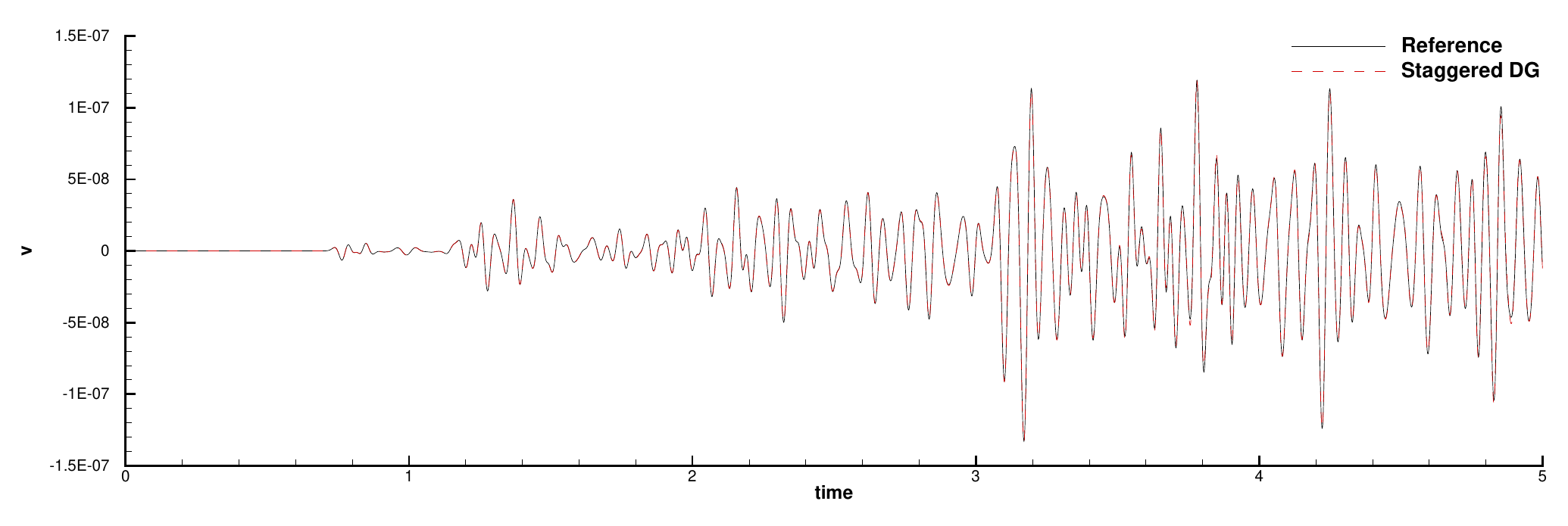} \\ 
\includegraphics[width=0.9\columnwidth]{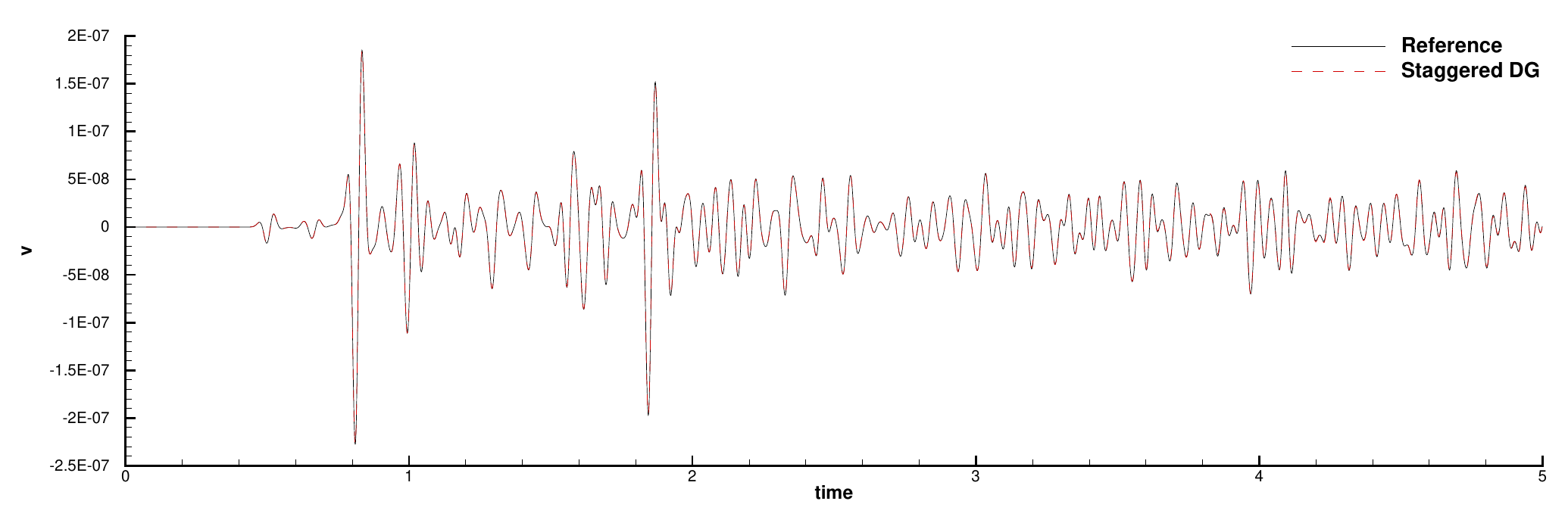} \\
\includegraphics[width=0.9\columnwidth]{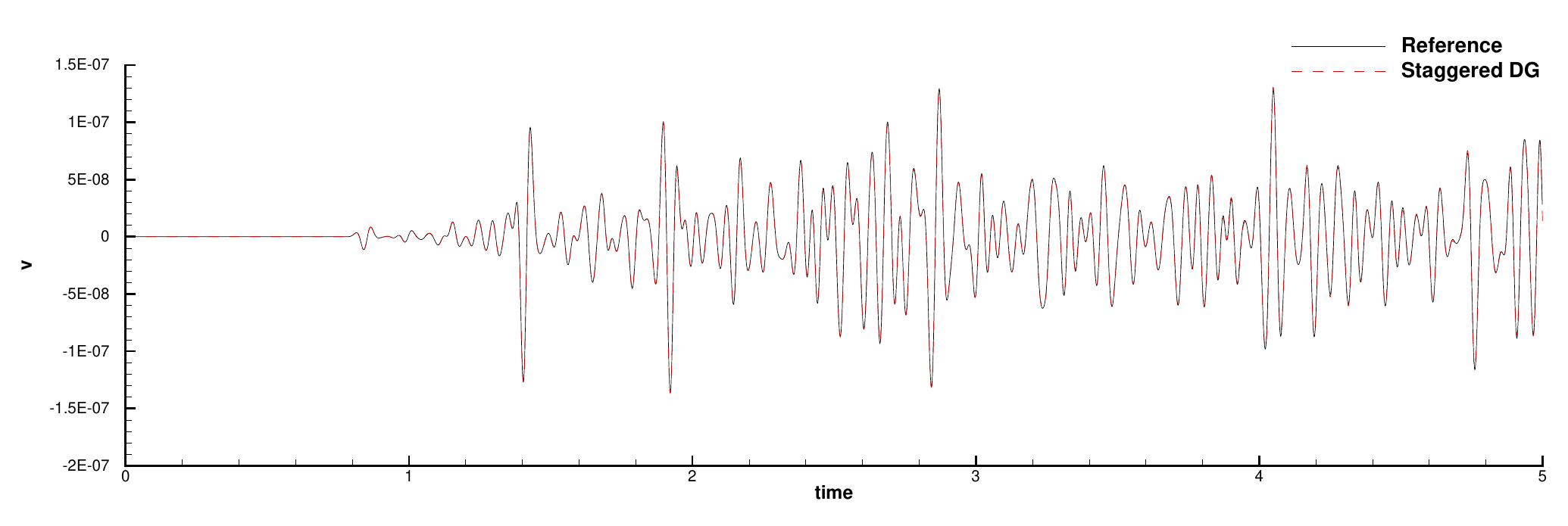} 
\caption{Comparison between numerical and reference time series for the velocity component $v$ in the three receivers $1\ldots3$ from top to bottom. }%
\label{fig.NT4.3}%
\end{figure*}

\subsection{Sliver element problem} 
Since in unstructured meshes for complex geometries or Cartesian cut cell approaches one can easily generate so-called \textit{sliver elements}, we want to test our new approach in the case 
where we have sliver elements in the computational domain, see \cite{gij5} for a similar study in the context of explicit ADER-DG schemes with time-accurate local time stepping (LTS). 
We will compare the number of iterations needed to solve the linear system in the case of a regular unstructured grid and the mesh containing the sliver elements. 
Since the resulting matrix for the velocity field becomes locally ill-conditioned, we will use here a couple of preconditioners in order to control the number of iterations. 
The simplest one (Pre1) consists in inverting only the diagonal block of the system matrix, while the second one (Pre2) requires to invert a local system composed of the element 
and its direct face neighbors. More details about the construction of those preconditioners are reported in \ref{App_pre}. We consider a computational domain 
$\Omega=[-1.5,1.5]^2$ covered with an almost uniform grid (mesh 1) and the same grid with two strongly deformed sliver elements (mesh 2), see Figure \ref{fig.NT5.1}. The incircle 
radius corresponding to the sliver elements in mesh 2 is reduced by a factor of $70.53$ with respect to mesh 1. We use the same setup as presented in Section \ref{sec.convtest}  
for a simple $p$-wave traveling in direction $\vec{n}=(1,0)$ and we use $(p,p_\gamma)=(4,2)$ with a time step size of {$\Delta t = 0.014$ for both meshes. This is
possible since our staggered space-time DG scheme is \textit{unconditionally stable}}. Figure \ref{fig.NT5.2} shows the numerical solutions obtained on the two different meshes. 
One can observe that the introduction of the sliver element in mesh 2 does not change the quality of the solution, but of course it changes the effort required to solve the 
linear system for the velocity. The mean number of iterations needed to solve the system is reported in Table \ref{tab:NT5.1}. The trend of the iterations in the different cases 
is shown in Figure \ref{fig.NT5.3}. As we can easily see, if we do not use any kind of preconditioner, the average number of iterations increases a lot. The use of the fully 
local preconditioner $1$ helps to reduce the number of iterations, while the second preconditioner is sufficient to solve this ill-conditioning problem and to keep the number of
iterations almost independent of the mesh.

\begin{figure*}%
\includegraphics[width=0.49\columnwidth]{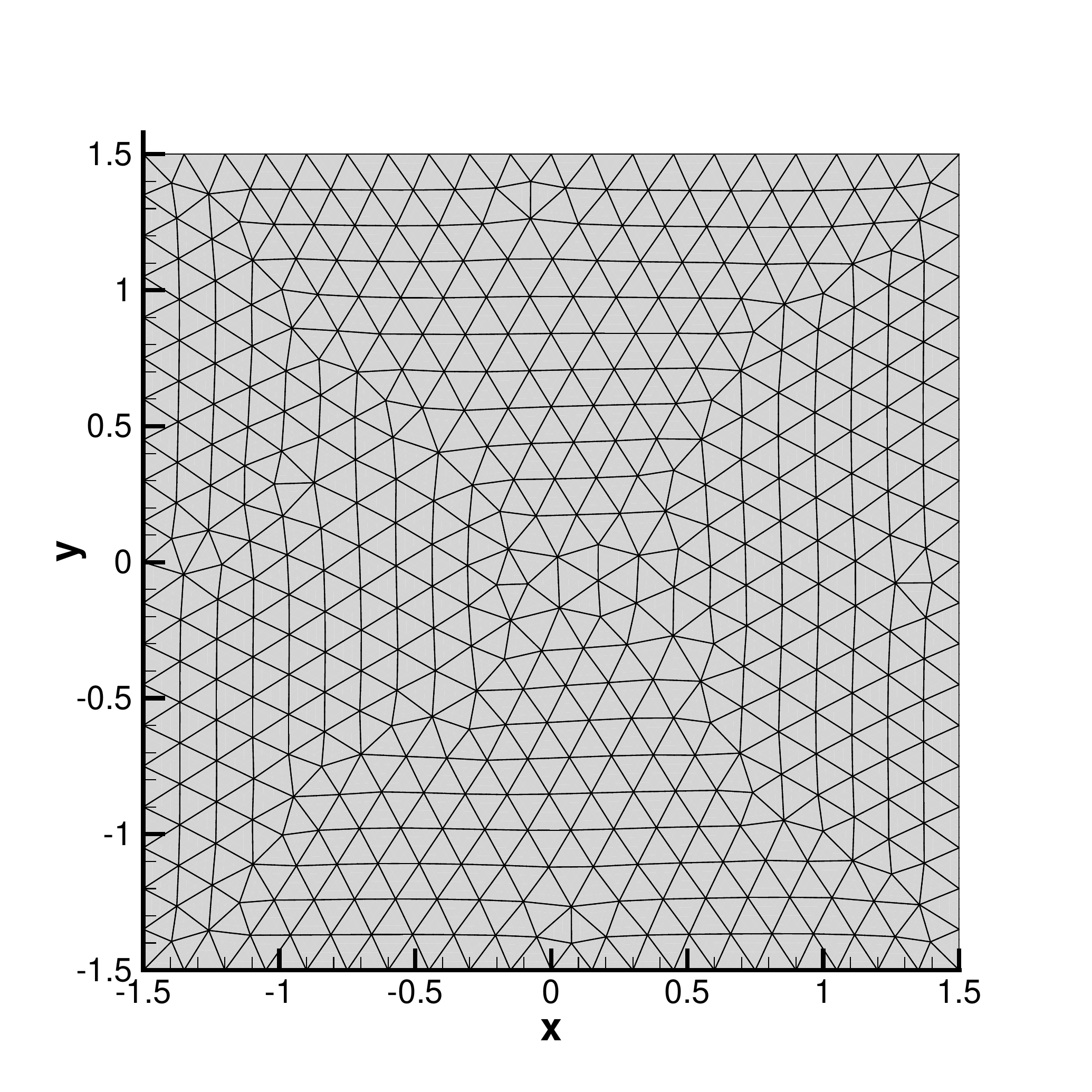} 
\includegraphics[width=0.49\columnwidth]{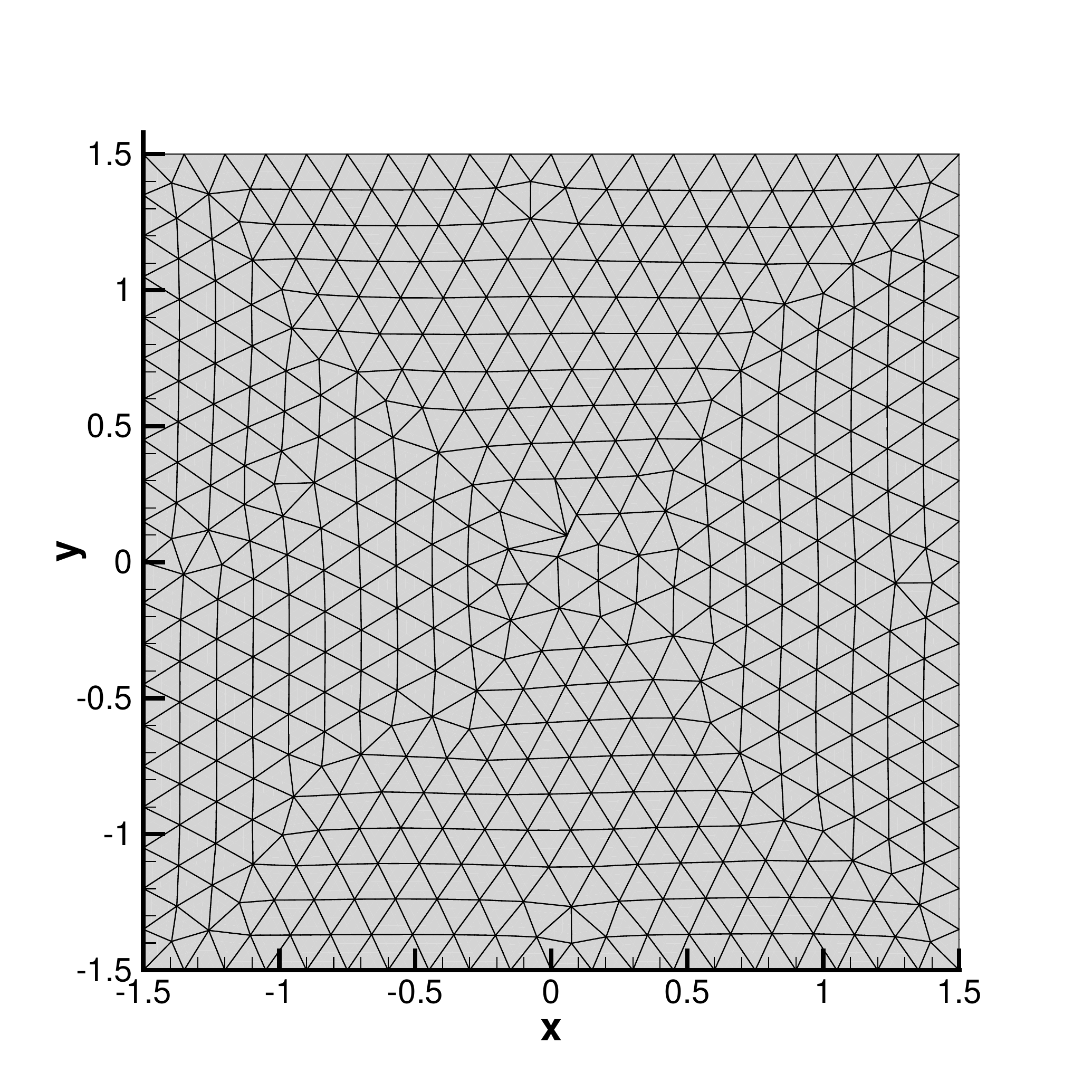}
\caption{Almost uniform mesh 1 (left) and mesh 2 containing two sliver elements (right).}%
\label{fig.NT5.1}%
\end{figure*}

\begin{figure*}%
\includegraphics[width=0.49\columnwidth]{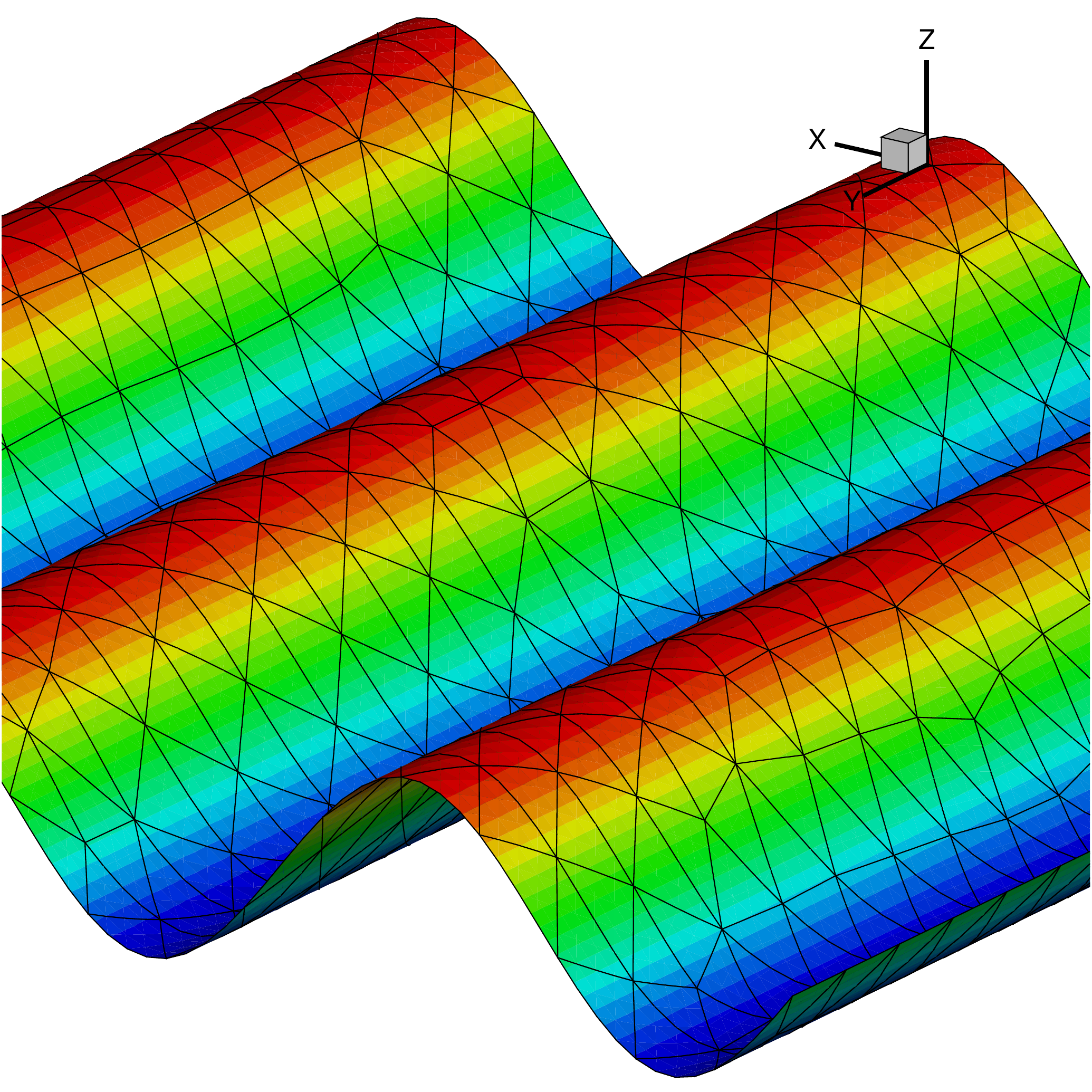} 
\includegraphics[width=0.49\columnwidth]{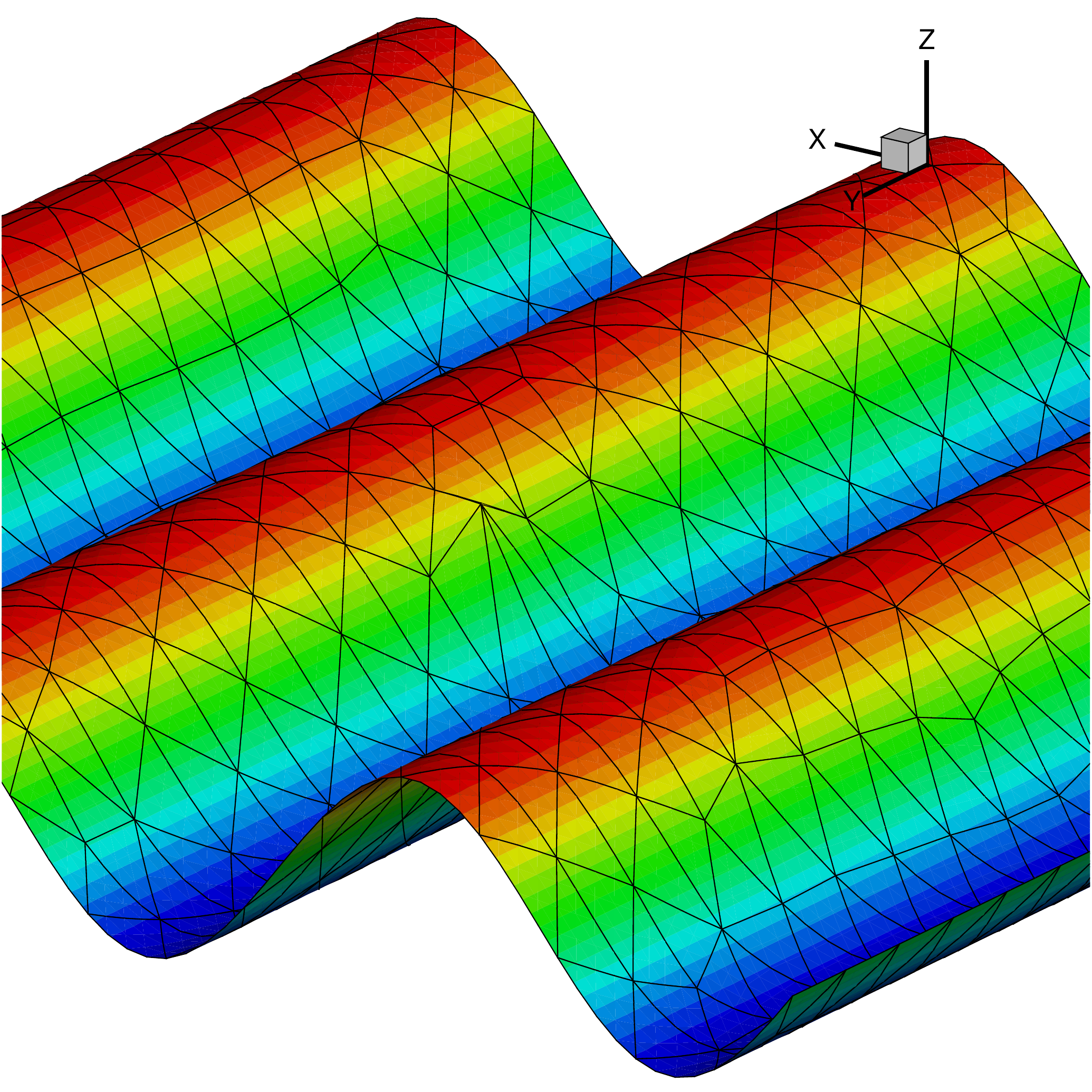}
\caption{Numerical results for the velocity component $u$ using the regular unstructured mesh 1 (left) and the unstructured mesh 2 containing the sliver elements (right). It can be clearly noted that also
on mesh 2 the solution is smooth and is not affected by the presence of the slivers. }%
\label{fig.NT5.2}%
\end{figure*}

\begin{table}%
\begin{center}
\begin{tabular}{cccc}
	Preconditioning & Iter. Mesh 1 	& Iter. Mesh 2 & Factor	 \\
	\hline
	None		&	112.59	&  611.95 & {\small 5.43} \\
	Pre 1		&	86.73	&  191.77 & {\small 2.21} \\
	Pre 2		&	53.27	&  53.38 & {\small 1.00} \\
	\hline
\end{tabular}
\end{center}
\caption{Number of average iterations needed for the GMRES algorithm with different preconditioners on the uniform unstructured grid (mesh 1) and the one containing the sliver elements (mesh 2).}
\label{tab:NT5.1}
\end{table}
\begin{figure*}%
\includegraphics[width=0.33\columnwidth]{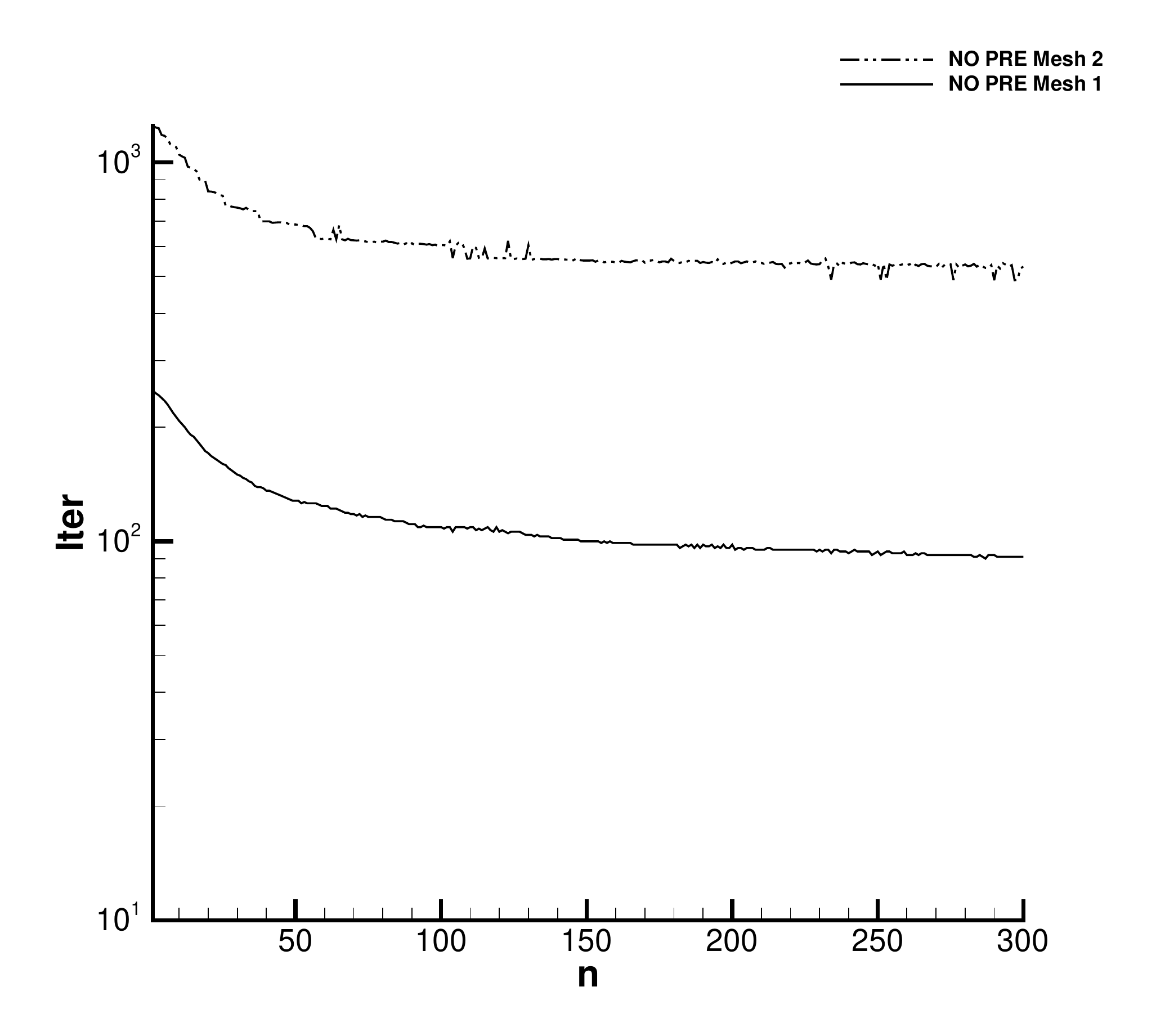} 
\includegraphics[width=0.33\columnwidth]{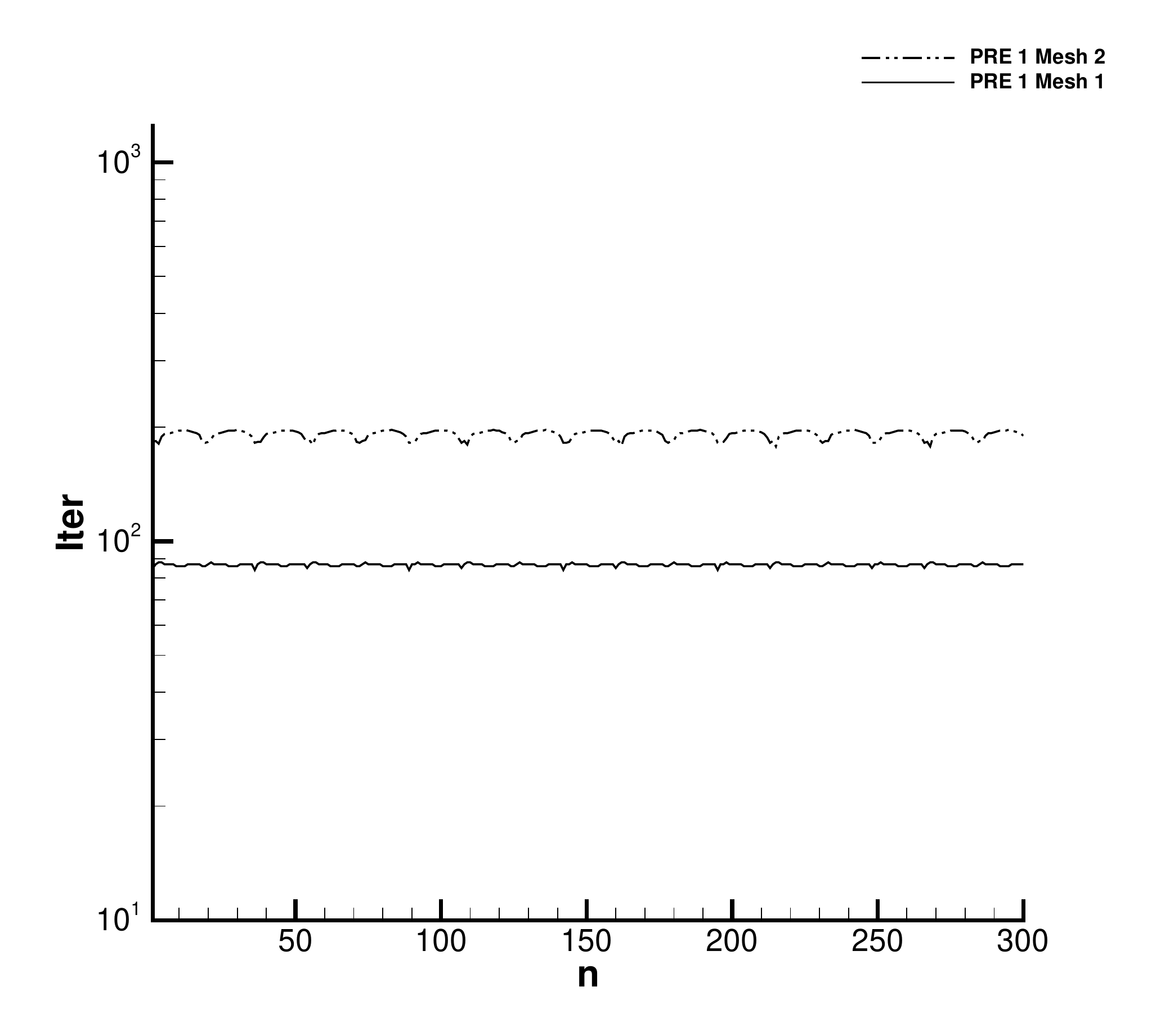}
\includegraphics[width=0.33\columnwidth]{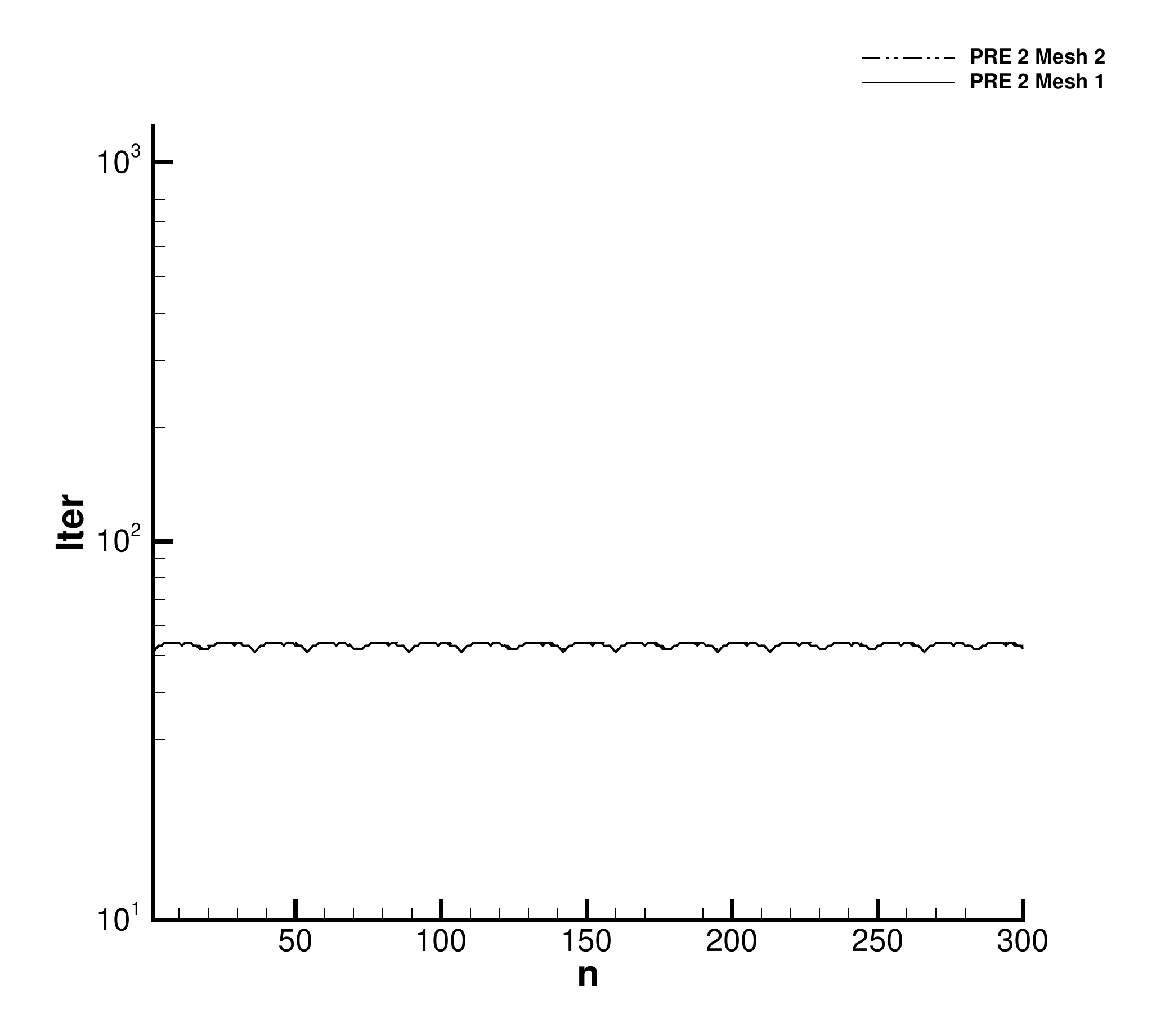}
\caption{Required number of iterations for the solution of the linear system plotted over the time step index $n$ in the case of no preconditioning (left), using the preconditioner 1 (center) and the preconditioner 2 (right).} 
\label{fig.NT5.3}%
\end{figure*}
\subsection{3D wave propagation} 
In this test case we want to check our numerical method in three space dimensions. We take a very simple material block of size $\Omega=[0,10000]\times [-8000,2000] \times [-5000,5000]$. We use a  homogeneous material with $c_p=3200$, $c_s=1847.5$ and $\rho=2200$. The resulting Lam\'e constants are $\lambda=7.51 \cdot 10^9$ and $\mu=7.51 \cdot 10^9$. The domain is covered with $\Ni=214893$  tetrahedral elements of average size {$h=388.55$}. For this test problem we use the particular case of the Crank-Nicolson time discretization ($p_\gamma=0$) and approximation degree $p=4$ 
in space. The wave is generated by an initial Gaussian profile imposed in the velocity component $w$ as  
\begin{eqnarray}
	w(\xx,0) = a e^{-r^2/R^2} 
\label{eq:NTCG3D_1a}
\end{eqnarray}
with $a=-10^{-2}$, $R=100$ and $r=|\xx-\mathbf{x}_0|$ is the distance from the center point $\xx_0=(5000,1900,0)$. All other state variables are initialized with zero. 
We place two receivers in $\Omega$, one close to the free surface at $\xx_1=(6000,1999,500)$ and the second one $500$ units below the free surface in $\xx_2=(6000,1500,500)$. 
A comparison of the velocity component $v$ obtained with the ADER-DG reference code \texttt{SeisSol}  and the new staggered DG scheme proposed in this paper is shown in Figure \ref{fig.NT3DI_1}, where 
we also show the location of the two receivers. {For the computation of the reference solution, we use the same computational mesh and the same order of accuracy, i.e. we use $N=M=4$ 
and $N_i=214893$.} In Figures \ref{fig.NT3DI_2} and \ref{fig.NT3DI_3} we present a comparison between of the time signal recorded in the two receivers with the two different schemes. We can observe 
a very good agreement between the ADER-DG reference solution and the numerical  solution obtained with the new staggered DG scheme. We can also observe that the stress components corresponding to the 
$y$ direction vanish at the free surface, as reported in Figure \ref{fig.NT3DI_3}. 
\begin{figure*}%
\includegraphics[width=0.49\columnwidth]{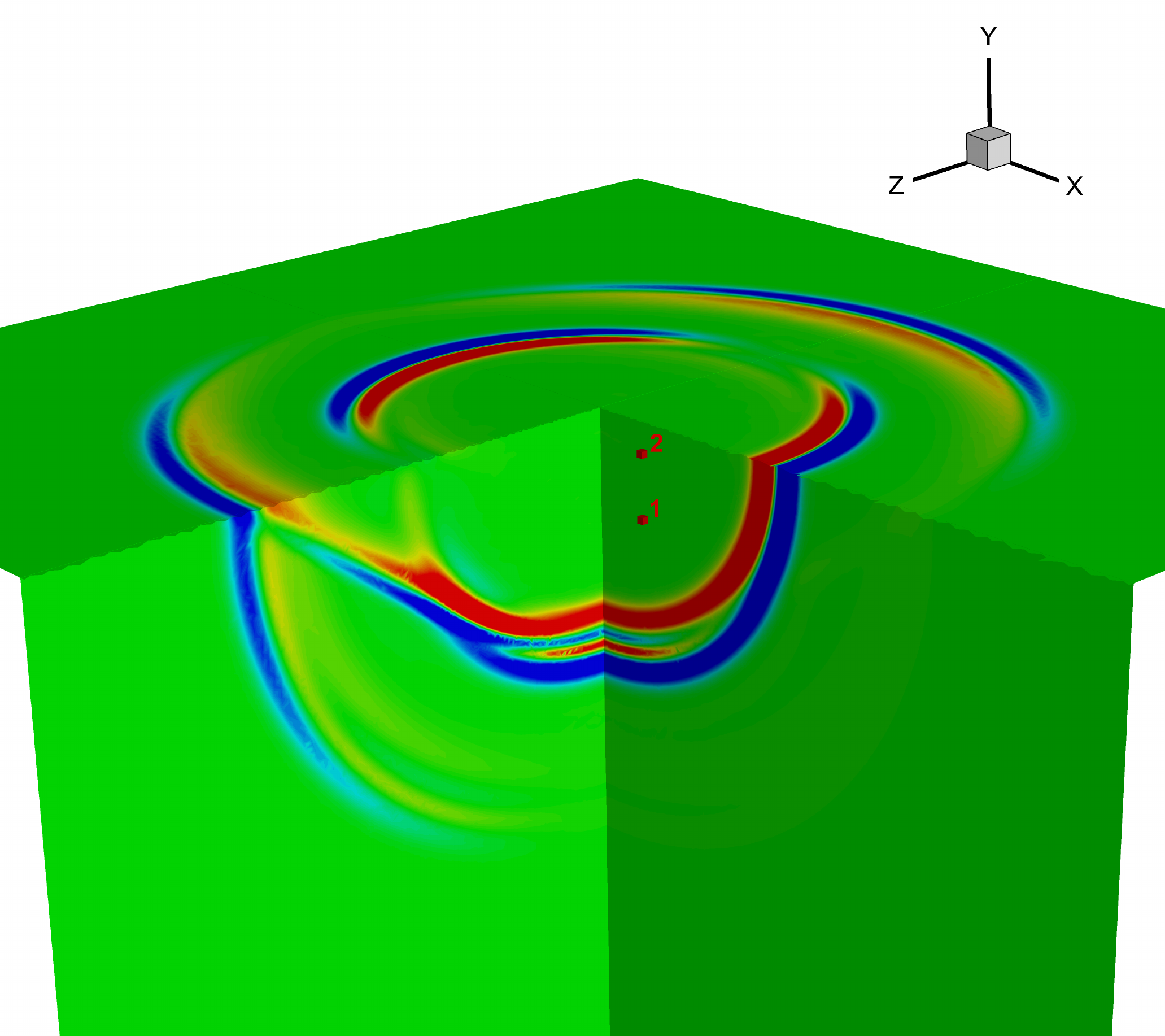} 
\includegraphics[width=0.49\columnwidth]{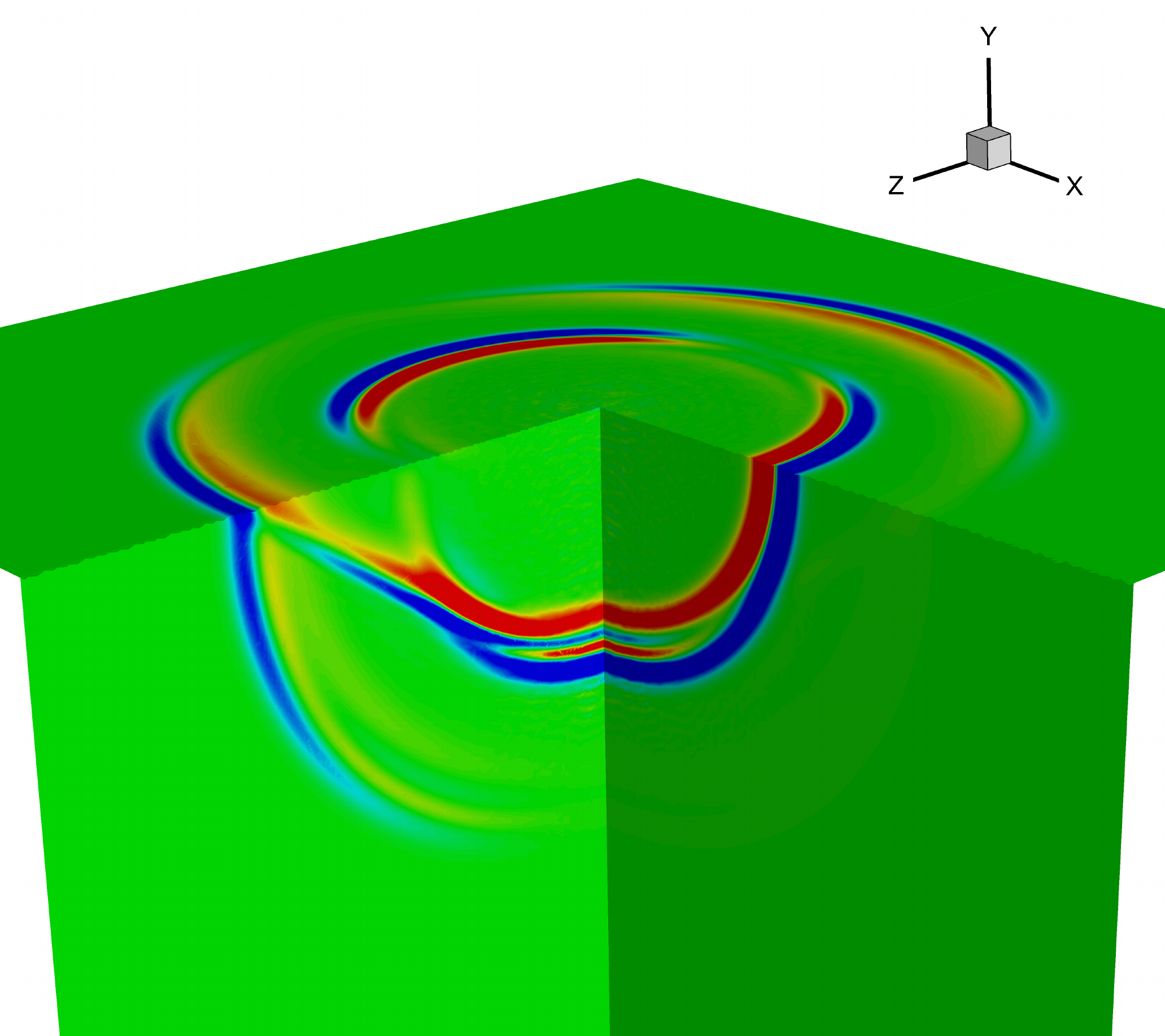}  
\caption{Simple 3D wave propagation problem. Numerical solution obtained for the {velocity component $w$ at time $t=1.0$} using an explicit ADER-DG reference scheme (left) and the new implicit staggered DG approach presented in this paper (right).}%
\label{fig.NT3DI_1}%
\end{figure*}
\begin{figure*}%
\includegraphics[width=0.33\columnwidth]{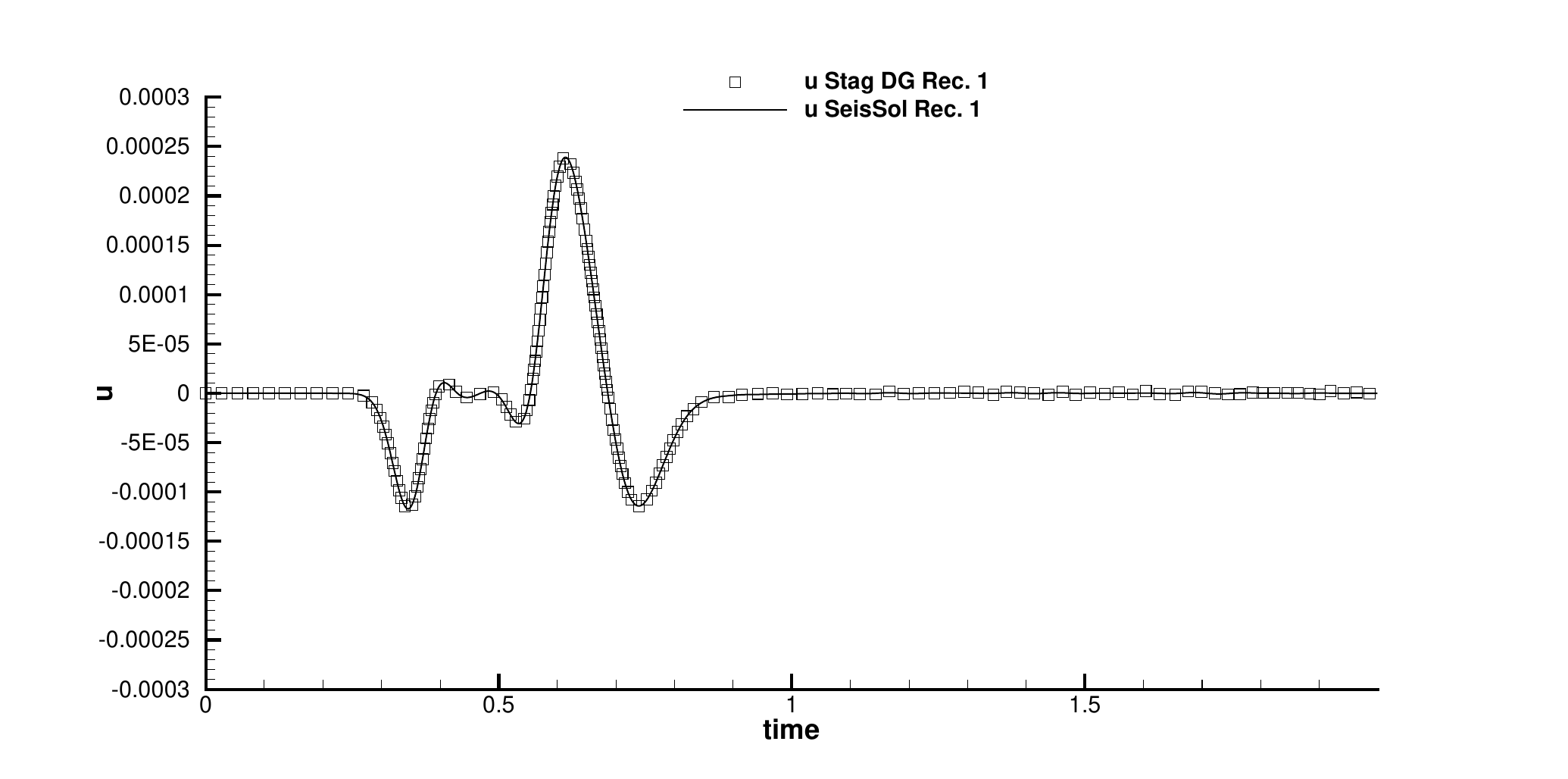} 
\includegraphics[width=0.33\columnwidth]{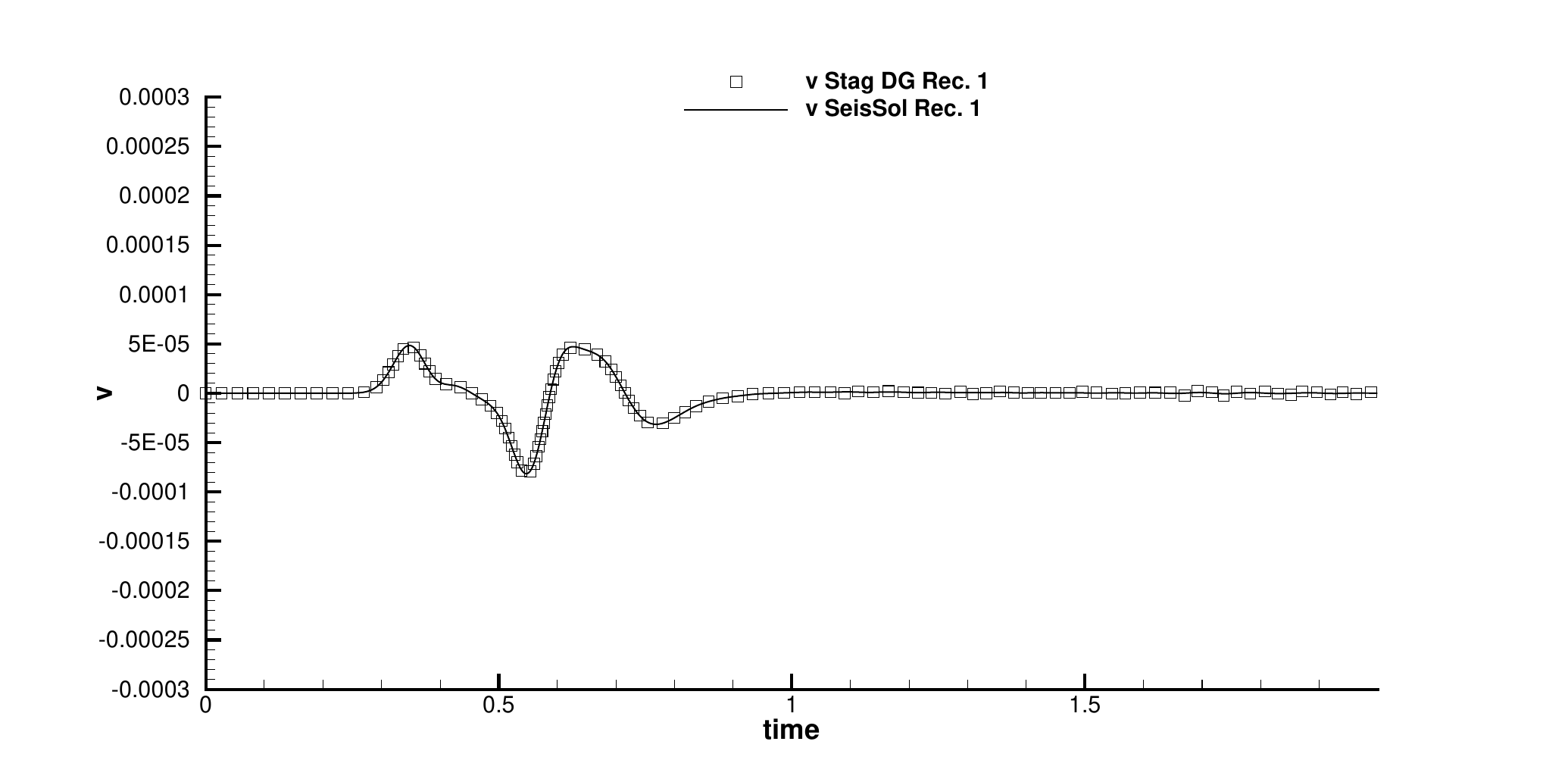} 
\includegraphics[width=0.33\columnwidth]{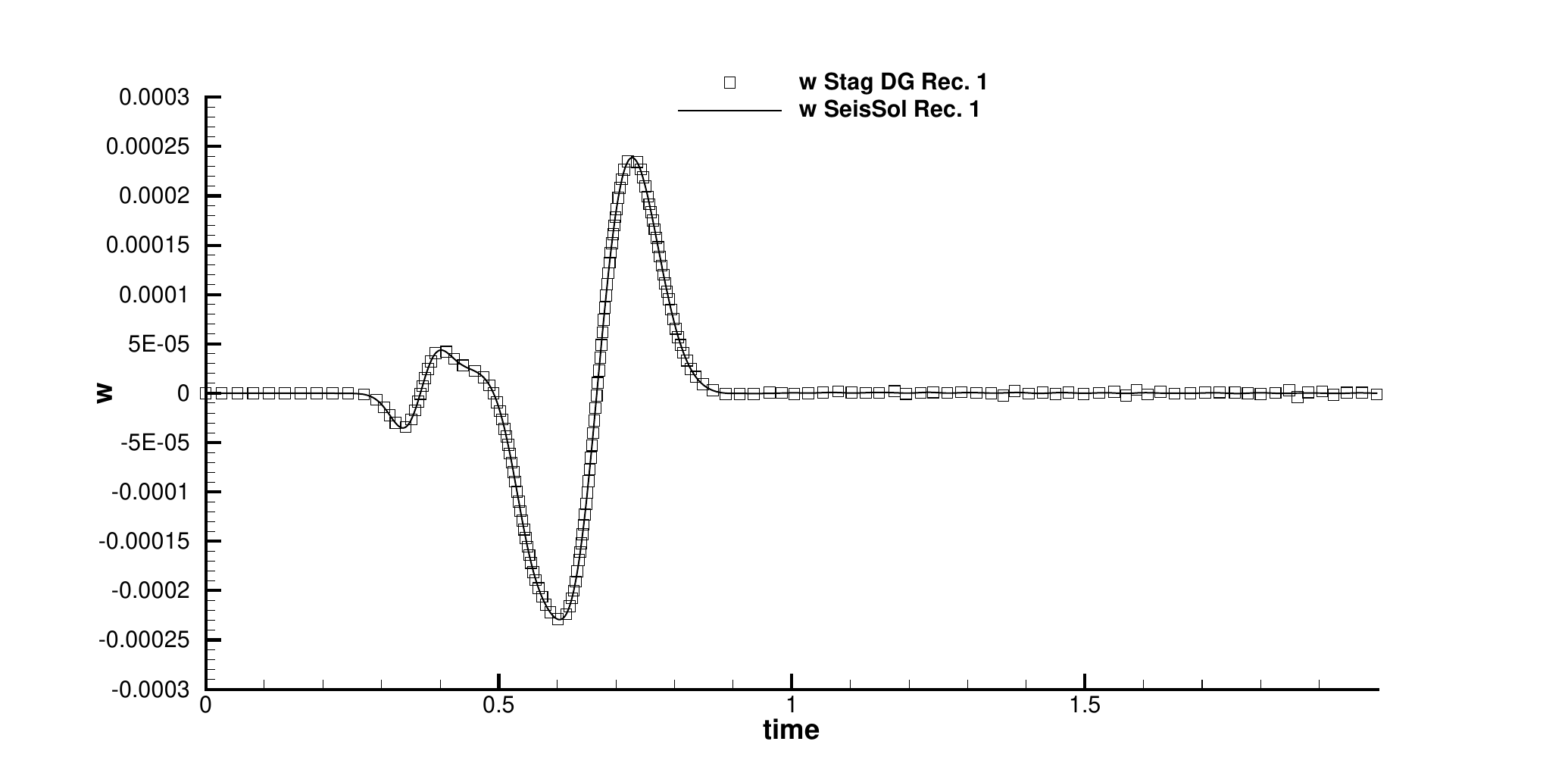} \\
\includegraphics[width=0.33\columnwidth]{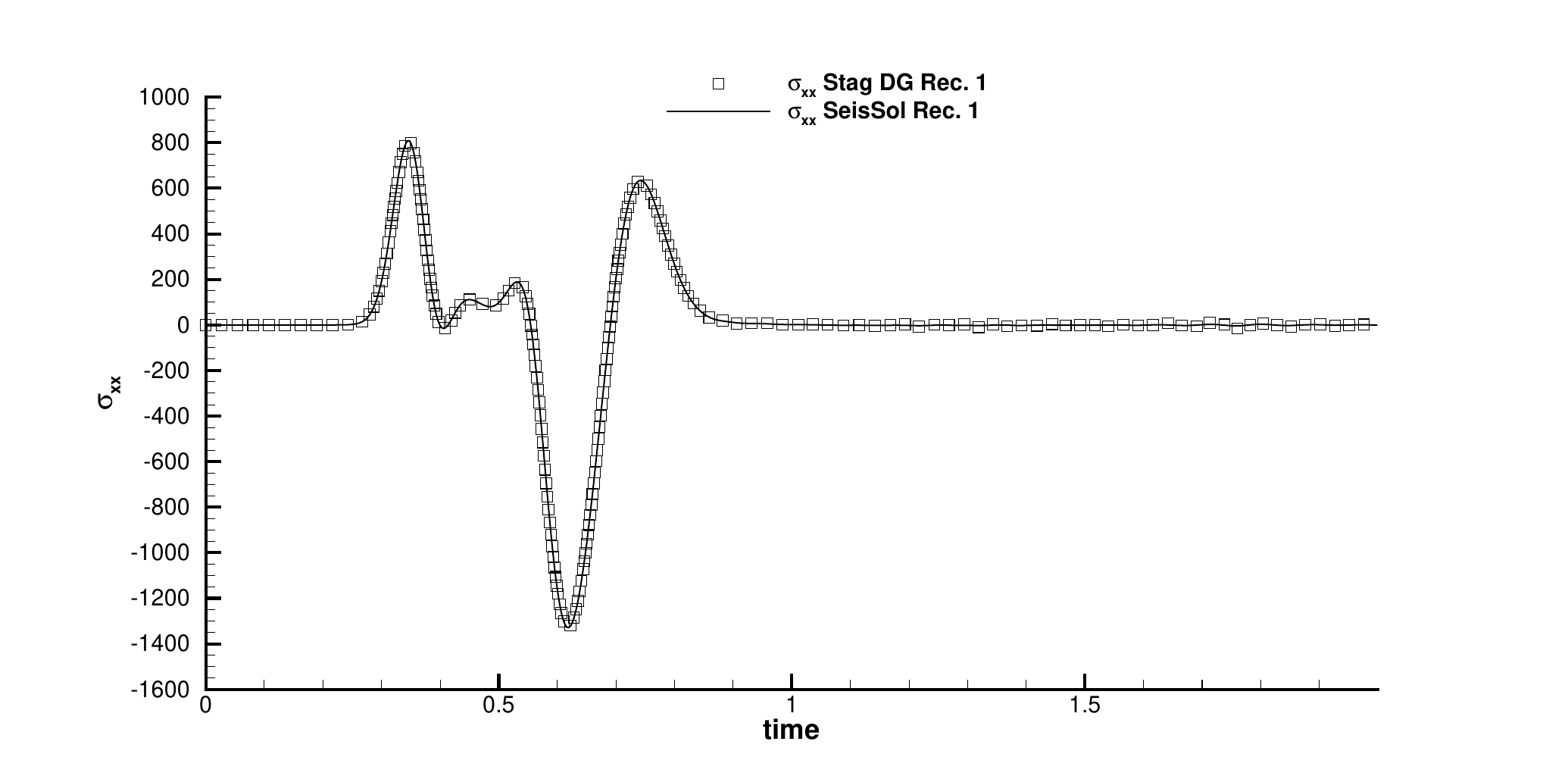} 
\includegraphics[width=0.33\columnwidth]{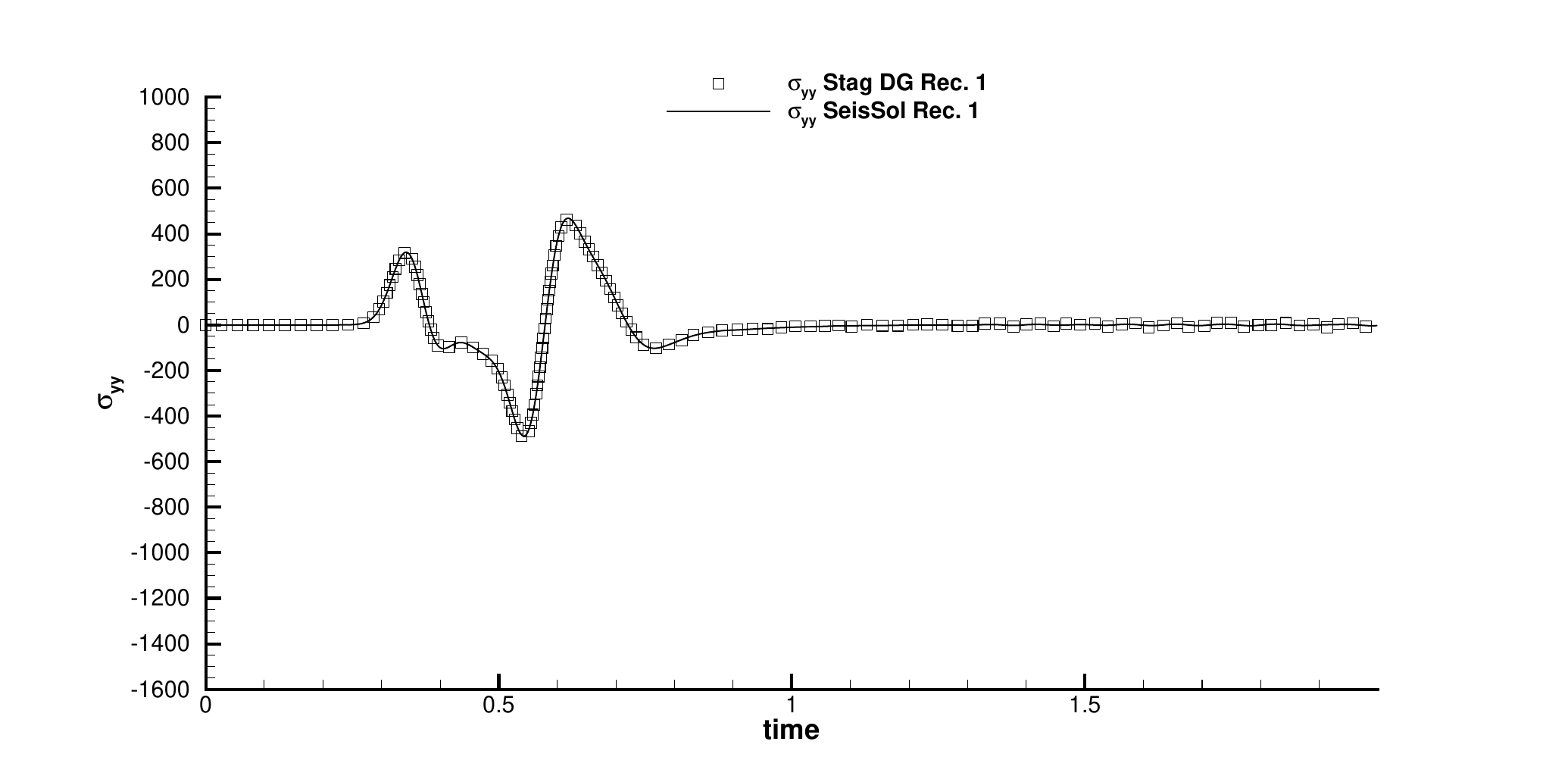} 
\includegraphics[width=0.33\columnwidth]{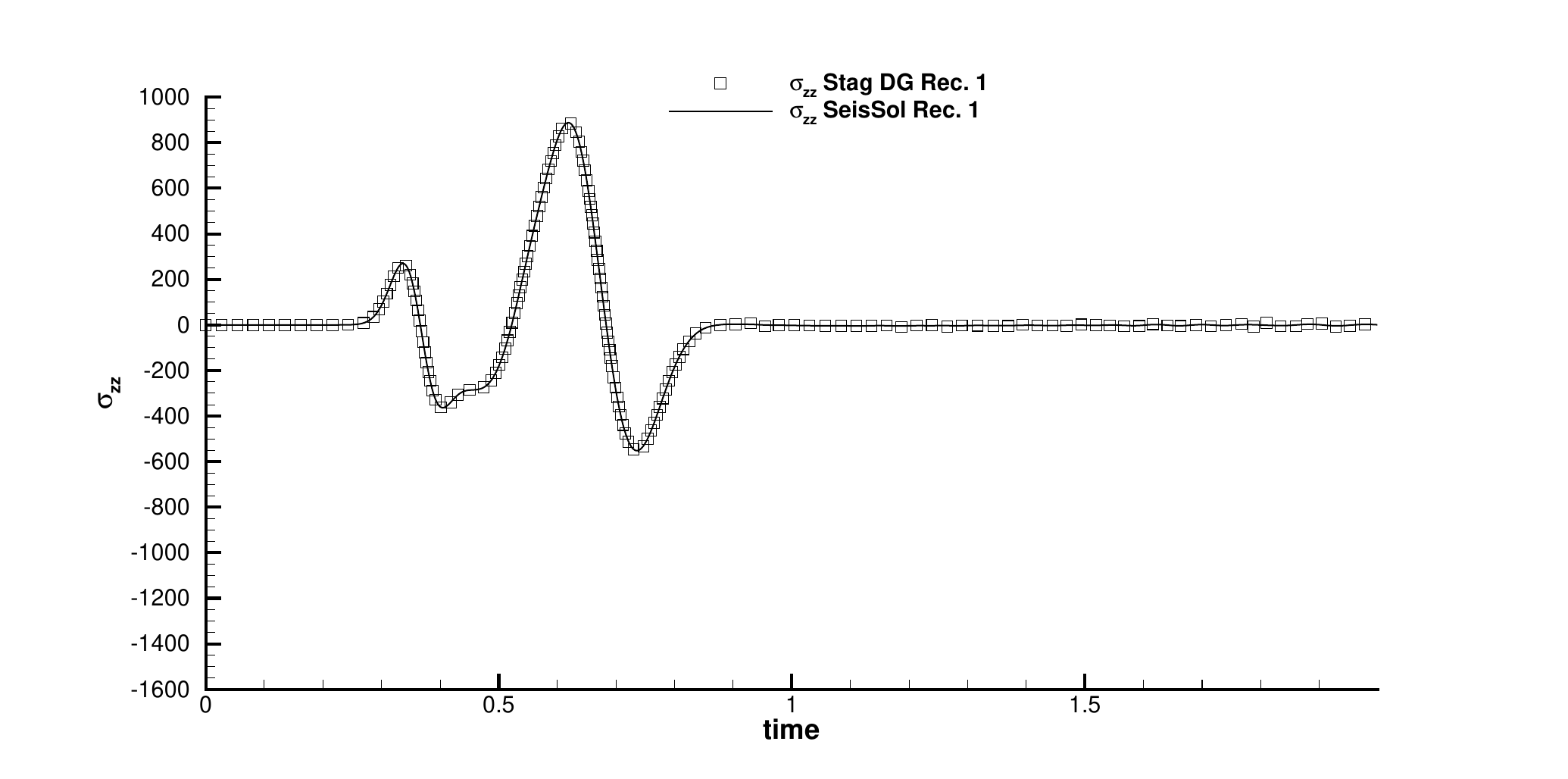} \\
\includegraphics[width=0.33\columnwidth]{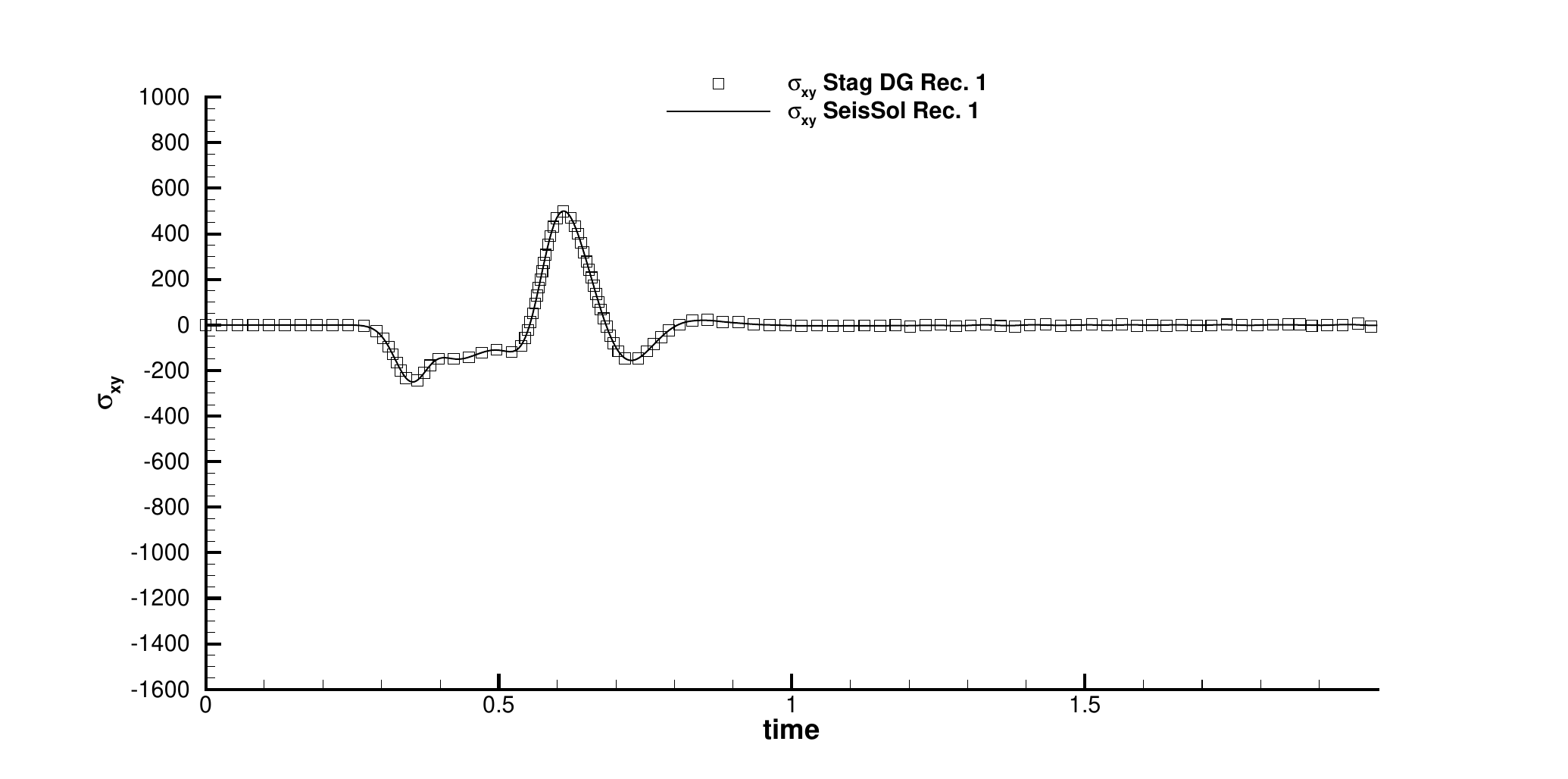} 
\includegraphics[width=0.33\columnwidth]{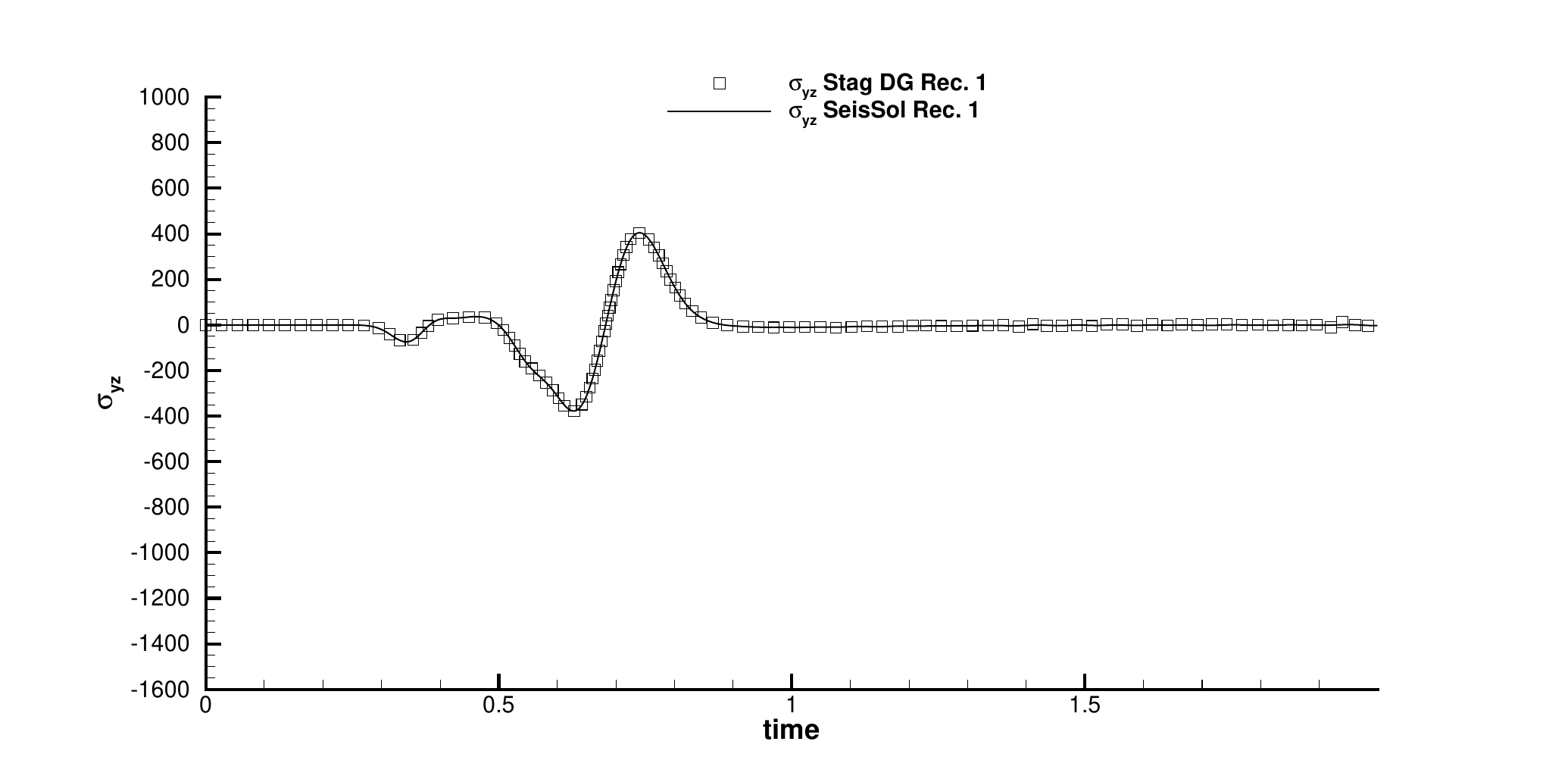} 
\includegraphics[width=0.33\columnwidth]{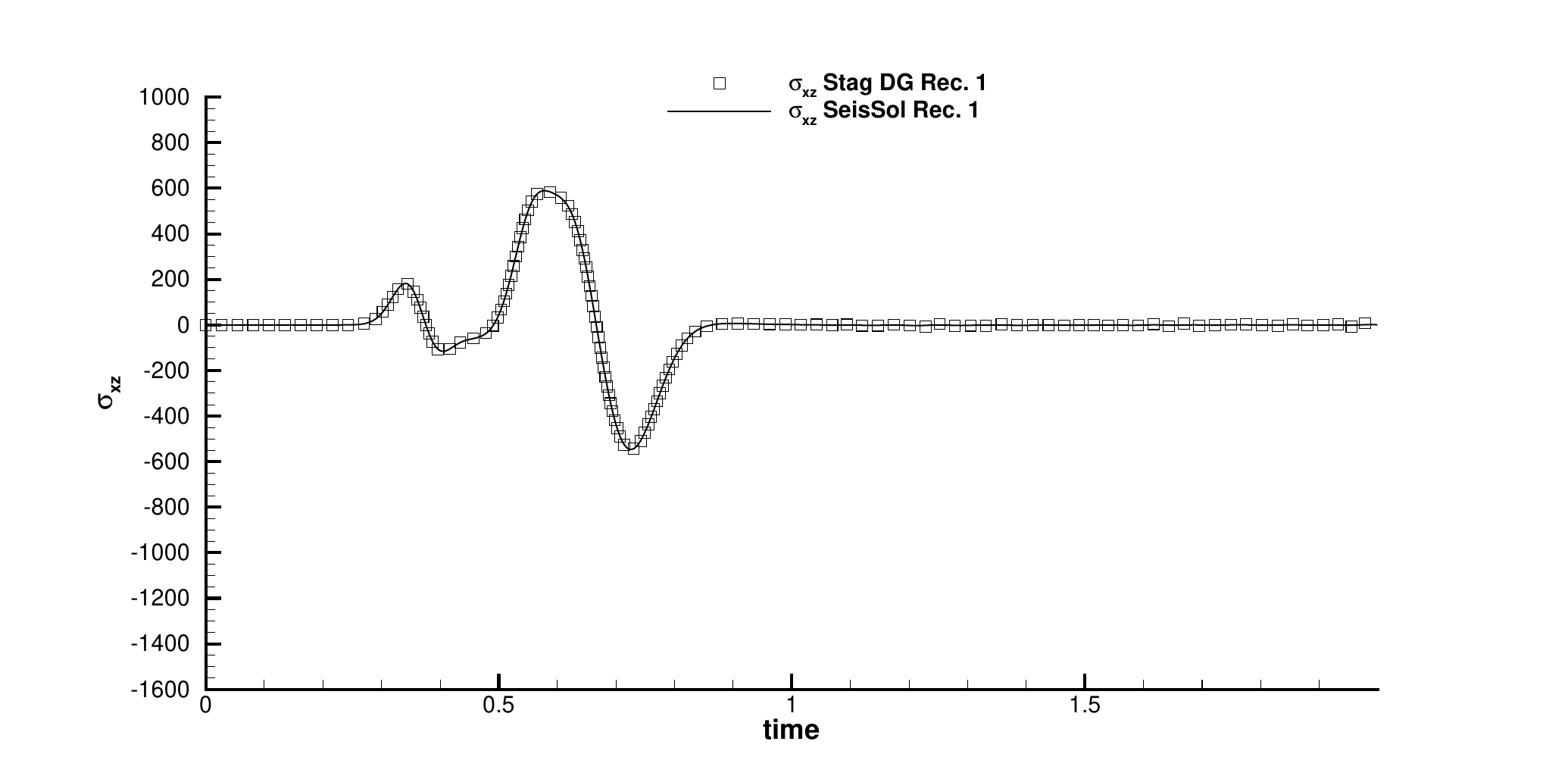}
\caption{Simple 3D wave propagation problem. Comparison of the numerical and reference solution in the first receiver, from top left to bottom right: $u,v,w,\sigma_{xx},\sigma_{yy},\sigma_{zz},\sigma_{xy},\sigma_{yz},\sigma_{xz}$.}%
\label{fig.NT3DI_2}%
\end{figure*}
\begin{figure*}%
\includegraphics[width=0.33\columnwidth]{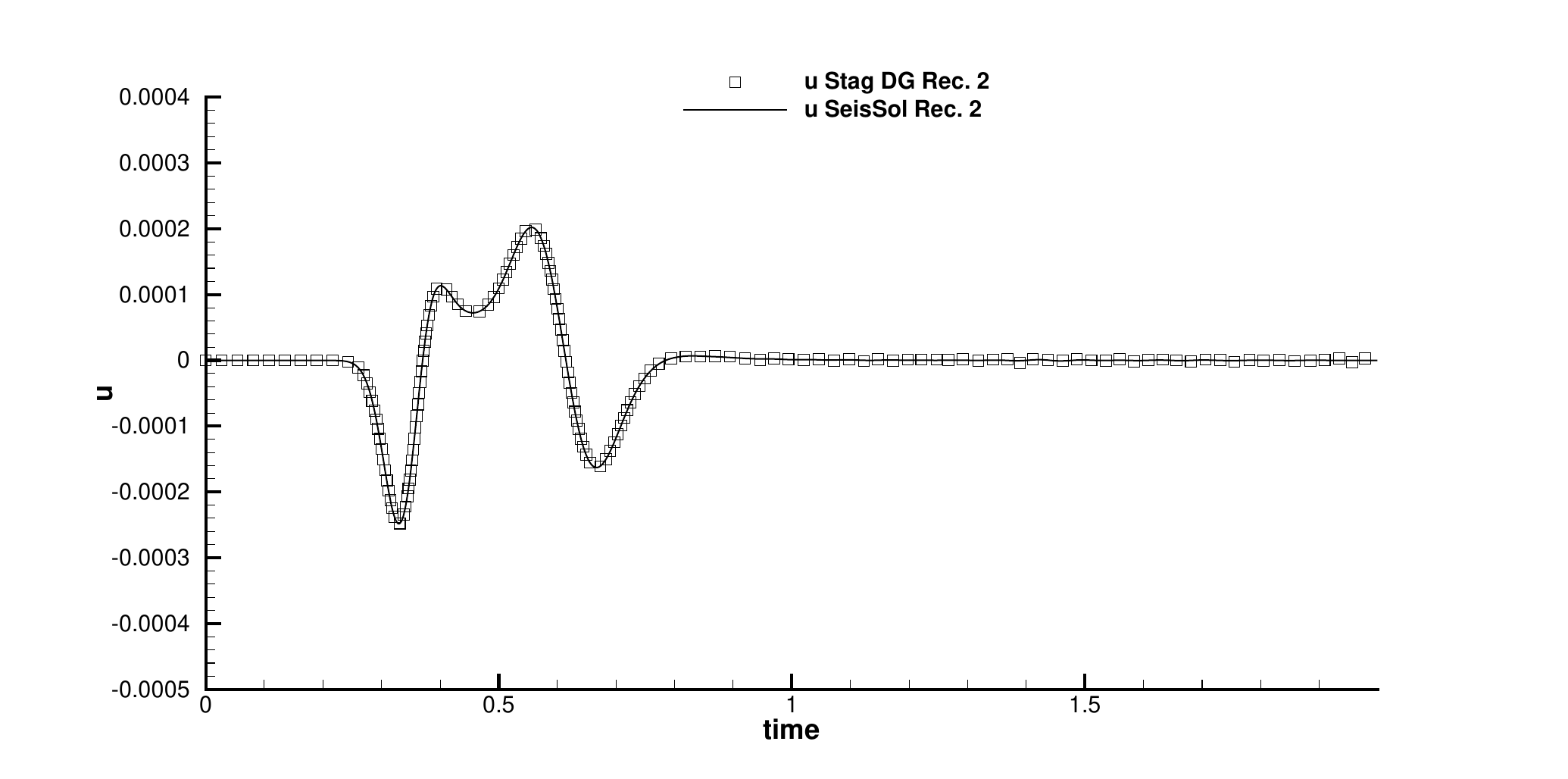} 
\includegraphics[width=0.33\columnwidth]{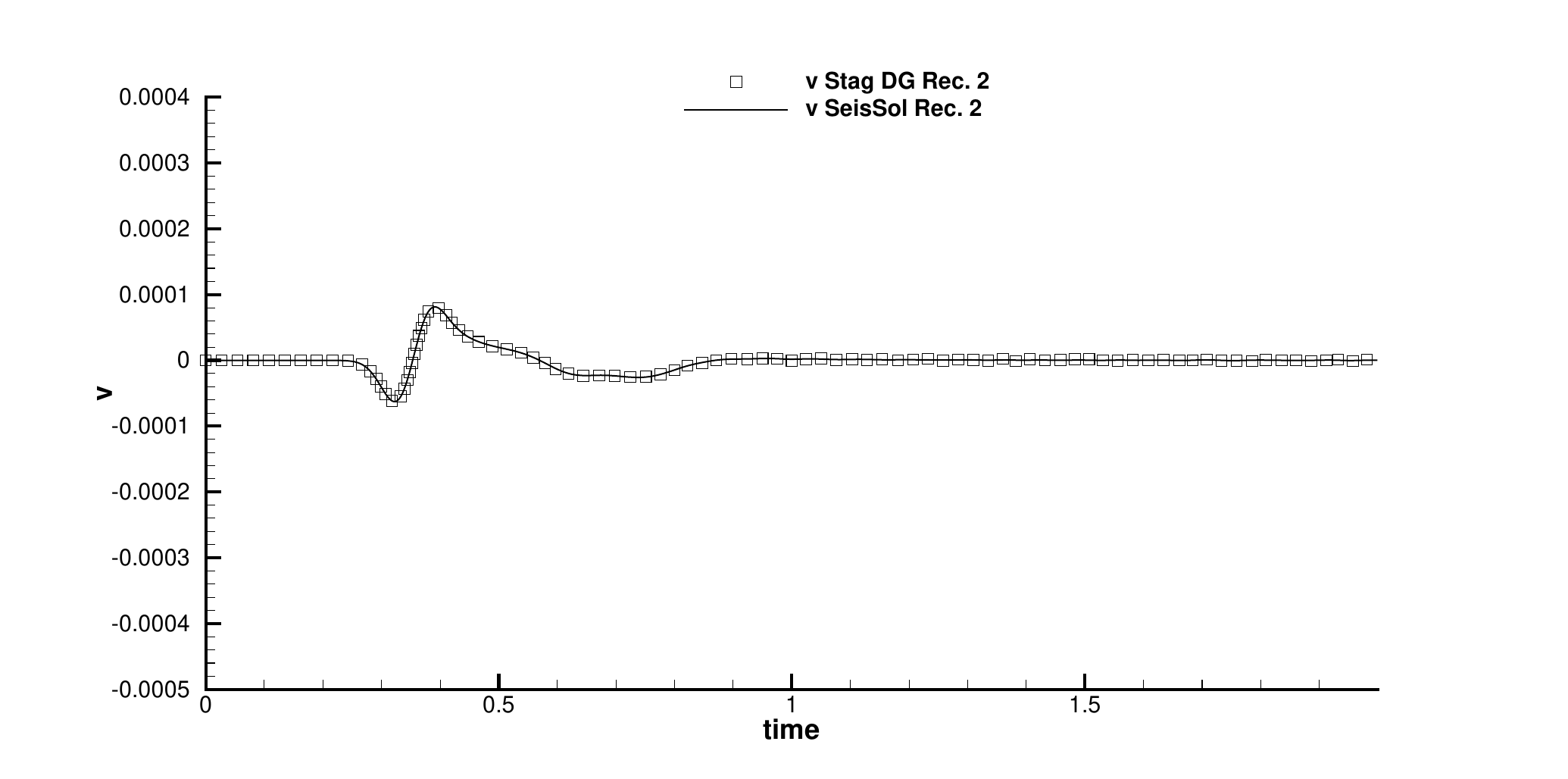} 
\includegraphics[width=0.33\columnwidth]{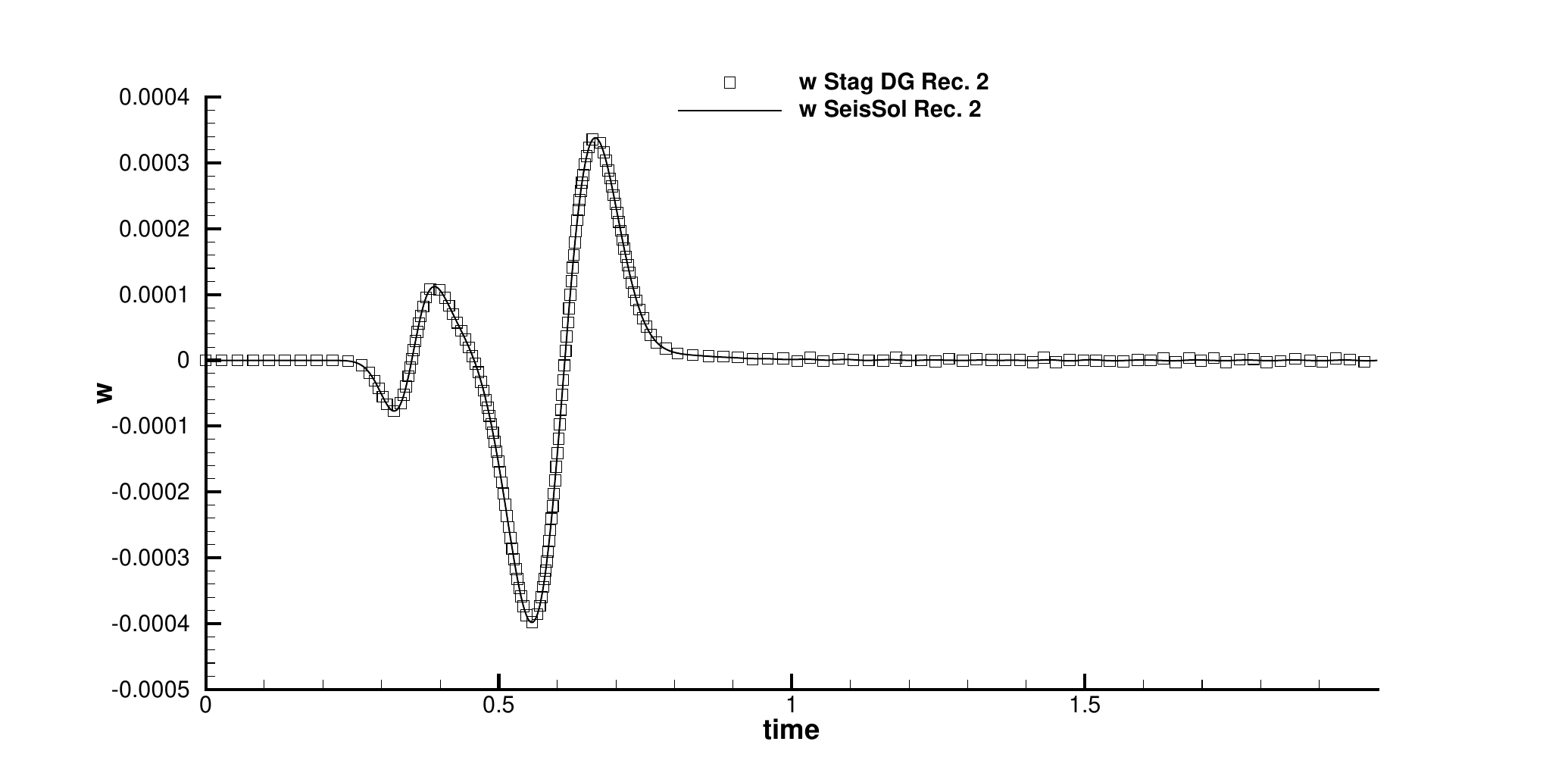} \\
\includegraphics[width=0.33\columnwidth]{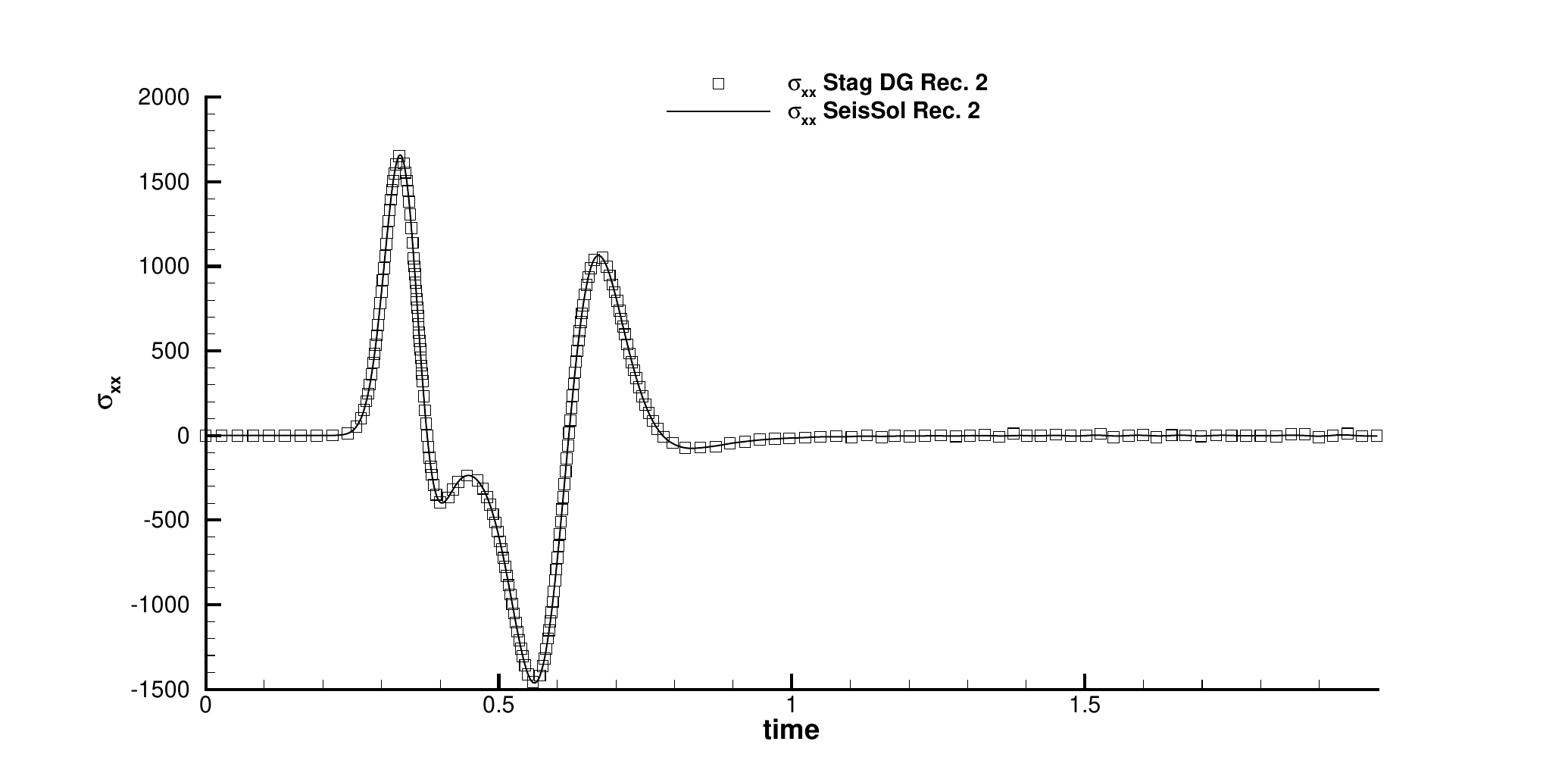} 
\includegraphics[width=0.33\columnwidth]{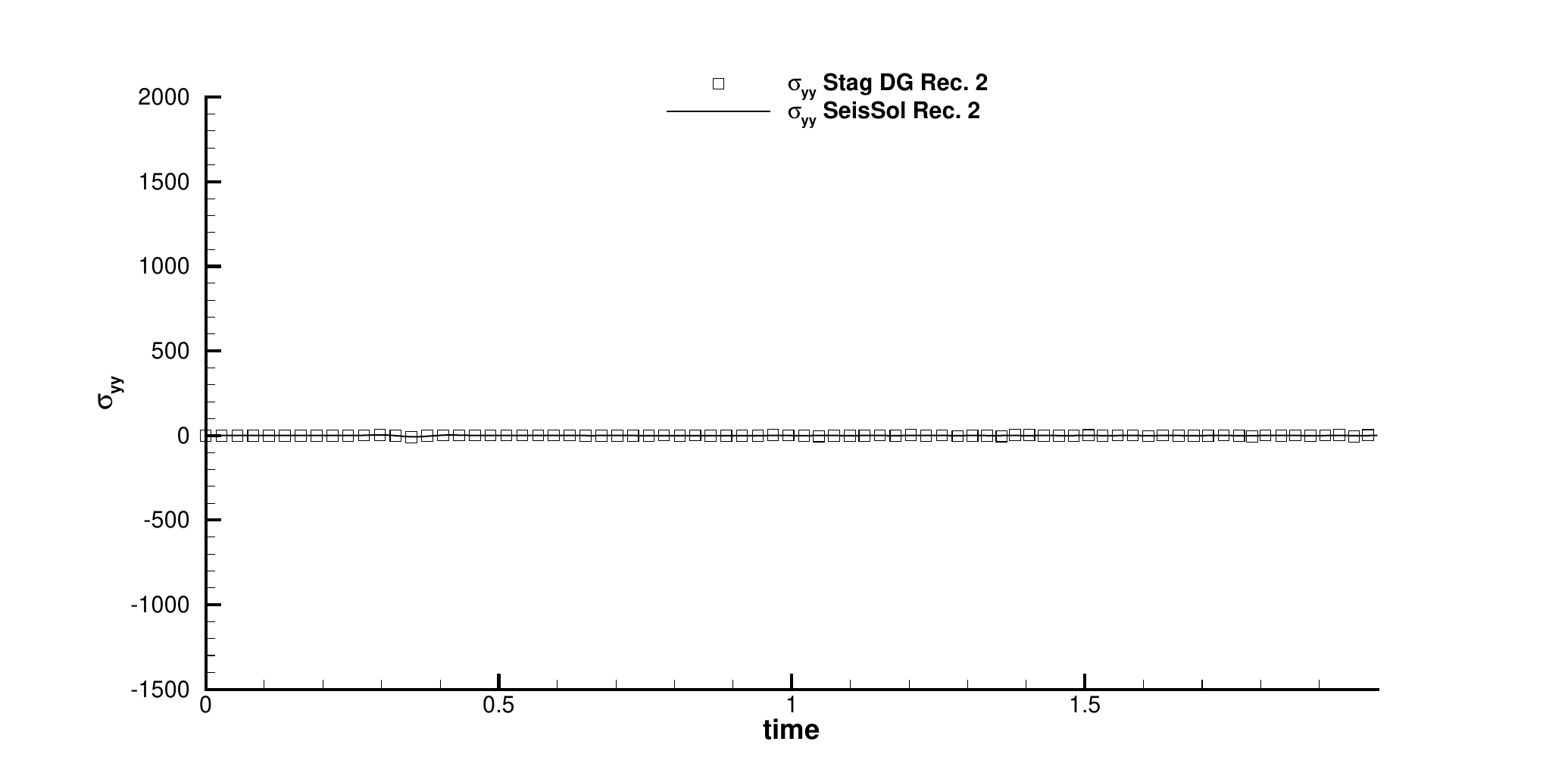} 
\includegraphics[width=0.33\columnwidth]{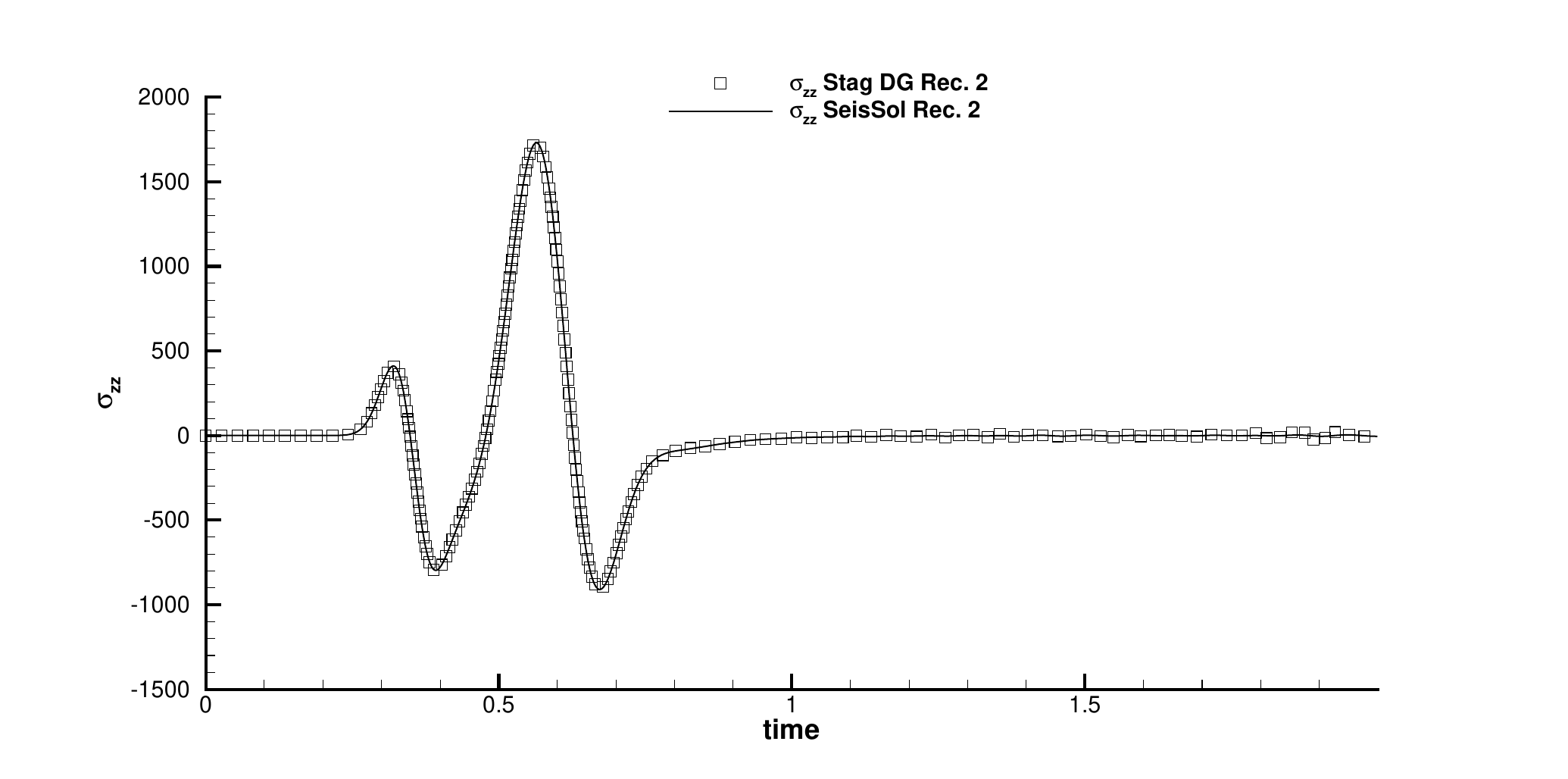} \\
\includegraphics[width=0.33\columnwidth]{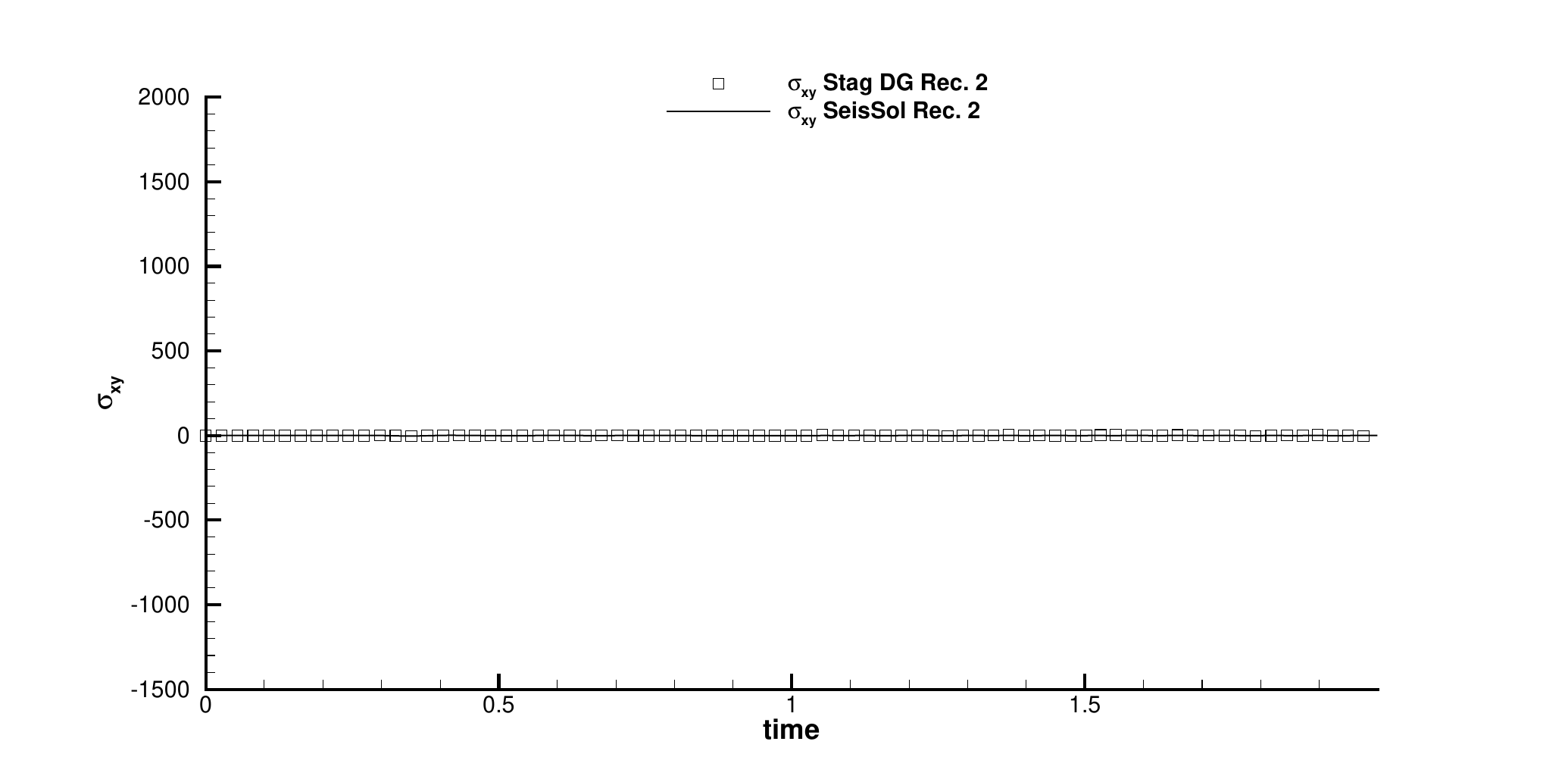} 
\includegraphics[width=0.33\columnwidth]{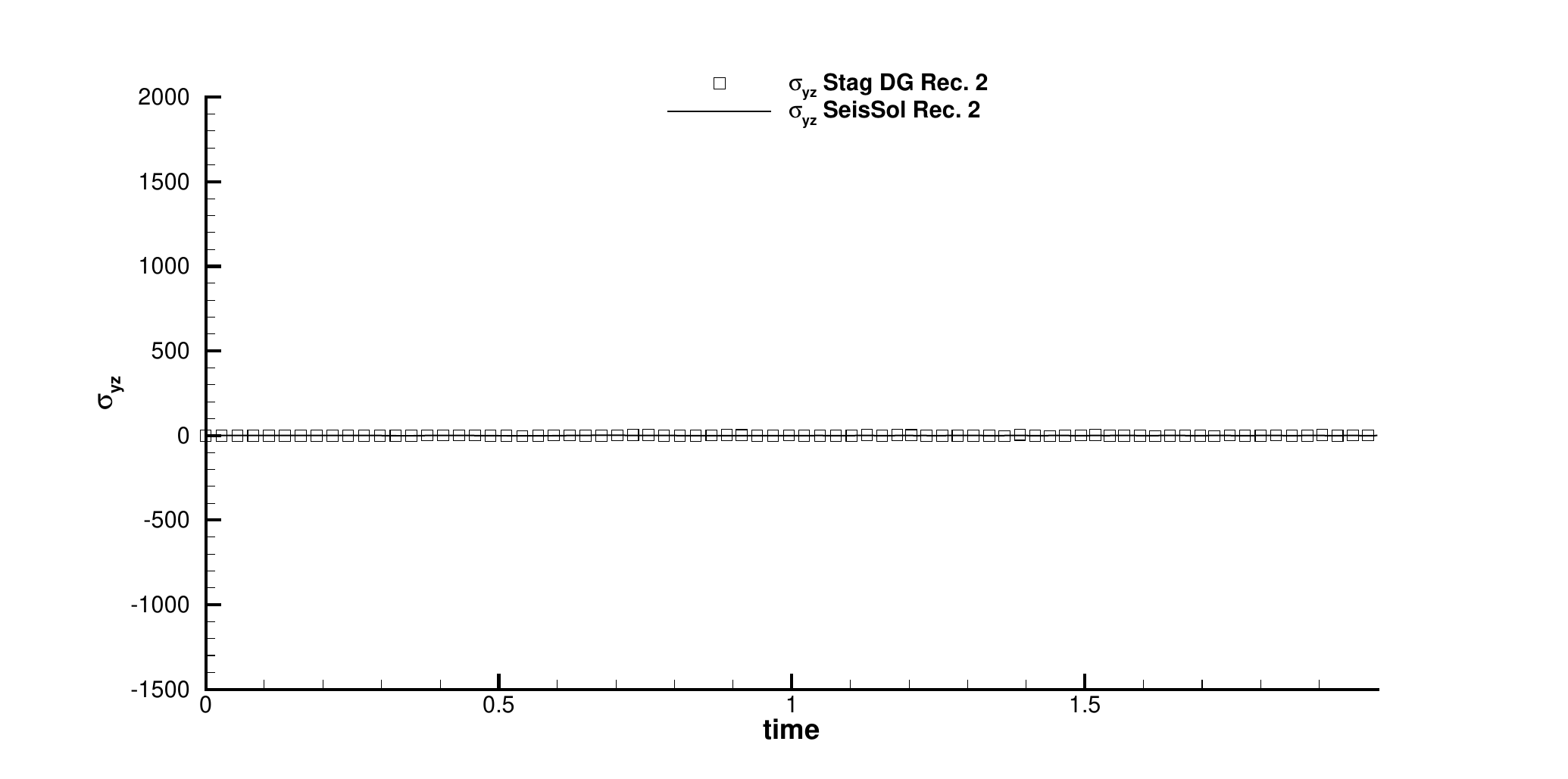} 
\includegraphics[width=0.33\columnwidth]{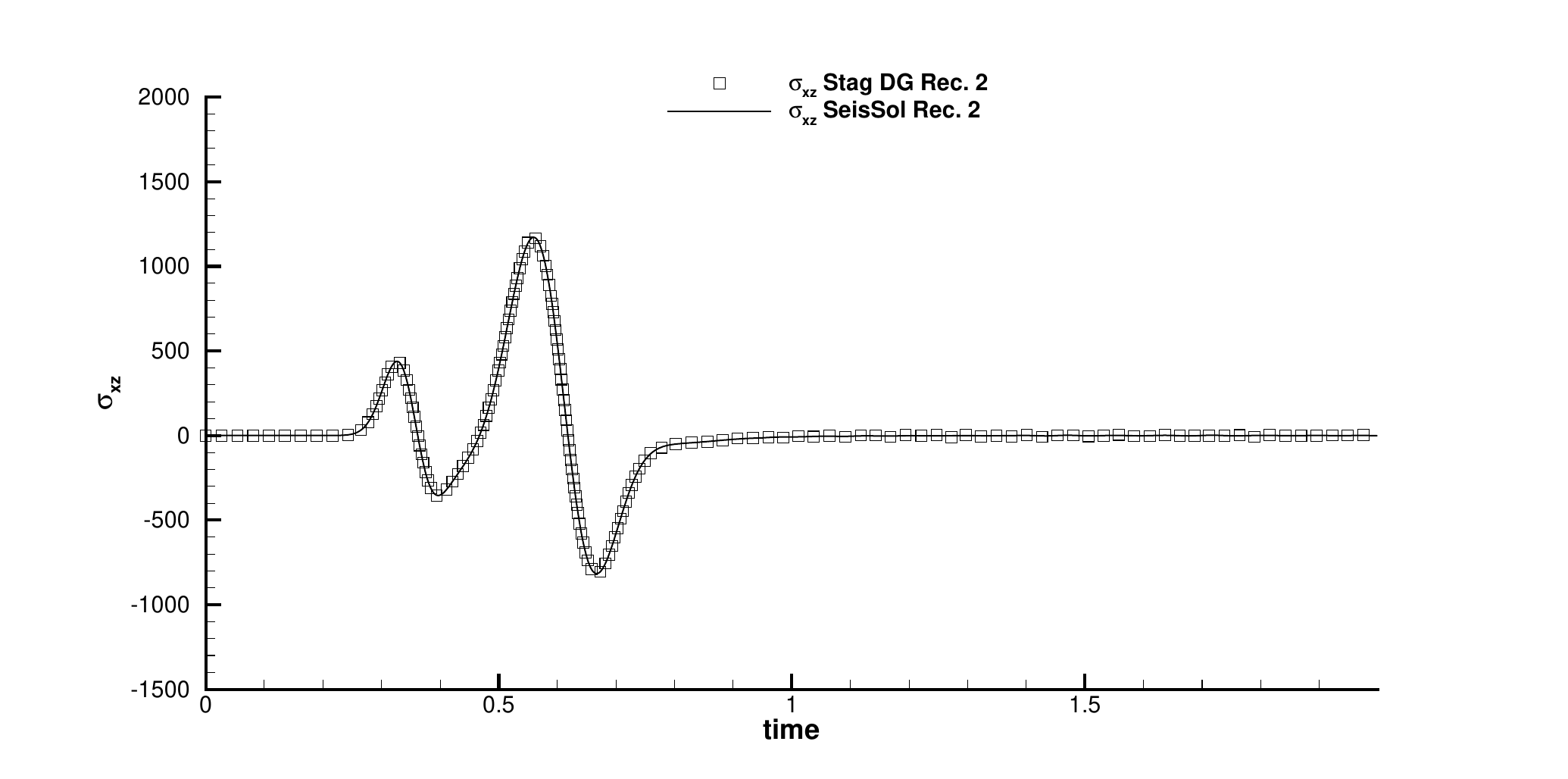}
\caption{Simple 3D wave propagation problem. Comparison of the numerical and reference solution in the second receiver, from top left to bottom right: $u,v,w,\sigma_{xx},\sigma_{yy},\sigma_{zz},\sigma_{xy},\sigma_{yz},\sigma_{xz}$}%
\label{fig.NT3DI_3}%
\end{figure*}

\subsection{Scattering of a planar wave on a sphere} 
We consider here the 3D extension of the test reported in Section \ref{sec_2dpw}, which consists of a planar $p-$wave traveling in the $x-$direction and hitting a sphere. 
The computational domain is $\Omega=[-3,3]^3 - B_{0.25}$, where $B_r$ is the ball of radius $r$. As numerical parameters we set $\Ni=31732$ {elements of average size $h=0.42$}, $(p,p_\gamma)=(4,2)$, $\Delta t=0.01$ 
and $t_{end}=1.0s$. We consider three receivers placed in $\xx_1=(-1,0,0)$, $\xx_2=(0,-1,0)$ and $\xx_3=(0.5,0.5,0.5)$. As a reference solution we use again the explicit ADER-DG 
scheme implemented in the \texttt{SeisSol} code using the same grid and piecewise polynomials of degree $N=4$ in space and time. The time series in the three receivers are reported 
in Figure \ref{fig.NT3DPW_1}. A very good agreement between the explicit ADER-DG scheme and the novel staggered space-time DG method can be observed also in this case. 
\begin{figure*}%
\includegraphics[width=0.33\columnwidth]{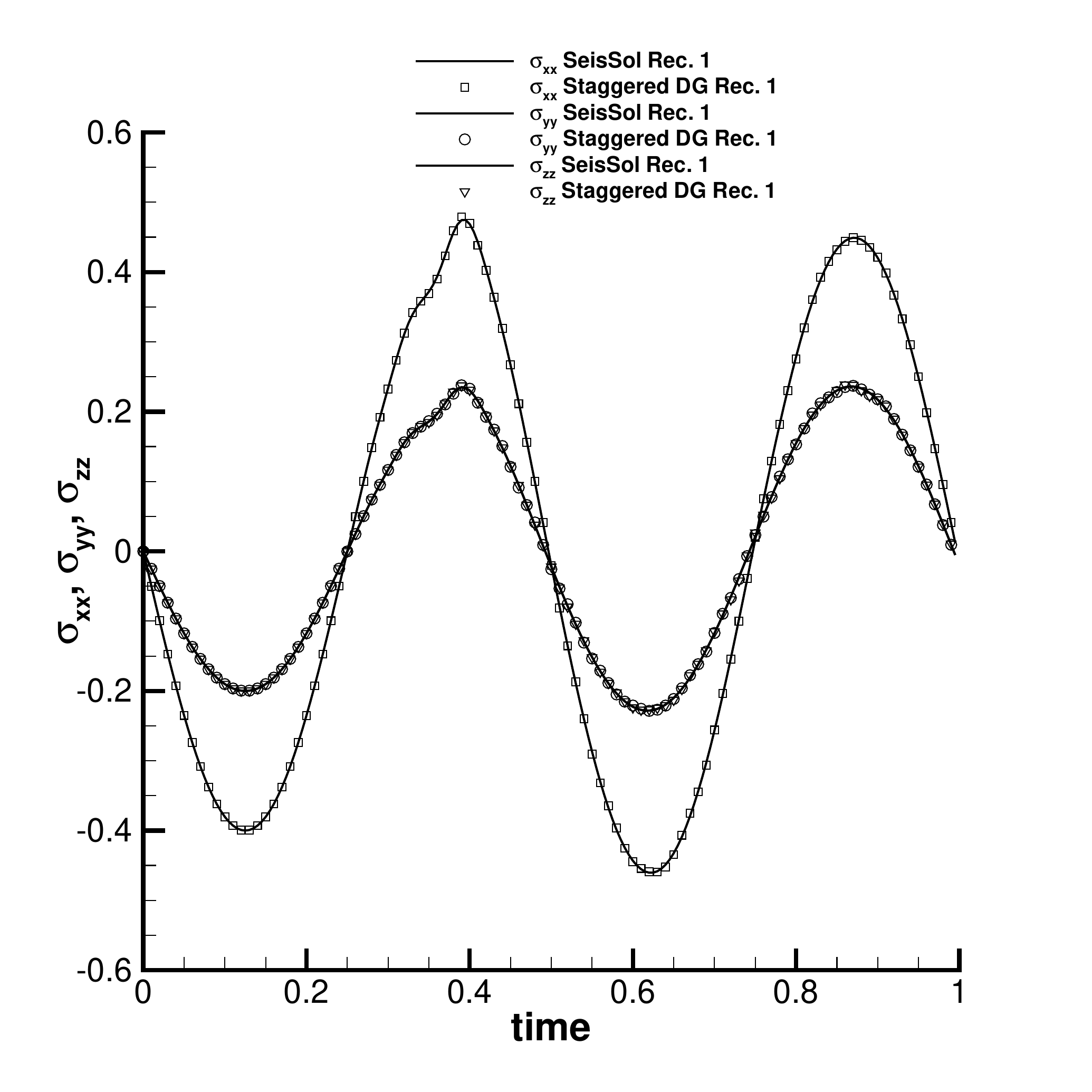} 
\includegraphics[width=0.33\columnwidth]{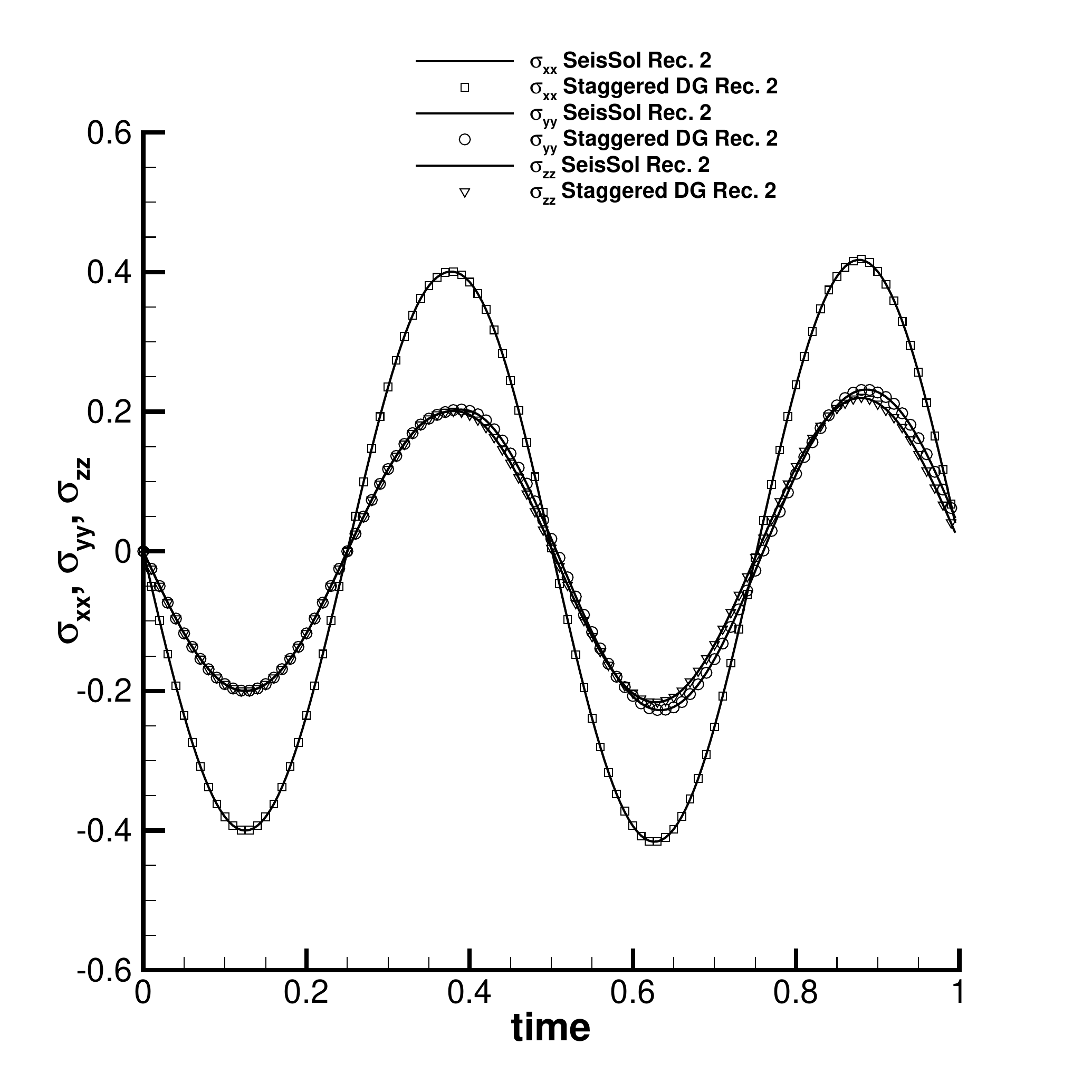} 
\includegraphics[width=0.33\columnwidth]{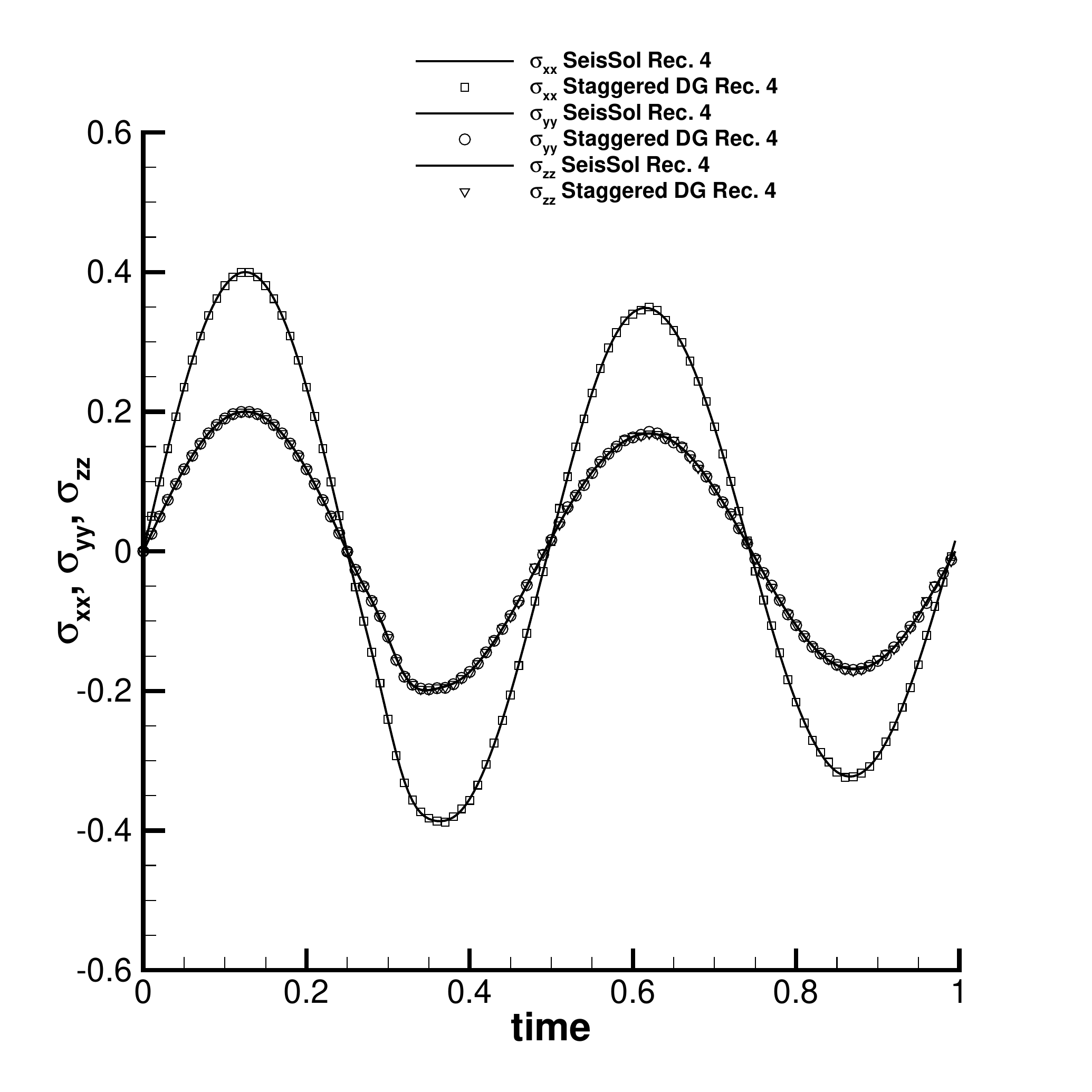} \\
\includegraphics[width=0.33\columnwidth]{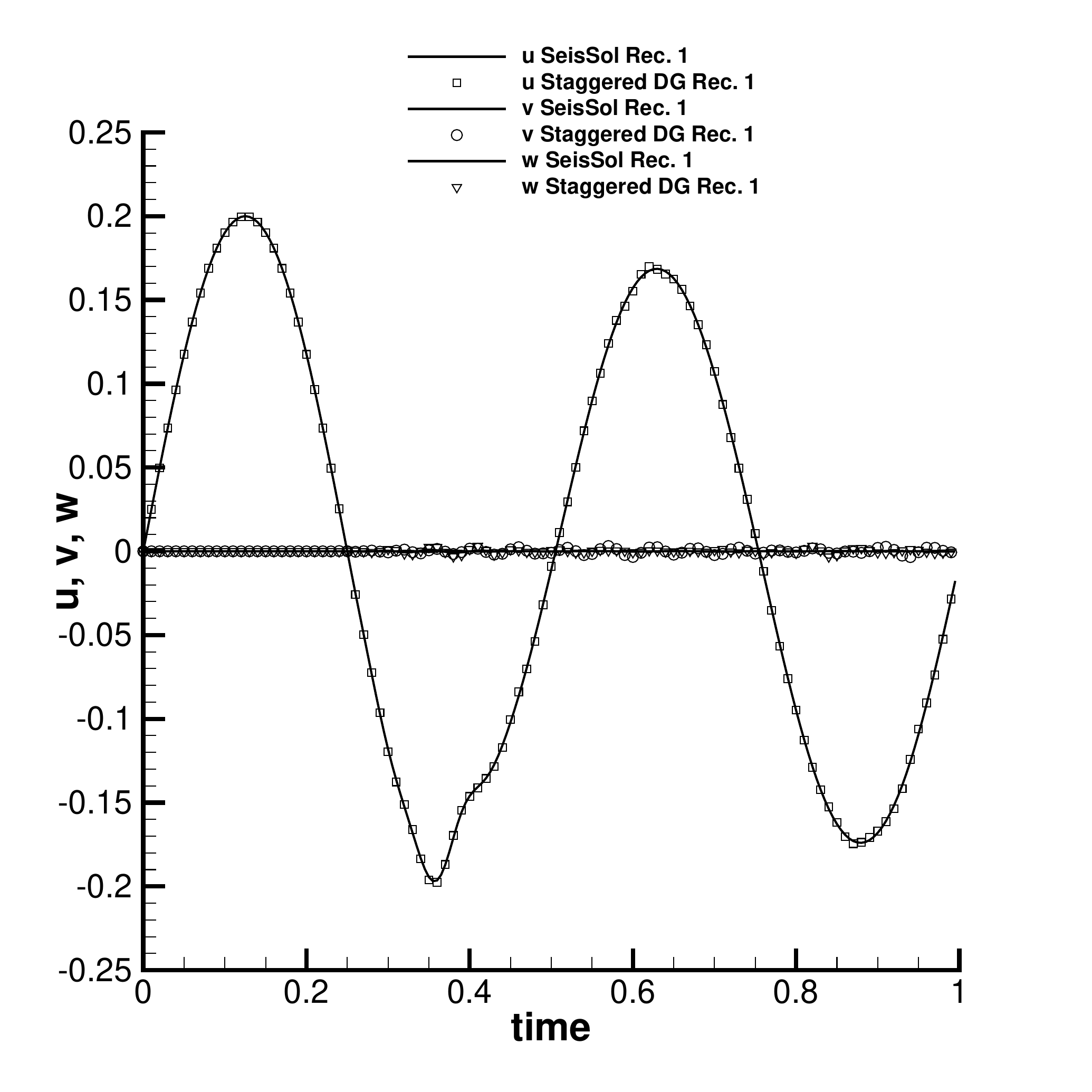} 
\includegraphics[width=0.33\columnwidth]{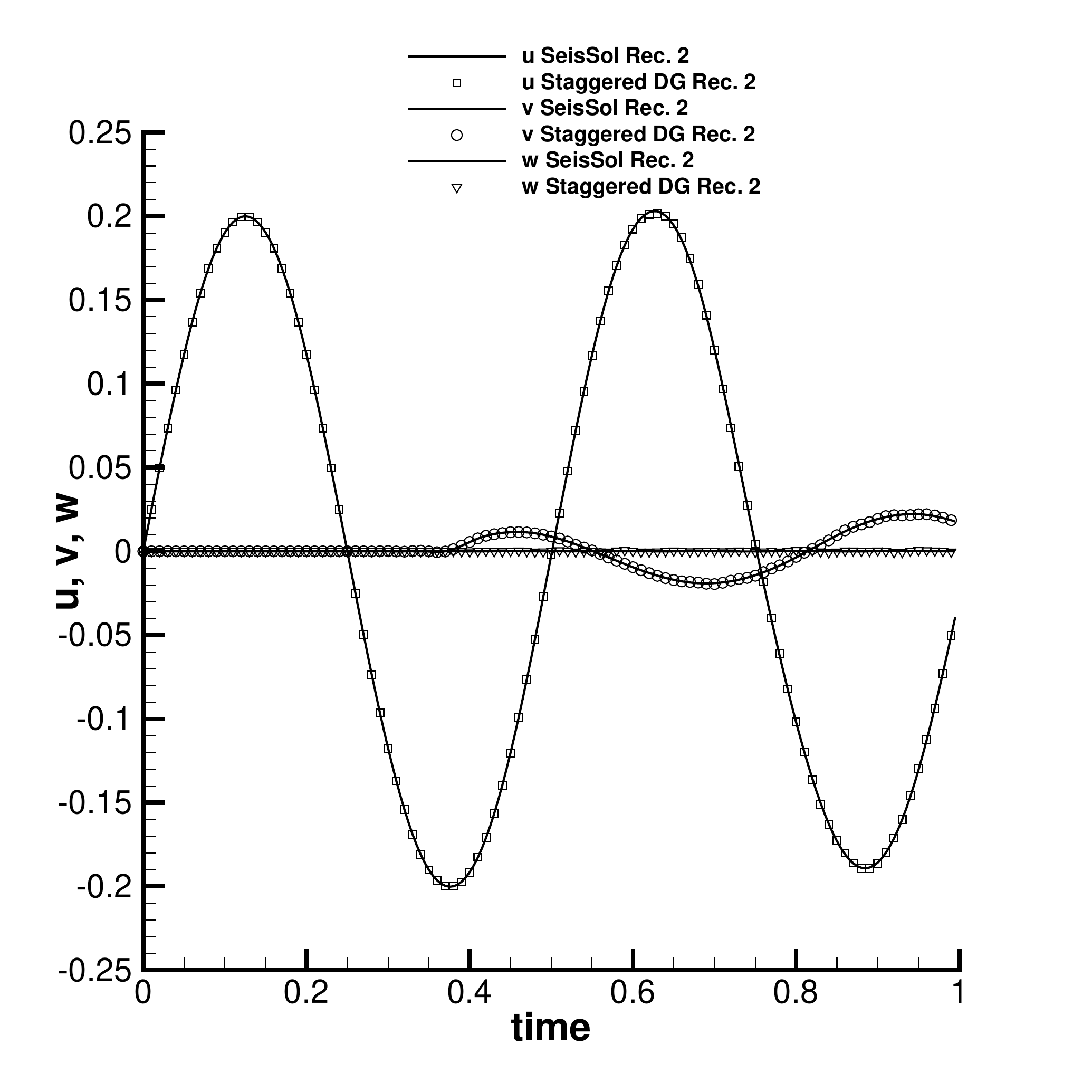} 
\includegraphics[width=0.33\columnwidth]{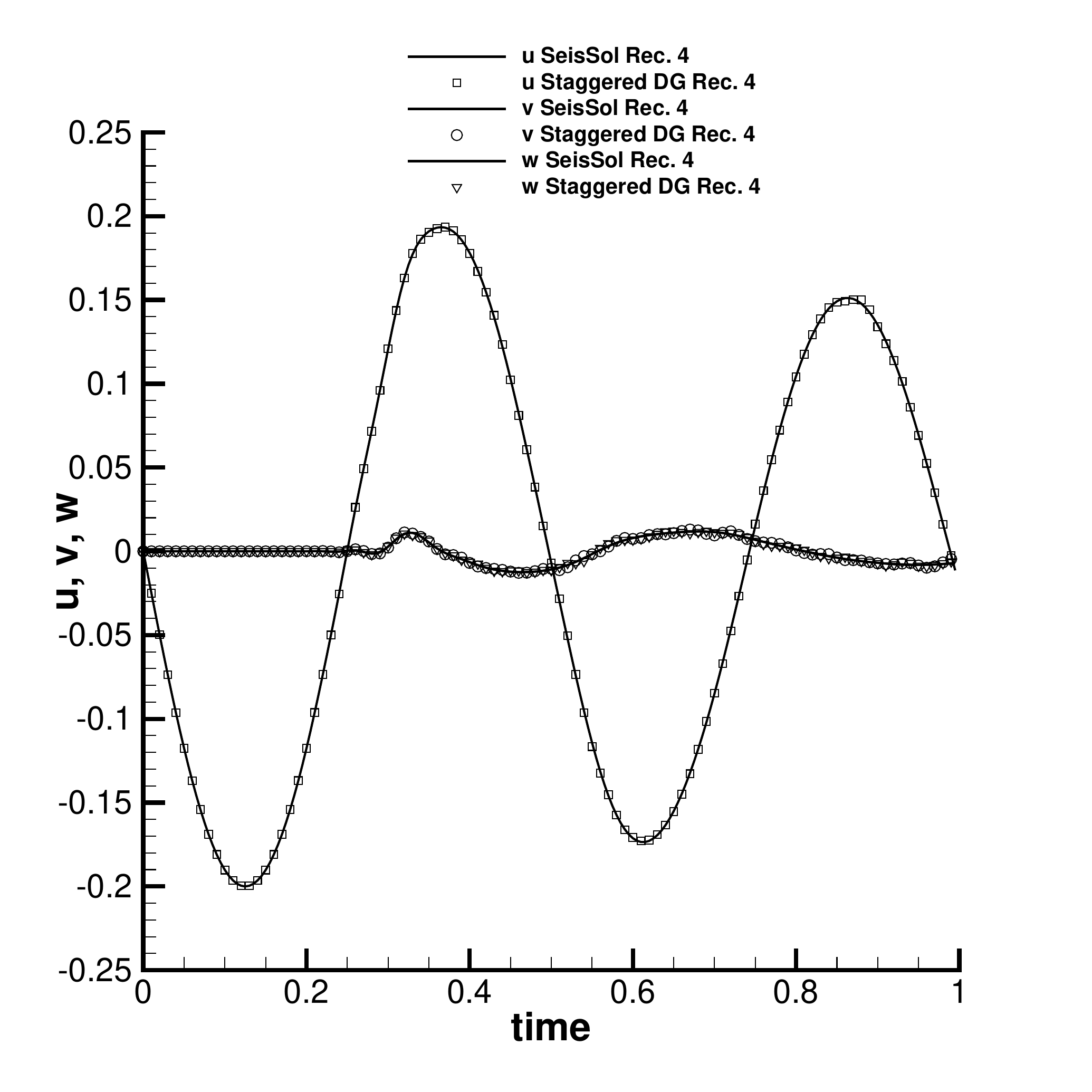} 
\caption{Scattering of a plane wave on a sphere. Comparison of the resulting signal in the three receivers. In the first row we report the time series of the stress tensor components  
$\sigma_{xx}$, $\sigma_{yy}$ and $\sigma_{zz}$ for the receivers $1,2,3$, respectively, from left to right. In the second row the velocity signal is reported for the same receivers. }%
\label{fig.NT3DPW_1}%
\end{figure*}

\subsection{Wave propagation in a complex 3D geometry}
We finally want to test the potential of our new numerical scheme for real applications. For this purpose we generate a tetrahedral mesh based on the real DTM data of the Mont Blanc region\footnote{ The DTM data have been taken from \textcolor{blue}{http://geodati.fmach.it/gfoss\_geodata/libro\_gfoss/.} Our computational domain is centered with respect to the UTM coordinates $(340000.0, 5075000.0)$ }. The horizontal extent of the domain is $30$ km in the $x$ and $y$ directions and ranges from $10$ km below the sea level to the free surface
given by the DTM data. We use a heterogeneous material distribution consisting in two different material layers. The first one is in the region $\{z>-1000\}m$, while the second one covers 
the region $z\leq -1000$m. The parameters for the material are reported in Table \ref{Tab:NT3DCG}. 
\begin{table*}%
\begin{center}
\begin{tabular}{ccccccc}
	& Position 	& $c_p (ms^{-1})$ & $c_s (ms^{-1})$ & $\rho (kg m^{-3})$ & $\lambda (GPa)$ & $\mu (GPa)$	 \\
	\hline
	Medium 1	& $z>-1000$m		  &	4000	&  2000 & 2600 & 20.8 & 10.4 \\
	Medium 2	& $z\leq-1000$m		&	6000	&  3464 & 2700 & 32.4 & 32.4 \\
	\hline
\end{tabular}
\end{center}
\caption{Material parameters for the wave propagation test in a complex 3D geometry.}
\label{Tab:NT3DCG}
\end{table*}
An initial velocity perturbation is placed in $\vec{x}=(0,0,0)$ for the vertical component of the velocity 
\begin{eqnarray}
	w(\xx,0) = a e^{-r^2/R^2}, 
\label{eq:NTCG3D_1b}
\end{eqnarray}
with $a=-10^{-2}$ and $R=300$m. All other variables are set to zero. The computational domain is covered with $N_i=288998$ tetrahedra, whose characteristic size is $500$m close to the 
free surface and $3000$m far from it. For this test 
we use $p=4$ and the Crank-Nicolson time discretization, for which we have the discrete energy preserving property. Furthermore, we set $\Delta t=10^{-3}$s and $t_{end}=4.0$s. 
As reference solution we use again the explicit ADER-DG scheme used in the \texttt{SeisSol} code with the same mesh and a polynomial approximation degree in space and time of $N=4$.  
A comparison of the numerical solution obtained with the new implicit staggered DG scheme and the explicit ADER-DG method at $t=4.0$ is shown in Figure \ref{fig.NTCG1}. 
We consider also the signal captured in four receivers, whose positions are reported in Table \ref{tab:NTCG3D_2} and which are also graphically represented in the right panel of Figure  
\ref{fig.NTCG1}. The resulting time history of the velocity signals for the four receivers is reported in Figure \ref{fig.NTCG2}. A very good agreement between the new staggered DG scheme and 
the reference scheme can be observed also in this case with complex 3D geometry. It is important to note that the use of the energy preserving variant is crucial here to obtain the proper 
wave amplitude with the new staggered implicit DG method. 
Furthermore, we can use the simple matrix-free conjugate gradient method in this case, thanks to the good properties of the matrix for the discrete wave equation for the velocity 
\eqref{eq:velocity.sys}, which is symmetric and positive definite for $p_\gamma=0$. The computation was performed in parallel on the HazelHen supercomputer at the HLRS in Stuttgart, Germany, 
using {144 Xeon E5-2680 Cores}. The parallelization of both schemes was achieved by using the pure MPI standard. It has to be stressed that the MPI parallelization of our new staggered  
space-time DG scheme is straightforward, since we use a matrix-free iterative Krylov subspace method for the solution of the linear system \eqref{eq:velocity.sys}, and the parallelization 
of the matrix-vector product inside the iterative solver can be done exactly in the same way as for an explicit ADER-DG scheme, i.e. based on domain decomposition. 
As in \cite{gij2,gij5} we employ the free Metis software package \cite{metis} for the domain decomposition onto the various MPI ranks. 

\begin{figure*}%
\includegraphics[width=0.49\columnwidth]{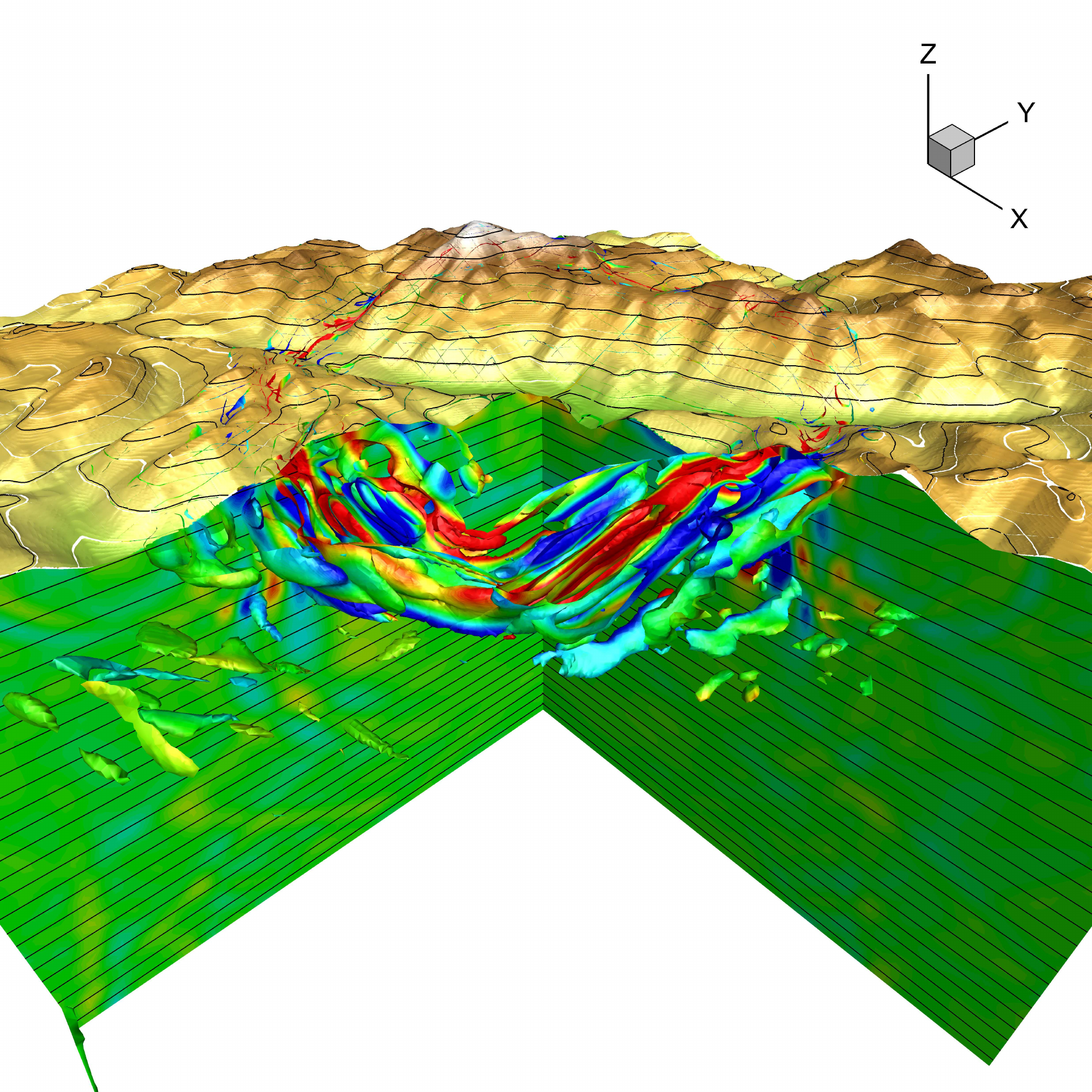} 
\includegraphics[width=0.49\columnwidth]{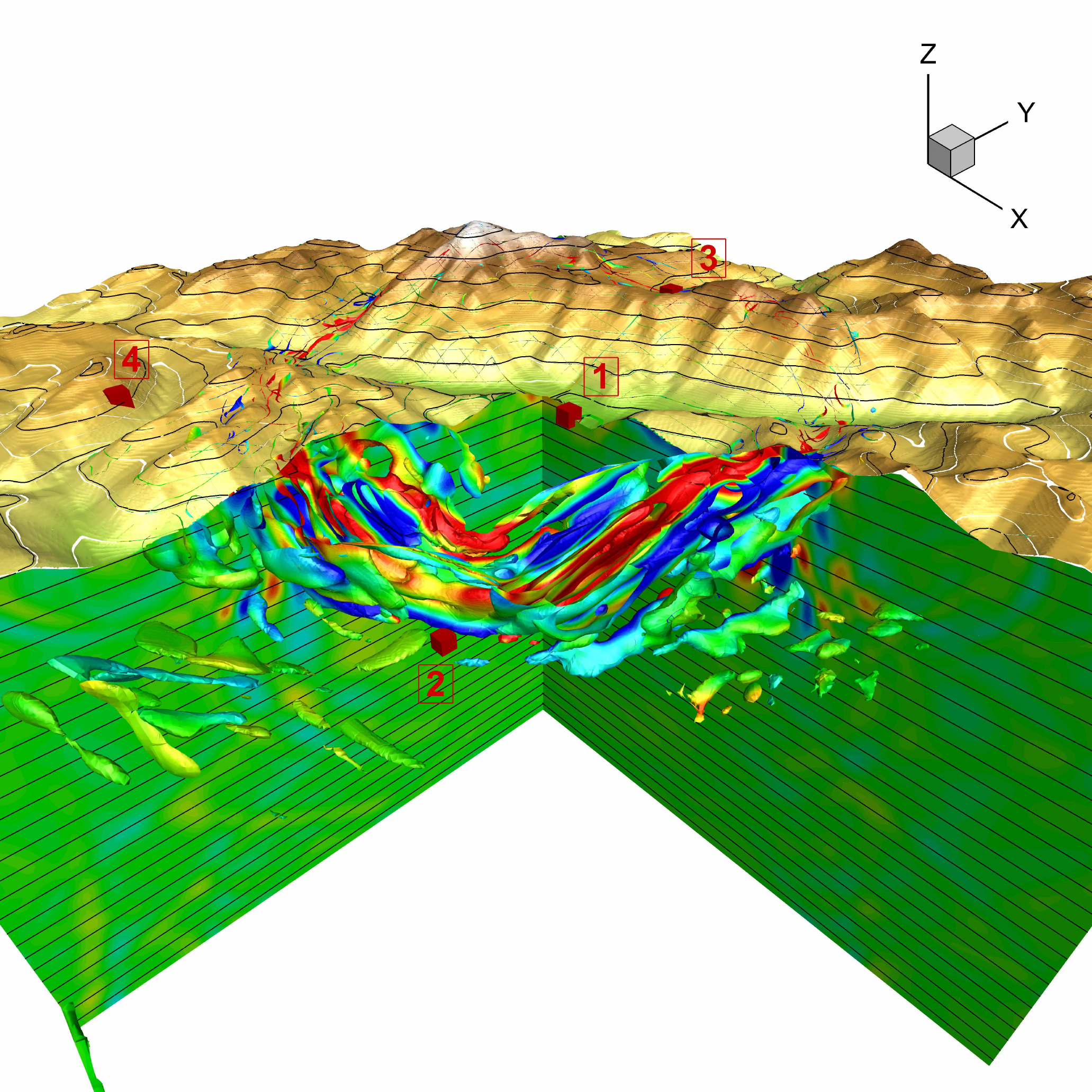}
\caption{Wave propagation test in a complex 3D geometry with real DTM data of the Mont Blanc region. Comparison of the explicit ADER-DG reference solution (left) with the numerical solution obtained with the new implicit staggered DG scheme (right) at time $t=4.0$.  In the right panel the receiver locations are indicated by the red boxes. We show the iso-surfaces $\pm 5\cdot 10^{-5}$ for the velocity components $u$ and $v$ colored by $w$.}
\label{fig.NTCG1}%
\end{figure*}

\begin{figure*}%
\includegraphics[width=0.49\columnwidth]{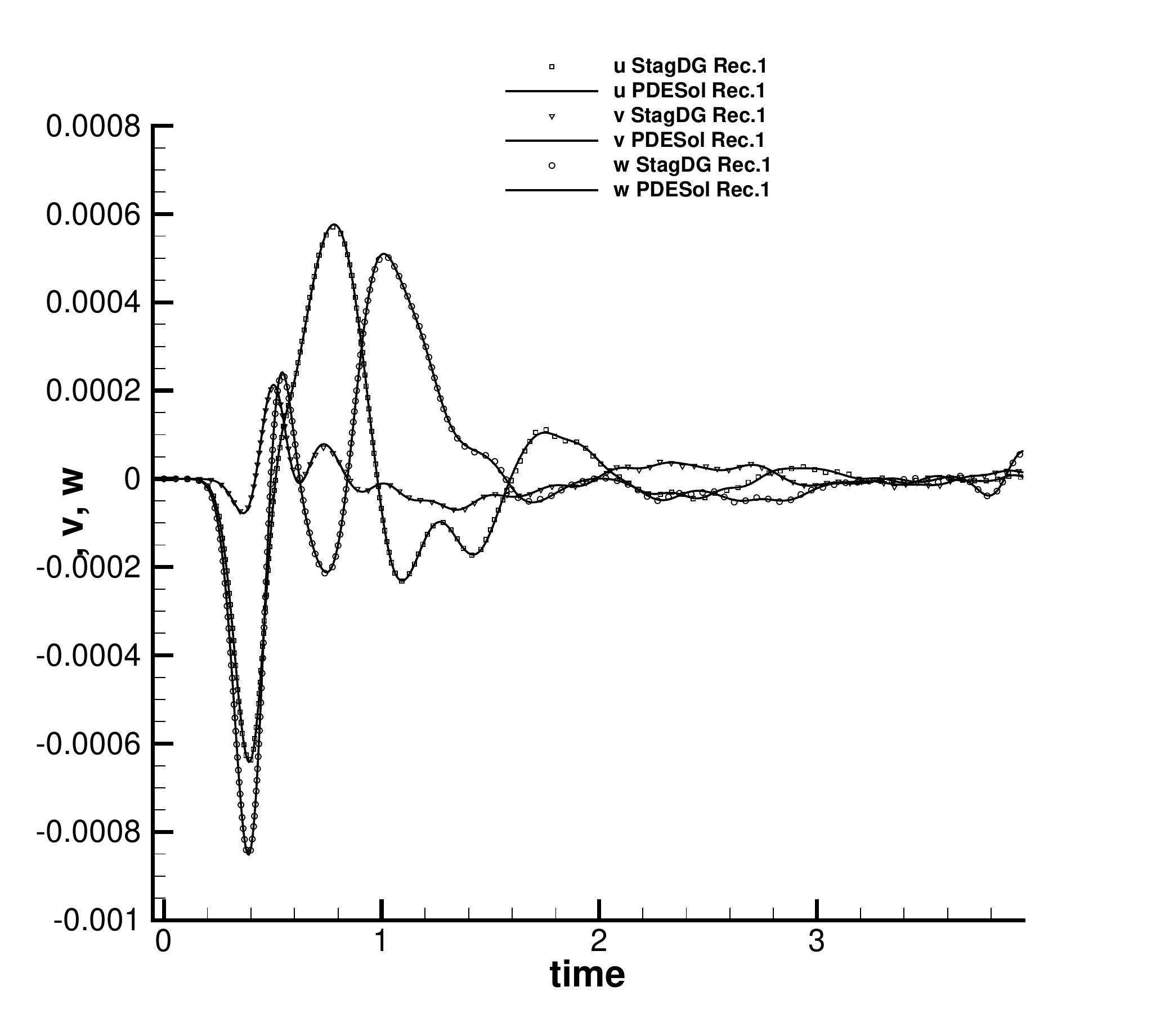} 
\includegraphics[width=0.49\columnwidth]{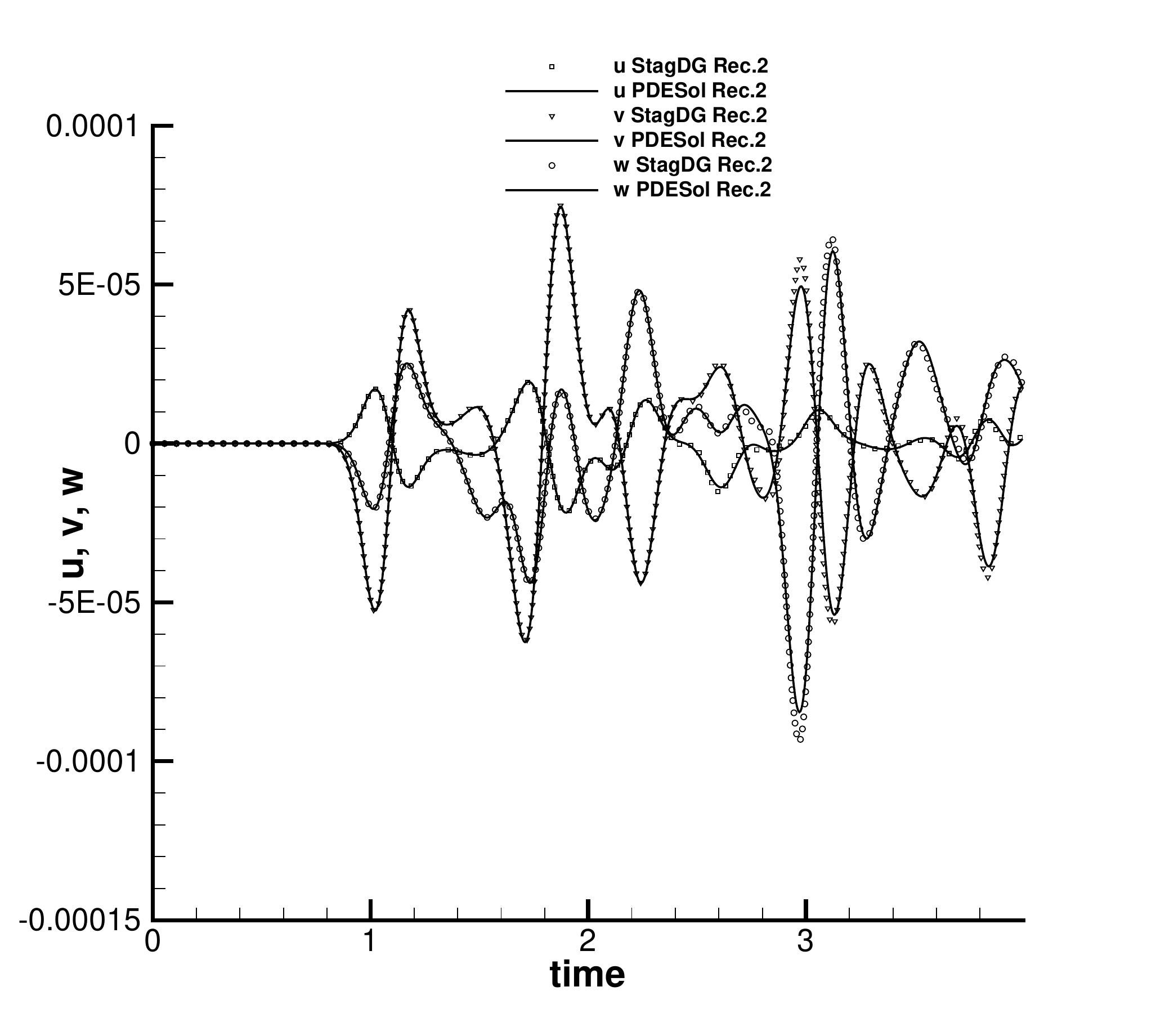} \\
\includegraphics[width=0.49\columnwidth]{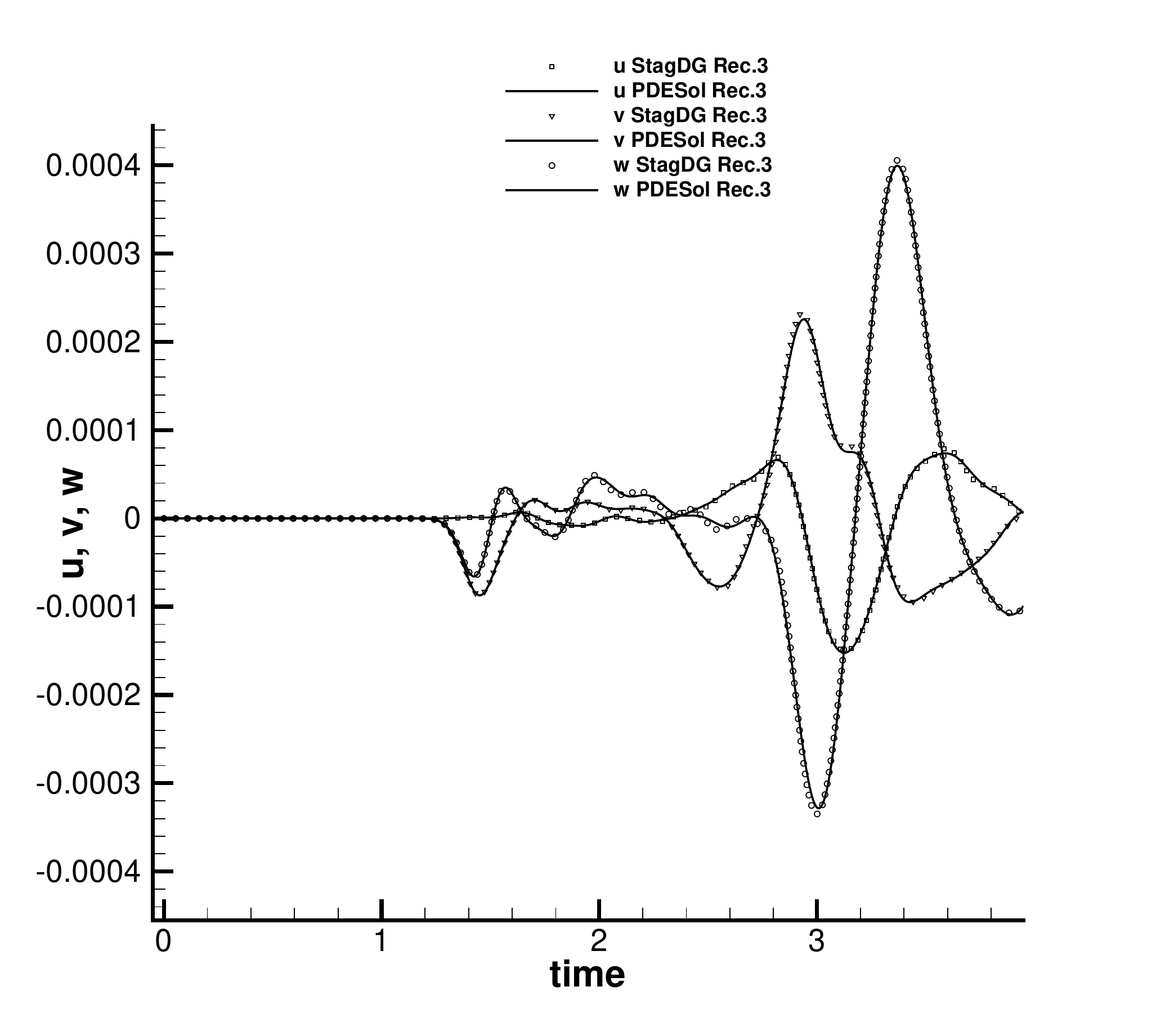} 
\includegraphics[width=0.49\columnwidth]{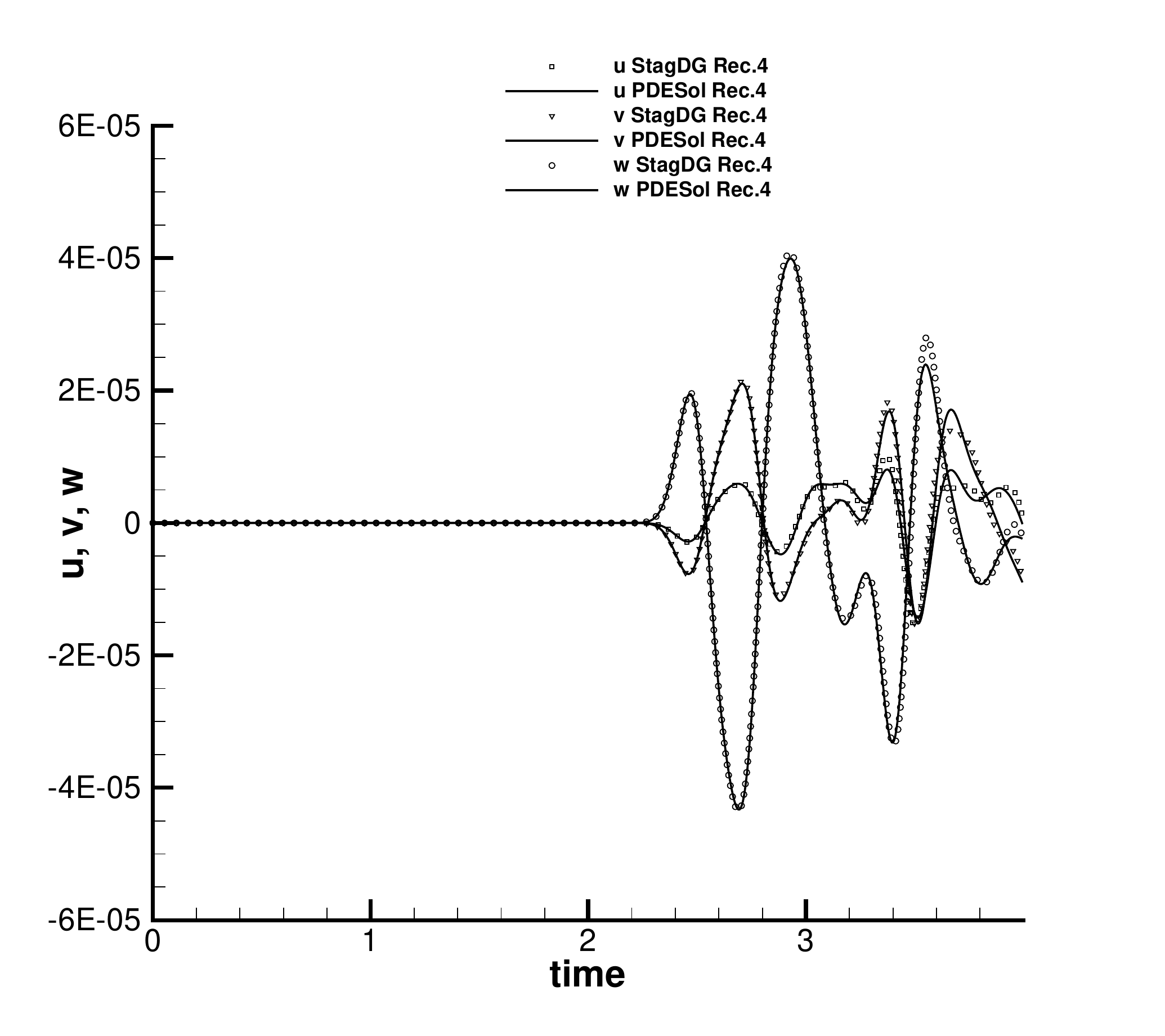}
\caption{Wave propagation test in a complex 3D geometry. Comparison of the numerical solution obtained with the new staggered DG scheme presented in this paper with the reference solution 
for receivers 1-4, respectively, from top left to bottom right. }
\label{fig.NTCG2}%
\end{figure*}

\begin{table}%
\begin{center}
\begin{tabular}{cccc}
	Receiver & $x$ 	& $y$	 & $z$ \\
	\hline
1&1000.000000 & 0.000000 &1397.723250 \\
2&1545.084972 &-4755.282581 &-3000.000000 \\
3&0.000000 &5000.000000 &3231.607925 \\
4&-5000.000000 &-8660.254038 &2494.481430 \\
	\hline
\end{tabular}
\end{center}
\caption{Receiver positions for the wave propagation test in complex 3D geometry.} 
\label{tab:NTCG3D_2}
\end{table}

\section{Conclusions}
\label{sec.concl} 

In this paper we have introduced a novel family of staggered space-time discontinuous Galerkin finite element schemes for the simulation of wave propagation in linear elastic media. The governing PDE  system is written in first order velocity-stress formulation. The key idea is the use of a staggered mesh, where the velocity field is defined on a primary mesh composed of simplex elements, i.e. triangles in 2D and tetrahedra in 3D. The stress tensor is defined on a face-based staggered dual mesh, which consists in quadrilateral elements in the 2D case and non-standard 5-point hexahedra in  
the 3D case. Arbitrary high order of accuracy in space and time are achieved via the use of space-time basis and test functions. The space-time DG method is fully implicit and therefore requires the
solution of a large sparse linear system. The number of unknowns can be easily reduced to the degrees of freedom of the velocity field by substituting the discrete Hooke law into the discrete 
momentum conservation law, which corresponds to the application of the Schur complement. The resulting linear system for the velocity is a discrete wave equation for the velocity and can be 
easily solved with modern iterative Krylov methods. For piecewise constant polynomials in time ($p_\gamma=0$) the final system can be proven to be symmetric and positive definite, hence it can
be efficiently solved with a matrix-free conjugate gradient method. In the general case ($p_\gamma \geq 1$) the system is non-symmetric and is therefore solved with a matrix-free implementation 
of the GMRES algorithm. The use of matrix-free iterative solvers allows a straightforward MPI parallelization of the algorithm on modern supercomputers. 
 
The main advantage of our new staggered space-time DG scheme is its unconditional stability and therefore its robustness with respect to the mesh quality. In particular in 
complex 3D geometries, it is very frequent that computational meshes generated even by modern mesh generation software produce so-called sliver elements, which are elements with a very  
high aspect ratio. Although our new method is unconditionally stable, for computational meshes that contain sliver elements, the linear system becomes ill-conditioned and therefore requires the 
use of a preconditioner. We have implemented two simple preconditioners: the first one is element-local and is based on the exact inverse of each block on the diagonal of the system matrix; 
the second one is more sophisticated and requires the inverse of the local system involving the element and its direct face neighbors. In numerical experiments we have found that the second 
preconditioner is fully sufficient to deal with sliver elements.  For the general case we can prove that the method is energy stable for arbitrary meshes and time step size. For the special case 
of a Crank-Nicolson time discretization, the method is proven to be exactly energy conserving. We have applied the method to a large set of test problems in two and three space dimensions and 
we have also studied the convergence of the scheme via numerical experiments on a smooth problem with exact solution. In all cases the new approach produces excellent results. 
The new numerical method presented in this paper is sufficiently general to allow varying material properties within each element and even anisotropic material behavior could be handled in 
principle. 

Future work will concern the extension of the method to dynamic rupture processes following the ideas outlined in \cite{Puente1,Puente2} for high order ADER-DG schemes. We 
furthermore plan to couple the present staggered space-time DG scheme with explicit ADER-DG methods on adaptive Cartesian meshes (AMR), see \cite{Zanotti2015,AMRDGSI}. Further work will 
also concern the generalization of the present scheme from simple linear elasticity to the equations of fully nonlinear hyperelasticity of Godunov and Romenski 
\cite{GodunovRomenski72} and their recent extension to a unified formulation of continuum mechanics achieved by Peshkov and Romenski and collaborators in \cite{PeshRom2014,GPRmodel,GPRmodelMHD}.   
Last but not least, we plan to extend our scheme to the Maxwell and MHD equations, where staggered meshes are necessary in order to enforce a divergence-free magnetic field. In particular, 
we plan to couple the present approach with some of the novel ideas recently outlined in \cite{BalsaraDGInd,SIMHD,BalsaraED1,BalsaraED2} concerning the use of multi-dimensional Riemann solvers
combined with appropriately staggered meshes for the solution of the Maxwell and MHD equations. 

\section*{Acknowledgements}

This research was funded by the European Union's Horizon 2020 
Research and Innovation Programme under the project \textit{ExaHyPE},
grant no. 671698 (call FETHPC-1-2014). The 3D simulations were performed 
on the HazelHen supercomputer at the HLRS in Stuttgart, Germany and 
on the SuperMUC supercomputer at the LRZ in Garching, Germany. 

\bibliographystyle{elsarticle-num}
\bibliography{SIDG}

\appendix 

\section{Preconditioning}
\label{App_pre}
Here we give some more details on how to implement the two simple preconditioners used to solve the sliver element test problem. 
For the first preconditioner, we only take the diagonal block of system \eqref{eq:velocity.sys}, which for the high order staggered space-time DG method reads 
\begin{equation} 
\mathbf{D}_i = \bar{\M}_i \, \hat{\rho}_i - \sum\limits_{j \in S_i} \D_{i,j} \M_j^{-1}  \hat{\E}_j \cdot \Q_{i,j}. 
\label{eq:pre1}
\end{equation} 
We then exactly invert it for each element and use the block-diagonal matrix $\mathcal{P}_1=\textnormal{diag} \left( \mathbf{D}_1^{-1}, \cdots \mathbf{D}_i^{-1}, \cdots \mathbf{D}_{\Ni}^{-1} \right)$ 
as preconditioner number one. 

The second preconditioner is more sophisticated and locally inverts a small linear system for each element involving also its neighbor elements. Let us locally renumber the elements 
around $\TT_i^{st}$ so that $i \to 0$ and the adjacent face neighbors are numbered as $\p(i,j) = \left\{ 1, 2, 3 \right\}$ in 2D and $\p(i,j) = \left\{ 1, 2, 3, 4 \right\}$ in 3D. 
Let us denote the contributions of the neighbors to the linear system by the off-diagonal blocks $\mathbf{N}_{l,m}$, which represent the contribution of element $m$ on element $l$.  
For the second preconditioner we now assemble a local system around $\TT_i^{st}$ which involves $\TT_i^{st}$ and its direct face neighbors and which constitute a local stencil $\mathcal{S}_i$. 
In the following we will denote by $|\mathcal{S}_i| = d + 2 $ the number of elements contained in the local stencil, where $d$ is the number of space dimensions. Using the renumbering of the 
elements, the auxiliary system matrix $\mathbf{A}_i$ for the  local  system reads in 2D 
\begin{eqnarray}
\mathbf{A}_i=
\left( 
\begin{array}{cccc}
	\mathbf{D}_0      & \mathbf{N}_{0,1} & \mathbf{N}_{0,2}  &  \mathbf{N}_{0,3}  \\
	\mathbf{N}_{1,0}  & \mathbf{D}_1     & \mathbf{N}_{1,2}  &  \mathbf{N}_{1,3}  \\
	\mathbf{N}_{2,0}  & \mathbf{N}_{2,1} & \mathbf{D}_2      &  \mathbf{N}_{2,3}  \\
	\mathbf{N}_{3,0}  & \mathbf{N}_{3,1} & \mathbf{N}_{3,2}  &  \mathbf{D}_3      \\
\end{array}
\right).
\label{eq:APPRE1}
\end{eqnarray}
The matrix $\mathbf{A}_i$ has dimension $N^{st}_{\phi} \cdot d \cdot |\mathcal{S}_i|$, and so it is easily invertible in a preprocessing step using a direct solver. 
We can therefore compute its inverse $\mathbf{A}_i^{-1}$ for each element and store its first row of blocks. With $\mathbf{A}_i^{-1}(e_1,e_2)$ for $e_1,e_2=0 \ldots |\mathcal{S}_i|-1$ 
we will denote the subblock in the inverse of $\mathbf{A}_i$ which corresponds to the interaction of element $e_1$ with $e_2$. The action of the preconditioner matrix $\mathcal{P}_2$ 
is then given for each element $i$ by 
\begin{eqnarray}
	\mathcal{P}_2 \hat{\vec{v}}^{n+1}_i = \sum \limits_{e=0}^{|\mathcal{S}_i|-1 } \mathbf{A}_i^{-1}(0,e) \hat{\vec{v}}^{n+1}_g,
\label{eq:APPRE2}
\end{eqnarray}
where $g=g(e)$ corresponds again to the global element number of the local index $e$. 
The computational cost of this preconditioner is $N^{st}_{\phi} \cdot d \cdot |\mathcal{S}_i| \cdot \Ni$ and so is of the same order of the matrix-vector product 
required in the iterative solver. 

\end{document}

%% file: tec_ugrid.tex
\begin{tikzpicture}[line cap=round,line join=round,>=triangle 45,x=0.6373937677053826cm,y=0.6177884615384613cm]
\clip(2.11,-8.53) rectangle (16.23,3.95);
\fill[color=zzttqq,fill=zzttqq,fill opacity=0.1] (5.19,-3.03) -- (9,3) -- (13,-5) -- cycle;
\fill[color=qqwuqq,fill=qqwuqq,fill opacity=0.05] (9,3) -- (14.49,2.37) -- (13,-5) -- cycle;
\fill[color=qqwuqq,fill=qqwuqq,fill opacity=0.05] (9,3) -- (4.13,2.43) -- (5.19,-3.03) -- cycle;
\fill[color=qqwuqq,fill=qqwuqq,fill opacity=0.05] (5.19,-3.03) -- (4.77,-7.83) -- (13,-5) -- cycle;
\fill[color=zzttqq,fill=zzttqq,fill opacity=0.1] (9.27,-1.79) -- (9,3) -- (5.67,0.43) -- (5.19,-3.03) -- cycle;
\fill[color=zzttqq,fill=zzttqq,fill opacity=0.1] (9.27,-1.79) -- (9,3) -- (12.15,1.37) -- (13,-5) -- cycle;
\fill[color=zzttqq,fill=zzttqq,fill opacity=0.1] (5.19,-3.03) -- (7.47,-5.49) -- (13,-5) -- (9.27,-1.79) -- cycle;
\fill[color=qqttzz,fill=qqttzz,fill opacity=0.1] (13,-5) -- (5.19,-3.03) -- (9.27,-1.79) -- cycle;
\draw [color=zzttqq] (5.19,-3.03)-- (9,3);
\draw [color=zzttqq] (9,3)-- (13,-5);
\draw [color=zzttqq] (13,-5)-- (5.19,-3.03);
\draw [color=qqwuqq] (9,3)-- (14.49,2.37);
\draw [color=qqwuqq] (14.49,2.37)-- (13,-5);
\draw [color=qqwuqq] (13,-5)-- (9,3);
\draw [color=qqwuqq] (9,3)-- (4.13,2.43);
\draw [color=qqwuqq] (4.13,2.43)-- (5.19,-3.03);
\draw [color=qqwuqq] (5.19,-3.03)-- (9,3);
\draw [color=qqwuqq] (5.19,-3.03)-- (4.77,-7.83);
\draw [color=qqwuqq] (4.77,-7.83)-- (13,-5);
\draw [color=qqwuqq] (13,-5)-- (5.19,-3.03);
\draw (9.51,-1.19) node[anchor=north west] {$i$};
\draw (12.41,1.67) node[anchor=north west] {$i_1$};
\draw (5.67,0.67) node[anchor=north west] {$i_2$};
\draw (7.47,-5.25) node[anchor=north west] {$i_3$};
\draw (11.25,-0.47) node[anchor=north west] {$j_1$};
\draw (6.95,1.43) node[anchor=north west] {$j_2$};
\draw (8.93,-3.93) node[anchor=north west] {$j_3$};
\draw (4.39,-2.79) node[anchor=north west] {$n_1$};
\draw (13.07,-4.83) node[anchor=north west] {$n_2$};
\draw (9.05,4.03) node[anchor=north west] {$n_3$};
\draw (7.71,0.35) node[anchor=north west] {$\TT_i$};
\draw [color=zzttqq] (9.27,-1.79)-- (9,3);
\draw [color=zzttqq] (9,3)-- (5.67,0.43);
\draw [color=zzttqq] (5.67,0.43)-- (5.19,-3.03);
\draw [color=zzttqq] (5.19,-3.03)-- (9.27,-1.79);
\draw [color=zzttqq] (9.27,-1.79)-- (9,3);
\draw [color=zzttqq] (9,3)-- (12.15,1.37);
\draw [color=zzttqq] (12.15,1.37)-- (13,-5);
\draw [color=zzttqq] (13,-5)-- (9.27,-1.79);
\draw [color=zzttqq] (5.19,-3.03)-- (7.47,-5.49);
\draw [color=zzttqq] (7.47,-5.49)-- (13,-5);
\draw [color=zzttqq] (13,-5)-- (9.27,-1.79);
\draw [color=zzttqq] (9.27,-1.79)-- (5.19,-3.03);
\draw (10.49,1.15) node[anchor=north west] {$\QQ_{j_1}$};
\draw [color=ffqqqq](10.07,-0.83) node[anchor=north west] {$\Gamma_{j_1}$};
\draw [line width=1.6pt,color=ffqqqq] (9,3)-- (13,-5);
\draw [color=qqttzz] (13,-5)-- (5.19,-3.03);
\draw [color=qqttzz] (5.19,-3.03)-- (9.27,-1.79);
\draw [color=qqttzz] (9.27,-1.79)-- (13,-5);
\draw [color=qqttzz](8.35,-2.17) node[anchor=north west] {$\TT_{i,j_3}$};
\draw (5.19,-3.03)-- (5.67,0.43);
\draw (5.67,0.43)-- (9,3);
\draw (9,3)-- (9.27,-1.79);
\draw (9.27,-1.79)-- (5.19,-3.03);
\draw (9,3)-- (12.15,1.37);
\draw (12.15,1.37)-- (13,-5);
\draw (13,-5)-- (9.27,-1.79);
\draw (13,-5)-- (7.47,-5.49);
\draw (7.47,-5.49)-- (5.19,-3.03);
\begin{scriptsize}
\fill [color=qqqqff] (5.19,-3.03) circle (1.5pt);
\fill [color=qqqqff] (9,3) circle (1.5pt);
\fill [color=qqqqff] (13,-5) circle (1.5pt);
\fill [color=qqqqff] (9.27,-1.79) circle (1.5pt);
\fill [color=qqqqff] (14.49,2.37) circle (1.5pt);
\fill [color=qqqqff] (4.13,2.43) circle (1.5pt);
\fill [color=qqqqff] (4.77,-7.83) circle (1.5pt);
\fill [color=qqqqff] (12.15,1.37) circle (1.5pt);
\fill [color=qqqqff] (5.67,0.43) circle (1.5pt);
\fill [color=qqqqff] (7.47,-5.49) circle (1.5pt);
\end{scriptsize}
\end{tikzpicture}

%% file: tec_3d_omega_i.tex
\begin{tikzpicture}[line cap=round,line join=round,>=triangle 45,x=0.8cm,y=0.8cm]
\clip(3.26,-5.21) rectangle (8.37,1.57);
\fill[color=zzttqq,fill=zzttqq,fill opacity=0.1] (4,-4) -- (8,-4) -- (4,-2) -- cycle;
\fill[color=zzttqq,fill=zzttqq,fill opacity=0.1] (4,-2) -- (6,1) -- (8,-4) -- cycle;
\fill[color=zzttqq,fill=zzttqq,fill opacity=0.1] (4,-2) -- (6,1) -- (4,-4) -- cycle;
\fill[color=zzttqq,fill=zzttqq,fill opacity=0.1] (6,1) -- (4,-4) -- (8,-4) -- cycle;
\fill[color=zzttqq,fill=zzttqq,fill opacity=0.1] (16,-4) -- (20,-4) -- (16,-2) -- cycle;
\fill[color=zzttqq,fill=zzttqq,fill opacity=0.1] (16,-2) -- (18,1) -- (20,-4) -- cycle;
\fill[color=zzttqq,fill=zzttqq,fill opacity=0.1] (18,1) -- (16,-4) -- (20,-4) -- cycle;
\fill[color=zzttqq,fill=zzttqq,fill opacity=0.1] (16,-2) -- (12.56,-0.2) -- (16,-4) -- cycle;
\fill[color=qqqqff,fill=qqqqff,fill opacity=0.1] (12.56,-0.2) -- (15.16,-1.08) -- (16,-4) -- cycle;
\fill[color=qqqqff,fill=qqqqff,fill opacity=0.1] (15.16,-1.08) -- (16,-2) -- (12.56,-0.2) -- cycle;
\fill[color=qqqqff,fill=qqqqff,fill opacity=0.1] (15.16,-1.08) -- (16,-4) -- (16,-2) -- cycle;
\fill[color=qqqqff,fill=qqqqff,fill opacity=0.1] (12.56,-0.2) -- (13.78,-2.9) -- (16,-4) -- cycle;
\fill[color=qqqqff,fill=qqqqff,fill opacity=0.1] (13.78,-2.9) -- (16,-2) -- (12.56,-0.2) -- cycle;
\fill[color=qqqqff,fill=qqqqff,fill opacity=0.1] (13.78,-2.9) -- (16,-4) -- (16,-2) -- cycle;
\fill[color=zzttqq,fill=zzttqq,fill opacity=0.15] (26,-6) -- (28,-3) -- (30,-8) -- cycle;
\draw [color=zzttqq] (4,-4)-- (8,-4);
\draw [color=zzttqq] (4,-2)-- (4,-4);
\draw [color=zzttqq] (4,-2)-- (6,1);
\draw [color=zzttqq] (6,1)-- (8,-4);
\draw [dash pattern=on 1pt off 1pt,color=zzttqq] (8,-4)-- (4,-2);
\draw [color=zzttqq] (4,-2)-- (6,1);
\draw [color=zzttqq] (6,1)-- (4,-4);
\draw [color=zzttqq] (4,-4)-- (4,-2);
\draw [color=zzttqq] (6,1)-- (4,-4);
\draw [color=zzttqq] (4,-4)-- (8,-4);
\draw [color=zzttqq] (8,-4)-- (6,1);
\draw [color=zzttqq] (16,-4)-- (20,-4);
\draw [color=zzttqq] (16,-2)-- (16,-4);
\draw [color=zzttqq] (16,-2)-- (18,1);
\draw [color=zzttqq] (18,1)-- (20,-4);
\draw [dash pattern=on 1pt off 1pt,color=zzttqq] (20,-4)-- (16,-2);
\draw [color=zzttqq] (18,1)-- (16,-4);
\draw [color=zzttqq] (16,-4)-- (20,-4);
\draw [color=zzttqq] (20,-4)-- (18,1);
\draw [color=zzttqq] (16,-2)-- (16.67,-2.33);
\draw [shift={(16,-3)},dotted]  plot[domain=1.11:2.46,variable=\t]({1*4.47*cos(\t r)+0*4.47*sin(\t r)},{0*4.47*cos(\t r)+1*4.47*sin(\t r)});
\draw [color=zzttqq] (12.56,-0.2)-- (16,-4);
\draw [color=zzttqq] (16,-4)-- (16,-2);
\draw [shift={(16,-3)},dotted]  plot[domain=0.66:1.98,variable=\t]({1*2.12*cos(\t r)+0*2.12*sin(\t r)},{0*2.12*cos(\t r)+1*2.12*sin(\t r)});
\draw [color=qqqqff] (12.56,-0.2)-- (15.16,-1.08);
\draw [color=qqqqff] (15.16,-1.08)-- (16,-4);
\draw [color=qqqqff] (16,-4)-- (12.56,-0.2);
\draw [color=qqqqff] (15.16,-1.08)-- (16,-2);
\draw [color=qqqqff] (12.56,-0.2)-- (15.16,-1.08);
\draw [color=qqqqff] (15.16,-1.08)-- (16,-4);
\draw [color=qqqqff] (16,-4)-- (16,-2);
\draw [color=qqqqff] (16,-2)-- (15.16,-1.08);
\draw [color=qqqqff] (12.56,-0.2)-- (13.78,-2.9);
\draw [color=qqqqff] (13.78,-2.9)-- (16,-4);
\draw [color=qqqqff] (16,-4)-- (12.56,-0.2);
\draw [dash pattern=on 1pt off 1pt,color=qqqqff] (16,-2)-- (12.56,-0.2);
\draw [color=qqqqff] (12.56,-0.2)-- (13.78,-2.9);
\draw [color=qqqqff] (13.78,-2.9)-- (16,-4);
\draw [color=qqqqff] (16,-4)-- (16,-2);
\draw [dash pattern=on 1pt off 1pt,color=qqqqff] (16,-2)-- (13.78,-2.9);
\draw [color=zzttqq] (26,-8)-- (30,-8);
\draw [color=zzttqq] (26,-6)-- (26,-8);
\draw [color=zzttqq] (26,-6)-- (28,-3);
\draw [color=zzttqq] (28,-3)-- (30,-8);
\draw [color=zzttqq] (30,-8)-- (26,-6);
\draw [color=zzttqq] (26,-6)-- (28,-3);
\draw [color=zzttqq] (28,-3)-- (26,-8);
\draw [color=zzttqq] (26,-8)-- (26,-6);
\draw [color=zzttqq] (28,-3)-- (26,-8);
\draw [color=zzttqq] (26,-8)-- (30,-8);
\draw [color=zzttqq] (30,-8)-- (28,-3);
\draw [dash pattern=on 1pt off 1pt,color=qqqqff] (30,-4)-- (26,-6);
\draw [color=qqqqff] (30,-8)-- (30,-4);
\draw [color=qqqqff] (28,-3)-- (30,-4);
\draw [color=qqqqff] (30,-4)-- (30,-8);
\draw [color=qqqqff] (30,-8)-- (28,-3);
\draw [color=qqqqff] (26,-6)-- (28,-3);
\draw [color=qqqqff] (28,-3)-- (30,-4);
\draw [dash pattern=on 1pt off 1pt,color=ffqqqq] (28.46,-5.11)-- (27.23,-5.8);
\draw [color=qqqqff] (28.67,-4.68)-- (30,-4);
\draw [color=ffqqqq] (27.23,-5.8)-- (27.86,-5.45);
\draw (5.56,-1.25) node[anchor=north west] {$\TT_i$};
\draw (6.85,0.18) node[anchor=north west] {$\Gamma_{j_1}$};
\draw (6.8,-2.49) node[anchor=north west] {$\Gamma_{j_2}$};
\draw (4.02,-0.04) node[anchor=north west] {$\Gamma_{j_3}$};
\draw (4.63,-4.18) node[anchor=north west] {$\Gamma_{j_4}$};
\draw (14.4,0.1) node[anchor=north west] {$\QQ_j$};
\draw (16.02,-0.08) node[anchor=north west] {$i$};
\draw (17.82,-1.74) node[anchor=north west] {$\TT_i$};
\draw (13.21,-3.03) node[anchor=north west] {$\mathbb{\wp}(i,j)$};
\draw [->,dash pattern=on 1pt off 1pt] (28.02,-5.62) -- (28.66,-5.29);
\draw (27.13,-5.78) node[anchor=north west] {$\ell(j)$};
\draw (28.99,-4.46) node[anchor=north west] {$r(j)$};
\draw (28.36,-5.41) node[anchor=north west] {$\vec{n_j}$};
\draw (30.3,-7.62) node[anchor=north west] {$\Gamma_j$};
\begin{scriptsize}
\fill [color=uuuuuu] (4,-4) circle (1.0pt);
\fill [color=uuuuuu] (8,-4) circle (1.0pt);
\fill [color=uuuuuu] (4,-2) circle (1.0pt);
\fill [color=uuuuuu] (6,1) circle (1.0pt);
\fill [color=uuuuuu] (16,-4) circle (1.0pt);
\fill [color=uuuuuu] (20,-4) circle (1.0pt);
\fill [color=uuuuuu] (16,-2) circle (1.0pt);
\fill [color=uuuuuu] (18,1) circle (1.0pt);
\fill [color=uuuuuu] (12.56,-0.2) circle (1.0pt);
\draw [color=ffqqqq] (17.68,-1.7) circle (1.5pt);
\draw [color=ffqqqq] (15.16,-1.08) circle (1.5pt);
\draw [color=ffqqqq] (13.78,-2.9) circle (1.5pt);
\fill [color=uuuuuu] (26,-8) circle (1.0pt);
\fill [color=uuuuuu] (30,-8) circle (1.0pt);
\fill [color=uuuuuu] (26,-6) circle (1.0pt);
\fill [color=uuuuuu] (28,-3) circle (1.0pt);
\fill [color=qqqqff] (30,-4) circle (1.5pt);
\draw[color=qqqqff] (0.19,5.94) node {$L$};
\draw [color=ffqqqq] (28.46,-5.11) circle (1.5pt);
\draw [color=ffqqqq] (27.23,-5.8) circle (1.5pt);
\end{scriptsize}
\end{tikzpicture}

%% file: tec_3d_omega_j.tex
\begin{tikzpicture}[line cap=round,line join=round,>=triangle 45,x=0.8cm,y=0.8cm]
\clip(11.74,-4.85) rectangle (20.91,2.12);
\fill[color=zzttqq,fill=zzttqq,fill opacity=0.1] (4,-4) -- (8,-4) -- (4,-2) -- cycle;
\fill[color=zzttqq,fill=zzttqq,fill opacity=0.1] (4,-2) -- (6,1) -- (8,-4) -- cycle;
\fill[color=zzttqq,fill=zzttqq,fill opacity=0.1] (4,-2) -- (6,1) -- (4,-4) -- cycle;
\fill[color=zzttqq,fill=zzttqq,fill opacity=0.1] (6,1) -- (4,-4) -- (8,-4) -- cycle;
\fill[color=zzttqq,fill=zzttqq,fill opacity=0.1] (16,-4) -- (20,-4) -- (16,-2) -- cycle;
\fill[color=zzttqq,fill=zzttqq,fill opacity=0.1] (16,-2) -- (18,1) -- (20,-4) -- cycle;
\fill[color=zzttqq,fill=zzttqq,fill opacity=0.1] (18,1) -- (16,-4) -- (20,-4) -- cycle;
\fill[color=zzttqq,fill=zzttqq,fill opacity=0.1] (16,-2) -- (12.56,-0.2) -- (16,-4) -- cycle;
\fill[color=qqqqff,fill=qqqqff,fill opacity=0.1] (12.56,-0.2) -- (15.16,-1.08) -- (16,-4) -- cycle;
\fill[color=qqqqff,fill=qqqqff,fill opacity=0.1] (15.16,-1.08) -- (16,-2) -- (12.56,-0.2) -- cycle;
\fill[color=qqqqff,fill=qqqqff,fill opacity=0.1] (15.16,-1.08) -- (16,-4) -- (16,-2) -- cycle;
\fill[color=qqqqff,fill=qqqqff,fill opacity=0.1] (12.56,-0.2) -- (13.78,-2.9) -- (16,-4) -- cycle;
\fill[color=qqqqff,fill=qqqqff,fill opacity=0.1] (13.78,-2.9) -- (16,-2) -- (12.56,-0.2) -- cycle;
\fill[color=qqqqff,fill=qqqqff,fill opacity=0.1] (13.78,-2.9) -- (16,-4) -- (16,-2) -- cycle;
\fill[color=zzttqq,fill=zzttqq,fill opacity=0.15] (26,-6) -- (28,-3) -- (30,-8) -- cycle;
\draw [color=zzttqq] (4,-4)-- (8,-4);
\draw [color=zzttqq] (4,-2)-- (4,-4);
\draw [color=zzttqq] (4,-2)-- (6,1);
\draw [color=zzttqq] (6,1)-- (8,-4);
\draw [dash pattern=on 1pt off 1pt,color=zzttqq] (8,-4)-- (4,-2);
\draw [color=zzttqq] (4,-2)-- (6,1);
\draw [color=zzttqq] (6,1)-- (4,-4);
\draw [color=zzttqq] (4,-4)-- (4,-2);
\draw [color=zzttqq] (6,1)-- (4,-4);
\draw [color=zzttqq] (4,-4)-- (8,-4);
\draw [color=zzttqq] (8,-4)-- (6,1);
\draw [color=zzttqq] (16,-4)-- (20,-4);
\draw [color=zzttqq] (16,-2)-- (16,-4);
\draw [color=zzttqq] (16,-2)-- (18,1);
\draw [color=zzttqq] (18,1)-- (20,-4);
\draw [dash pattern=on 1pt off 1pt,color=zzttqq] (20,-4)-- (16,-2);
\draw [color=zzttqq] (18,1)-- (16,-4);
\draw [color=zzttqq] (16,-4)-- (20,-4);
\draw [color=zzttqq] (20,-4)-- (18,1);
\draw [color=zzttqq] (16,-2)-- (16.67,-2.33);
\draw [shift={(16,-3)},dotted]  plot[domain=1.11:2.46,variable=\t]({1*4.47*cos(\t r)+0*4.47*sin(\t r)},{0*4.47*cos(\t r)+1*4.47*sin(\t r)});
\draw [color=zzttqq] (12.56,-0.2)-- (16,-4);
\draw [color=zzttqq] (16,-4)-- (16,-2);
\draw [shift={(16,-3)},dotted]  plot[domain=0.66:1.98,variable=\t]({1*2.12*cos(\t r)+0*2.12*sin(\t r)},{0*2.12*cos(\t r)+1*2.12*sin(\t r)});
\draw [color=qqqqff] (12.56,-0.2)-- (15.16,-1.08);
\draw [color=qqqqff] (15.16,-1.08)-- (16,-4);
\draw [color=qqqqff] (16,-4)-- (12.56,-0.2);
\draw [color=qqqqff] (15.16,-1.08)-- (16,-2);
\draw [color=qqqqff] (12.56,-0.2)-- (15.16,-1.08);
\draw [color=qqqqff] (15.16,-1.08)-- (16,-4);
\draw [color=qqqqff] (16,-4)-- (16,-2);
\draw [color=qqqqff] (16,-2)-- (15.16,-1.08);
\draw [color=qqqqff] (12.56,-0.2)-- (13.78,-2.9);
\draw [color=qqqqff] (13.78,-2.9)-- (16,-4);
\draw [color=qqqqff] (16,-4)-- (12.56,-0.2);
\draw [dash pattern=on 1pt off 1pt,color=qqqqff] (16,-2)-- (12.56,-0.2);
\draw [color=qqqqff] (12.56,-0.2)-- (13.78,-2.9);
\draw [color=qqqqff] (13.78,-2.9)-- (16,-4);
\draw [color=qqqqff] (16,-4)-- (16,-2);
\draw [dash pattern=on 1pt off 1pt,color=qqqqff] (16,-2)-- (13.78,-2.9);
\draw [color=zzttqq] (26,-8)-- (30,-8);
\draw [color=zzttqq] (26,-6)-- (26,-8);
\draw [color=zzttqq] (26,-6)-- (28,-3);
\draw [color=zzttqq] (28,-3)-- (30,-8);
\draw [color=zzttqq] (30,-8)-- (26,-6);
\draw [color=zzttqq] (26,-6)-- (28,-3);
\draw [color=zzttqq] (28,-3)-- (26,-8);
\draw [color=zzttqq] (26,-8)-- (26,-6);
\draw [color=zzttqq] (28,-3)-- (26,-8);
\draw [color=zzttqq] (26,-8)-- (30,-8);
\draw [color=zzttqq] (30,-8)-- (28,-3);
\draw [dash pattern=on 1pt off 1pt,color=qqqqff] (30,-4)-- (26,-6);
\draw [color=qqqqff] (30,-8)-- (30,-4);
\draw [color=qqqqff] (28,-3)-- (30,-4);
\draw [color=qqqqff] (30,-4)-- (30,-8);
\draw [color=qqqqff] (30,-8)-- (28,-3);
\draw [color=qqqqff] (26,-6)-- (28,-3);
\draw [color=qqqqff] (28,-3)-- (30,-4);
\draw [dash pattern=on 1pt off 1pt,color=ffqqqq] (28.46,-5.11)-- (27.23,-5.8);
\draw [color=qqqqff] (28.67,-4.68)-- (30,-4);
\draw [color=ffqqqq] (27.23,-5.8)-- (27.86,-5.45);
\draw (5.56,-1.25) node[anchor=north west] {$\TT_i$};
\draw (6.85,0.18) node[anchor=north west] {$\Gamma_{j_1}$};
\draw (6.99,-2.49) node[anchor=north west] {$\Gamma_{j_2}$};
\draw (4.02,-0.04) node[anchor=north west] {$\Gamma_{j_3}$};
\draw (4.63,-4.18) node[anchor=north west] {$\Gamma_{j_4}$};
\draw (14.4,0.1) node[anchor=north west] {$\QQ_{j_3}$};
\draw (16.02,-0.08) node[anchor=north west] {$i$};
\draw (17.82,-1.74) node[anchor=north west] {$\TT_i$};
\draw (13.21,-3.03) node[anchor=north west] {$\mathbb{\wp}(i,j_3)$};
\draw [->,dash pattern=on 1pt off 1pt] (28.02,-5.62) -- (28.66,-5.29);
\draw (27.13,-5.78) node[anchor=north west] {$\ell(j)$};
\draw (28.99,-4.46) node[anchor=north west] {$r(j)$};
\draw (28.36,-5.41) node[anchor=north west] {$\vec{n_j}$};
\draw (30.3,-7.62) node[anchor=north west] {$\Gamma_j$};
\begin{scriptsize}
\fill [color=uuuuuu] (4,-4) circle (1.0pt);
\fill [color=uuuuuu] (8,-4) circle (1.0pt);
\fill [color=uuuuuu] (4,-2) circle (1.0pt);
\fill [color=uuuuuu] (6,1) circle (1.0pt);
\fill [color=uuuuuu] (16,-4) circle (1.0pt);
\fill [color=uuuuuu] (20,-4) circle (1.0pt);
\fill [color=uuuuuu] (16,-2) circle (1.0pt);
\fill [color=uuuuuu] (18,1) circle (1.0pt);
\fill [color=uuuuuu] (12.56,-0.2) circle (1.0pt);
\draw [color=ffqqqq] (17.68,-1.7) circle (1.5pt);
\draw [color=ffqqqq] (15.16,-1.08) circle (1.5pt);
\draw [color=ffqqqq] (13.78,-2.9) circle (1.5pt);
\fill [color=uuuuuu] (26,-8) circle (1.0pt);
\fill [color=uuuuuu] (30,-8) circle (1.0pt);
\fill [color=uuuuuu] (26,-6) circle (1.0pt);
\fill [color=uuuuuu] (28,-3) circle (1.0pt);
\fill [color=qqqqff] (30,-4) circle (1.5pt);
\draw[color=qqqqff] (0.19,5.94) node {$L$};
\draw [color=ffqqqq] (28.46,-5.11) circle (1.5pt);
\draw [color=ffqqqq] (27.23,-5.8) circle (1.5pt);
\end{scriptsize}
\end{tikzpicture}